\documentclass[12pt,twoside,a4paper]{amsart}
\usepackage{amssymb,amscd,amsxtra,calc,mathrsfs}
\usepackage{amsmath}
\usepackage{xypic}

\title[K-stability and the volume functions]{On K-stability and the volume functions 
of $\mathbb{Q}$-Fano varieties}
\author{Kento Fujita} 
\date{\today}
\address{Department of Mathematics, Faculty of Science, 
Kyoto University, Kyoto 606-8502, Japan}
\email{fujita@math.kyoto-u.ac.jp}

\newcommand{\pr}{\mathbb{P}}

\newcommand{\Z}{\mathbb{Z}}
\newcommand{\Q}{\mathbb{Q}}
\newcommand{\R}{\mathbb{R}}
\newcommand{\C}{\mathbb{C}}
\newcommand{\F}{\mathbb{F}}
\newcommand{\B}{\mathbb{B}}
\newcommand{\A}{\mathbb{A}}
\newcommand{\G}{\mathbb{G}}

\newcommand{\ND}{\operatorname{N}^1}
\newcommand{\NE}{\operatorname{NE}}
\newcommand{\Nef}{\operatorname{Nef}}

\newcommand{\Eff}{\operatorname{Eff}}
\newcommand{\BIG}{\operatorname{Big}}

\newcommand{\Exc}{\operatorname{Exc}}

\newcommand{\Spec}{\operatorname{Spec}}
\newcommand{\Pic}{\operatorname{Pic}}

\newcommand{\DIV}{\operatorname{div}}
\newcommand{\Hom}{\operatorname{Hom}}

\newcommand{\Aut}{\operatorname{Aut}}
\newcommand{\Proj}{\operatorname{Proj}}

\newcommand{\cont}{\operatorname{cont}}

\newcommand{\DF}{\operatorname{DF}}

\newcommand{\vol}{\operatorname{vol}}
\newcommand{\sI}{\mathcal{I}}
\newcommand{\sC}{\mathcal{C}}
\newcommand{\sO}{\mathcal{O}}
\newcommand{\sN}{\mathcal{N}}
\newcommand{\sE}{\mathcal{E}}

\newcommand{\sL}{\mathcal{L}}

\newcommand{\sB}{\mathcal{B}}
\newcommand{\sA}{\mathcal{A}}

\newtheorem{thm}{Theorem}[section]
\newtheorem{lemma}[thm]{Lemma}
\newtheorem{proposition}[thm]{Proposition}
\newtheorem{corollary}[thm]{Corollary}
\newtheorem{claim}[thm]{Claim}

\theoremstyle{definition}
\newtheorem{definition}[thm]{Definition}
\newtheorem{remark}[thm]{Remark}

\newtheorem{example}[thm]{Example}
\newtheorem*{ack}{Acknowledgments}

\begin{document}

\maketitle 

\begin{abstract}
We introduce a new effective stability named ``divisorial stability" for Fano 
manifolds which is weaker than K-stability and is stronger than slope stability 
along divisors. We show that we can test divisorial stability 
via the volume function. As a corollary, we prove that 
the first coordinate of the barycenter of the 
Okounkov body of the anticanonical divisor is not bigger than one 
for any K\"ahler-Einstein Fano manifold. 
In particular, for toric Fano manifolds, the existence of K\"ahler-Einstein metrics 
is equivalent to divisorial semistability. 
Moreover, we find many non-K\"ahler-Einstein 
Fano manifolds of dimension three. 
\end{abstract}

\setcounter{tocdepth}{1}
\tableofcontents

\section{Introduction}\label{intro_section}

Let $X$ be a \emph{$\Q$-Fano variety}, that is, a projective variety which has at most 
log-terminal singularities such that the anticanonical divisor $-K_X$ of $X$ is 
ample ($\Q$-Cartier). If $X$ is a \emph{Fano manifold} (i.e., $X$ is smooth), 
then it is known that the existence of 
K\"ahler-Einstein metrics is equivalent to K-polystability of the pair $(X, -K_X)$ 
(see \cite{tian1, don05, CT, stoppa, mab1, mab2, Berman, CDS1, CDS2, CDS3, tian2}). 
The notion of K-polystability is weaker then the notion of K-stability and is 
stronger than the notion of K-semistability. 
Our main interest is to test K-(semi)stability 
of the pair $(X, -K_X)$. (In this paper, we do not treat K-polystability.) 
However, in general, it is hard to test K-(semi)stability of the pair $(X, -K_X)$. 
To overcome the difficulties, Ross and Thomas introduced the notion of 
\emph{slope stability} in \cite{RT}. 
This is an epoch-making notion since we can easily calculate. 
In particular, slope stability of Fano manifolds along divisors is interpreted by 
the volume function (see \cite{fjt}). 
However, unfortunately, slope stability is strictly weaker than K-stability. 
In fact, as in \cite[Example 7.6]{PR}, if $X$ is the blowup of $\pr^2$ along 
distinct two points, then $X$ is a toric Fano manifold and 
$(X, -K_X)$ is not K-semistable, but $(X, -K_X)$ is slope stable. 

The purpose of this paper is to get a new necessary condition, called 
\emph{divisorial stability}, of K-(semi)stability 
of the pair $(X, -K_X)$ which is strictly sharper than slope (semi)stability along divisors
(if $X$ is smooth) and is easy to test. 
We quickly describe the key idea. 
We consider the case $X$ is smooth for simplicity. 
For slope stability along a nonzero effective divisor $D$, 
we consider the following flag ideal 
\[
\sO_X(-MD)+\sO_X(-(M-1)D)t^1+\cdots+\sO_X(-D)t^{M-1}+(t^M)\subset
\sO_{X\times\A^1_t}
\]
for some $M\in \Z_{>0}$. For divisorial stability along $D$, 
we consider the flag ideal
\[
I_M+I_{M-1}t^1+\cdots+I_1t^{M-1}+(t^M)\subset\sO_{X\times\A^1_t}, 
\]
where each $I_j\subset \sO_X$ is the \emph{base ideal} of the sub linear system 
of the complete linear system $|-rK_X|$ (for some $r$) associated to the embedding 
\[
H^0(X, \sO_X(-rK_X-jD))\subset H^0(X, \sO_X(-rK_X))
\]
(see Section \ref{div_section} in detail). The point is, despite each $I_j$ is smaller than 
$\sO_X(-jD)$, the associated global sections are same. More precisely, 
as subspaces of $H^0(X, \sO_X(-rK_X))$, the following equality holds: 
\[
H^0(X, \sO_X(-rK_X-jD))=H^0(X, \sO_X(-rK_X)\cdot I_j). 
\]
Thanks to this property, together with the finite generation of certain section rings 
(due to \cite{BCHM}), we can easily calculate the Donaldson-Futaki invariant of 
the semi test configuration 
(called the \emph{basic semi test configuration}) obtained by the above flag ideal. 
The following definition is not the original definition 
but a consequence of a certain transformation.

\begin{definition}[{see Definition \ref{divst_dfn}, Proposition \ref{divst_prop} and 
Theorem \ref{volint_thm}}]\label{intro_dfn}
Let $X$ be a $\Q$-Fano variety. 
\begin{enumerate}
\renewcommand{\theenumi}{\arabic{enumi}}
\renewcommand{\labelenumi}{(\theenumi)}
\item\label{intro_dfn1}
Let $D$ be a nonzero effective Weil divisor on $X$. 
The pair $(X, -K_X)$ is said to be \emph{divisorially stable} (resp.\ 
\emph{divisorially semistable}) \emph{along $D$} if the value 
\[
\eta(D):=\vol_X(-K_X)-\int_0^{\infty}\vol_X(-K_X-xD)dx
\]
satisfies that $\eta(D)>0$ (resp.\ $\eta(D)\geq 0$), 
where $\vol_X$ is the volume function (see Definition \ref{vol_dfn}). 
\item\label{intro_dfn2}
The pair $(X, -K_X)$ is said to be \emph{divisorially stable} (resp.\ 
\emph{divisorially semistable}) if $(X, -K_X)$ is divisorially stable (resp.\ divisorially 
semistable) along any nonzero effective Weil divisor. 
\end{enumerate}
\end{definition}

Divisorial (semi)stability is weaker than K-(semi)stability 
(see Remarks \ref{divst_rmk} and \ref{divvsK_rmk}). Moreover, if $X$ is smooth, 
then it is stronger than slope (semi)stability along divisors (see Corollary \ref{slope_cor} 
\eqref{slope_cor22}). 
Moreover, we show that divisorial semistability is equivalent to K-semistability 
if $X$ is toric. 

\begin{thm}[{see Corollary \ref{toric_cor}}]\label{introtoric_thm}
Let $X$ be a toric $\Q$-Fano variety and let $P\subset M_\R$ be the associated 
polytope $($see Section \ref{toric_section}$)$. 
Then the pair $(X, -K_X)$ is divisorially semistable if and only if the barycenter of $P$ 
is equal to the origin. 
\end{thm}

By the similar argument, we show that divisorial (semi)stability can be interpreted 
by a certain structure property of the Okounkov body of $-K_X$. 
We remark that the relationship between K-stability and Okounkov bodies has 
been already pointed out in \cite{nystrom}.

\begin{thm}[{see Theorem \ref{Okounkov_thm}}]\label{intoOkounkov_thm}
Let $X$ be a $\Q$-Fano variety of dimension $n$ and 
\[
Z_\bullet\colon X=Z_0\supset Z_1\supset\dots\supset Z_n=\{\text{point}\}
\]
be an admissible flag in the sense of \cite[(1.1)]{LM}. 
Let $\Delta(-K_X)\subset\R^n$ be the Okounkov body of $-K_X$ with respects to 
$Z_\bullet$ in the sense of \cite{LM}. If the pair $(X, -K_X)$ is K-stable 
$($resp.\ K-semistable$)$, then the first coordinate $b_1$ of the barycenter of 
$\Delta(-K_X)$ satisfies that $b_1<1$ $($resp.\ $b_1\leq 1$$)$.
\end{thm}

We also see that we can calculate the Donaldson-Futaki invariants 
of the basic semi test configurations 
via intersection numbers after we run certain minimal model program 
(\emph{MMP}, in short) with scaling (see Section \ref{MMP_section}). 
Thus divisorial (semi)stability is easy to test. 
In fact, we determine divisorial 
(semi)stability for all smooth $X$ of dimension at most three 
(see Proposition \ref{surf_prop} and Theorem \ref{three_thm}). 
As an immediate corollary, we find many (possibly non-toric) 
non-K\"ahler-Einstein Fano manifolds 
of dimension three. (We heavily depends on the classification result of 
Mori and Mukai \cite{MoMu}.)

\begin{thm}[{see Theorem \ref{three_thm}}]\label{introthree_thm}
Let $X$ be a non-toric Fano manifold of dimension three. 
Assume that $X$ belongs to one of the following list 
\begin{itemize}
\item
No.\ 23, No.\ 28, No.\ 30 or No.\ 31 in Table 2 $(\rho(X)=2)$, 
\item
No.\ 14, No.\ 16, No.\ 18, No.\ 21, No.\ 22, No.\ 23 or No.\ 24 in Table 3 
$(\rho(X)=3)$, or
\item
No.\ 5 or No.\ 8 in Table 4 $(\rho(X)=4)$
\end{itemize}
in \cite{MoMu}. Then the pair $(X, -K_X)$ is not K-semistable. 
In particular, $X$ does not admit K\"ahler-Einstein metrics. 
\end{thm}

\begin{remark}\label{suss_rmk}
After the author wrote the article, the author found the results \cite{suss, IS}.
It has been already known that 
some (but not all) of $X$ in the list of Theorem \ref{introthree_thm} does not admit 
K\"ahler-Einstein metrics (see \cite[Theorem 1.1]{suss} and \cite[Theorem 6.1]{IS}). 
\end{remark}

The article is organized as follows. In Section \ref{prelim_section}, 
we recall the notion of K-(semi)stability and geography of models. 
We also consider the volume function (and the restricted volume functions) 
of $X$ for possibly non-$\R$-Cartier divisors. 
In Section \ref{div_section}, we define the notion of divisorial stability. 
We construct the basic semi test configuration from a nonzero effective Weil 
divisor. In Proposition \ref{divst_prop}, we see that we can calculate 
the Donaldson-Futaki invariant of basic semi test configurations via the 
growth of the sum of the dimension of certain global sections. 
In Section \ref{RR_section}, we prove a kind of the asymptotic Riemann-Roch theorem 
in order to calculate the Donaldson-Futaki invariants of basic semi test configurations. 
Thanks to the argument in Section \ref{RR_section}, we can rephrase 
divisorial (semi)stability in Section \ref{interpret_section}. 
In Section \ref{toric_section}, we consider the case that $X$ is toric. We see in 
Corollary \ref{toric_cor} that if the barycenter of the associated polytope 
is not the origin, then $(X, -K_X)$ is not divisorially semistable along some 
torus invariant prime divisor. As a corollary, for a non-K-semistable toric $\Q$-Fano 
variety, we can explicitly construct a flag ideal such that the Donaldson-Futaki invariant 
of the semi test configuration obtained by the flag ideal is strictly negative 
(see Examples \ref{toric1_ex} and \ref{toric2_ex}). 
In Section \ref{Okounkov_section}, by the similar argument in Section \ref{toric_section}, 
we show that divisorial stability along prime divisors can be interpreted by 
the structure property of Okounkov bodies of $-K_X$ 
(see Theorem \ref{Okounkov_thm}). 
In Section \ref{MMP_section}, we rephrase the condition of divisorial (semi)stability 
via MMP with scaling. As a corollary, we see the relationship between divisorial 
(semi)stability and slope (semi)stability along divisors. 
In Section \ref{basic_section}, we see some basic properties of divisorial stability. 
Moreover, we see some examples in order to prove Theorem \ref{three_thm}. 
Finally, in Section \ref{three_section}, we determine divisorial (semi)stability for 
all Fano manifolds of dimension three.

\begin{ack}
The author thank Doctor Yuji Odaka, who suggested him to write down 
Section \ref{three_section}, and Professor Toru Tsukioka and Doctor Kazunori 
Yasutake, who informed him the book \cite{mat}. 
The author is partially supported by a JSPS Fellowship for Young Scientists. 
\end{ack}

Throughout this paper, we work in the category of algebraic (separated and of 
finite type) scheme over the complex number field $\C$. 
A \emph{variety} means a reduced and irreducible algebraic scheme. 
For the theory of minimal model program, we refer the readers to \cite{KoMo}. 
For a complete variety $X$, $\rho(X)$ denotes the Picard number of $X$. 
For a normal projective variety $X$, $\Nef(X)$ (resp.\ $\overline{\Eff}(X)$) denotes the nef (resp.\ pseudo-effective) cone, that is, the closure of the cone in $\ND(X)$ 
spanned by classes of nef (resp.\ effective) 
divisors on $X$, and $\BIG(X)$ denotes the interior of the cone $\overline{\Eff}(X)$. 
For a Weil divisor $D$ on a normal projective variety $X$, the 
\emph{divisorial sheaf} on $X$ is denoted by $\sO_X(D)$. More precisely, 
for any open subscheme $U\subset X$, the section of $\sO_X(D)$ on $U$ is defined by 
\[
\{f\in k(X)\,|\,\DIV(f)|_U+D|_U\geq 0\},
\]
where $k(X)$ is the function field of $X$.

\section{Preliminaries}\label{prelim_section}

In this section, we fix a $\Q$-Fano variety $X$ of dimension $n$. 

\subsection{K-stability}\label{K_section}

In this section, we recall the notion of K-stability. 

\begin{definition}[{\cite{tian1, don, RT, odk, odk15}}]\label{K_dfn}
\begin{enumerate}
\renewcommand{\theenumi}{\arabic{enumi}}
\renewcommand{\labelenumi}{(\theenumi)}
\item\label{K_dfn1}
A \emph{flag ideal} $\sI$ is a coherent ideal sheaf $\sI\subset\sO_{X\times\A_t^1}$
of the form 
\[
\sI=I_M+I_{M-1}t^1+\cdots+I_1t^{M-1}+(t^M)\subset\sO_{X\times\A_t^1},
\]
where $\sO_X\supset I_1\supset\cdots\supset I_M$ is a decreasing sequence of 
coherent ideal sheaves. 
\item\label{K_dfn2}
Let $r\in\Z_{>0}$ with $-rK_X$ Cartier. 
A \emph{semi test configuration} $(\sB, \sL)/\A^1$ \emph{of} 
$(X, -rK_X)$ \emph{obtained by} $\sI$ is defined by the following datum: 
\begin{itemize}
\item
$\Pi\colon\sB\to X\times\A^1$ is the blowup along $\sI$, and $E_\sB\subset\sB$ 
is the Cartier divisor defined by $\sO_\sB(-E_\sB)=\sI\cdot\sO_\sB$,
\item
$\sL:=\Pi^*p_1^*\sO_X(-rK_X)\otimes\sO_\sB(-E_\sB)$, where 
$p_1\colon X\times\A^1\to X$ is the first projection, 
\end{itemize}
and we require the following: 
\begin{itemize}
\item
$\sI$ is not of the form $(t^M)$, and
\item
$\sL$ is semiample  over $\A^1$.
\end{itemize}
\item\label{K_dfn3}
Let $\alpha\colon(\sB, \sL)\to\A^1$ be the semi test configuration of $(X, -rK_X)$
obtained by $\sI$. Then the multiplicative group $\G_m$ naturally acts on 
$(\sB, \sL)$ and the morphism $\alpha$ is $\G_m$-invariant, where the action 
$\G_m\times\A^1\to\A^1$ is in a standard way $(a, t)\mapsto at$. 
For $k\in\Z_{>0}$, $\G_m$ also naturally acts on $(\alpha_*\sL^{\otimes k})|_{\{0\}}$. 
Let $w(k)$ be the total weight of the action. It is known that $w(k)$ is a polynomial 
of degree at most $n+1$ for $k\gg 0$. Let $w_{n+1}$, $w_n$ be the $(n+1)$-th, 
$n$-th coefficient of $w(k)$, respectively. We define the 
\emph{Donaldson-Futaki invariant} $\DF(\sB, \sL)$ of $(\sB, \sL)/\A^1$ such that 
\begin{eqnarray*}
\DF(\sB, \sL) & := & \frac{\left((-rK_X)^{\cdot n-1}\cdot(-K_X)\right)}
{2\cdot(n-1)!}w_{n+1}-
\frac{((-rK_X)^{\cdot n})}{n!}w_n\\
 & = & \frac{((-rK_X)^{\cdot n})}{n!}\left(\frac{n}{2r}w_{n+1}-w_n\right).
\end{eqnarray*}
\item\label{K_dfn4}
We say that the pair $(X, -K_X)$ is \emph{K-stable} (resp.\ \emph{K-semistable}) 
if $\DF(\sB, \sL)>0$ (resp.\ $\geq 0$) holds for any $r\in\Z_{>0}$, for any flag ideal 
$\sI$, and for any semi test configuration $(\sB, \sL)/\A^1$ of $(X, -rK_X)$ 
obtained by $\sI$. 
\end{enumerate}
\end{definition}

\subsection{On geography of models}\label{KKL_section}

In this section, we recall the theory of ``geography of models" introduced in 
\cite[Section 6]{shokurov}. 
For the notation in this section, we refer the readers to \cite{KKL}. 

\begin{definition}[{\cite[Definition 2.3]{KKL}}]\label{KKL_dfn}
Let $V$ be a normal projective variety, $E_V$ be an $\R$-Cartier $\R$-divisor 
on $V$, and $\phi\colon V \dashrightarrow W$ be a contraction map to a normal 
projective variety $W$ such that $E_W:=\phi_*E_V$ is $\R$-Cartier. 
\begin{enumerate}
\renewcommand{\theenumi}{\arabic{enumi}}
\renewcommand{\labelenumi}{(\theenumi)}
\item\label{KKL_dfn1}
The map $\phi$ is said to be \emph{$E_V$-nonpositive} if $\phi$ is birational, and for a 
common resolution $(p, q)\colon\tilde{V}\to V\times W$, we can write 
$p^*E_V=q^*E_W+F$, where $F$ is effective and $q$-exceptional. 
\item\label{KKL_dfn2}
The map $\phi$ is said to be a \emph{semiample model of $E_V$} if $\phi$ is 
$E_V$-nonpositive and $E_W$ is semiample. 
\item\label{KKL_dfn3}
The map $\phi$ is said to be the \emph{ample model of $E_V$} is there exist a 
birational contraction map $\phi'\colon V\dashrightarrow W'$ and a morphism 
$\psi\colon W'\to W$ with connected fibers such that $\phi=\psi\circ\phi'$, 
the map $\phi'$ is a semiample model of $E_V$ and 
$\phi'_*E_V=\psi^*A$, where $A$ is an ample $\R$-divisor on $W$. 
\end{enumerate}
\end{definition}

Throughout the end of the section, we fix a nonzero effective Weil divisor $D$ 
on $X$. 
By \cite[Corollary 1.4.3]{BCHM}, we can take a projective small $\Q$-factorial 
modification morphism $\sigma\colon\tilde{X}\to X$, that is, $\sigma$ is a projective birational morphism which is isomorphism in codimension one and $\tilde{X}$ is 
$\Q$-factorial. We set $\tilde{D}:=\sigma^{-1}_*D$. 

\begin{lemma}\label{MDS_lem}
The variety $\tilde{X}$ is a Mori dream space in the sense of \cite[Definition 1.10]{HK}. 
\end{lemma}

\begin{proof}
Since $\tilde{X}$ is projective, having at most log-terminal singularities 
and $-K_{\tilde{X}}$ is nef and big, 
there exists an effective $\Q$-divisor $\Delta$ on $\tilde{X}$ such that 
the pair $(\tilde{X}, \Delta)$ is klt and $-(K_{\tilde{X}}+\Delta)$ is an ample $\Q$-divisor. 
Thus the assertion follows by \cite[Corollary 1.3.2]{BCHM}.
\end{proof}

\begin{definition}
Let $X$, $D$ and $\sigma$ be as above. 
The \emph{pseudo-effective threshold} $\tau(D)$ \emph{of} $D$ 
\emph{with respects to} $(X, -K_X)$ is defined by: 
\[
\tau(D):=\max\{\tau\in\R_{>0}\,\,|\,\,-K_{\tilde{X}}-\tau\tilde{D}
\in\overline{\Eff}(\tilde{X})\}.
\]
\end{definition}

The following theorem is important in this paper. 

\begin{thm}[{\cite[Theorem 4.2]{KKL}}]\label{model_thm}
Let $X$, $D$ and $\sigma$ be as above. 
Then there exist 
\begin{itemize}
\item
an increasing sequence of rational numbers
\[
0=\tau_0<\tau_1<\cdots<\tau_m=\tau(D),
\] 
\item
normal projective varieties
$X_1,\dots,X_m$, and 
\item
mutually distinct birational contraction maps
$\phi_i\colon\tilde{X}\dashrightarrow X_i$  
$(1\leq i\leq m)$
\end{itemize}
such that the following hold: 
\begin{itemize}
\item
for any $x\in[\tau_{i-1}, \tau_i]$, the map $\phi_i$ is a semiample model 
of $-K_{\tilde{X}}-x\tilde{D}$, and 
\item
if $x\in(\tau_{i-1}, \tau_i)$, then the map $\phi_i$ is the ample model of 
$-K_{\tilde{X}}-x\tilde{D}$. 
\end{itemize}
\end{thm}

\begin{proof}
Follows immediately from Lemma \ref{MDS_lem} and \cite[Theorem 4.2]{KKL}. 
\end{proof}

\begin{definition}\label{model_dfn}
The sequence $\{(\tau_i, X_i)\}_{1\leq i\leq m}$ obtained in Theorem \ref{model_thm} 
is called the \emph{ample model sequence of} $(X, -K_X; -D)$. We set 
$D_i:=(\phi_i)_*\tilde{D}$ for $1\leq i\leq m$. 
\end{definition}

\begin{remark}\label{model_rmk}
\begin{enumerate}
\renewcommand{\theenumi}{\arabic{enumi}}
\renewcommand{\labelenumi}{(\theenumi)}
\item\label{model_rmk1}
Since $\tau_{i-1}<\tau_i$, both $K_{X_i}$ and $D_i$ are $\Q$-Cartier divisors on $X_i$. 
\item\label{model_rmk2}
For any $x\in(\tau_{i-1}, \tau_i)\cap\Q$ and for any $k_0\in\Z_{>0}$ with 
$-k_0K_{\tilde{X}}-k_0x\tilde{D}$ Cartier, we have 
\[
X_i\simeq\Proj\bigoplus_{k\geq 0}H^0\left(\tilde{X}, \sO_{\tilde{X}}(-kk_0K_{\tilde{X}}
-kk_0x\tilde{D})\right)
\]
by \cite[Remark 2.4 (i)]{KKL}. Since $\sigma\colon\tilde{X}\to X$ is small, 
the above is also isomorphic to 
\[
\Proj\bigoplus_{k\geq 0}H^0\left(X, \sO_X(-kk_0K_X-kk_0xD)\right).
\]
Thus the ample model sequence of $(X, -K_X; -D)$ does not depend on the choice 
of $\sigma$. In particular, the pseudo-effective threshold $\tau(D)$ does not 
depend on the choice of $\sigma$ and $\tau(D)\in\Q_{>0}$ holds. 
\item\label{model_rmk3}
The map $\phi_i$ is $K_{\tilde{X}}$-nonpositive since $-K_{\tilde{X}}$ is nef (see 
\cite[Lemma 2.5]{KKL}). 
Hence $X_i$ has at most log-terminal singularities.
\end{enumerate}
\end{remark}

\subsection{On the volume functions}\label{volfcn_section}

In this section, we recall the theory of the volume functions and the restricted 
volume functions. We refer the readers to \cite{L} and \cite{LM}. 
In this section, we fix a projective small $\Q$-factorial modification 
$\sigma\colon\tilde{X}\to X$ and we set $\tilde{D}:=\sigma^{-1}_*D$ as in 
Section \ref{KKL_section}. 

\begin{definition}\label{vol_dfn}
For any $x\in[0, +\infty)$, we define 
\[
\vol_X(-K_X-xD):=\vol_{\tilde{X}}\left(-K_{\tilde{X}}-x\tilde{D}\right), 
\]
where $\vol_{\tilde{X}}$ is the volume function on $\tilde{X}$ 
(see \cite[Corollary 2.2.45]{L}). Thus the function $\vol_X(-K_X-xD)$ is continuous 
over $[0, +\infty)$, and for any $x\in[0, +\infty)\cap\Q$, 
$\vol_X(-K_X-xD)$ is equal to 
\begin{eqnarray*}
\limsup_{k\to\infty}
\frac{h^0\left(\tilde{X}, \sO_{\tilde{X}}(-kK_{\tilde{X}}+\lfloor-kx\tilde{D}\rfloor)\right)}
{k^n/n!}
=\limsup_{k\to\infty}
\frac{h^0\left(X, \sO_X(-kK_X+\lfloor-kxD\rfloor)\right)}
{k^n/n!}, 
\end{eqnarray*}
where $\lfloor\bullet\rfloor$ is the round-down 
(see \cite[Notation 0.4 (12)]{KoMo}). 
In particular, $\vol_X(-K_X-xD)$ does not depend on the choice of $\sigma$. 
\end{definition}

\begin{lemma}\label{vol_lem}
Let $x\in[0, +\infty)$. 
\begin{enumerate}
\renewcommand{\theenumi}{\arabic{enumi}}
\renewcommand{\labelenumi}{(\theenumi)}
\item\label{vol_lem1}
$\vol_X(-K_X-xD)=0$ holds if and only if $x\geq\tau(D)$. 
\item\label{vol_lem2}
If $x\in[\tau_{i-1}, \tau_i]$, then $\vol_X(-K_X-xD)=((-K_{X_i}-xD_i)^{\cdot n})$. 
\end{enumerate}
\end{lemma}

\begin{proof}
\eqref{vol_lem1}
The $\R$-divisor $-K_{\tilde{X}}-x\tilde{D}$ is big if and only if $x\in[0, \tau(D))$. 
Thus the assertion follows. 

\eqref{vol_lem2}
Both $\vol_X(-K_X-xD)$ and $((-K_{X_i}-xD_i)^{\cdot n})$ are continuous functions 
over $x\in[\tau_{i-1}, \tau_i]$. Thus we can assume that $x\in(\tau_{i-1}, \tau_i)\cap\Q$. 
We have 
\begin{eqnarray*}
\vol_X(-K_X-xD) = \limsup_{k\to\infty}
\frac{h^0\left(X_i, \sO_{X_i}(-kK_{X_i}+\lfloor-kxD_i\rfloor)\right)}{k^n/n!}
= ((-K_{X_i}-xD_i)^{\cdot n})
\end{eqnarray*}
by \cite[Remark 2.4 (i)]{KKL} and the Serre vanishing theorem. 
\end{proof}

We define the notion of restricted volume functions. 

\begin{definition}\label{rvol_dfn}
For $x\in[0, \tau(D))$, we define the \emph{restricted volume} 
$\vol_{X|D}(-K_X-xD)$ such that
\[
\vol_{X|D}(-K_X-xD):=-\frac{1}{n}\frac{d}{dx}\vol_X(-K_X-xD).
\]
Note that the function $\vol_X(-K_X-xD)$ is $\sC^1$ over $x\in[0, \tau(D))$ 
by \cite[Theorem A]{BFJ}. Thus $\vol_{X|D}(-K_X-xD)$ is 
well-defined and continuous over $x\in[0, \tau(D))$. 
\end{definition}

\begin{proposition}\label{rvol_prop}
Assume that $x\in[\tau_{i-1}, \tau_i]$. Then 
\[
\vol_{X|D}(-K_X-xD)=((-K_{X_i}-xD_i)^{\cdot n-1}\cdot D_i). 
\]
In particular, 
$\vol_{X|D}(-K_X-xD)\geq 0$ for any $x\in[0, \tau(D))$. 
\end{proposition}

\begin{proof}
Follows from Lemma \ref{vol_lem} \eqref{vol_lem2}. 
\end{proof}

\begin{definition}\label{rvf_dfn}
We define 
\[
\vol_{X|D}(-K_X-\tau(D)D):=\left((-K_{X_m}-\tau(D)D_m)^{\cdot n-1}\cdot D_m\right)
\]
for convenience. 
By Proposition \ref{rvol_prop}, $\vol_{X|D}(-K_X-xD)$ is continuous over 
$x\in[0, \tau(D)]$. 
\end{definition}

\begin{proposition}\label{vr_prop}
Assume that $D$ is a prime divisor and $x\in[0, \tau(D))\cap\Q$. 
Then the value $\vol_{X|D}(-K_X-xD)$ coincides with the usual restricted volume 
$\vol_{\tilde{X}|\tilde{D}}(-K_{\tilde{X}}-x\tilde{D})$ in \cite{ELMNP, LM}. 
More precisely, 
\[
\vol_{X|D}(-K_X-xD)=\limsup_{\substack{k\to\infty\\-kK_{\tilde{X}}-kx\tilde{D}:
\text{ Cartier}}}
\frac{\dim V_{(k)}}{k^{n-1}/(n-1)!} 
\]
holds, where $V_{(k)}$ is the image of the homomorphism
\[
H^0\left(\tilde{X}, \sO_{\tilde{X}}(-kK_{\tilde{X}}-kx\tilde{D})\right)\to
H^0\left(\tilde{D}, \sO_{\tilde{X}}(-kK_{\tilde{X}}-kx\tilde{D})|_{\tilde{D}}\right).
\]
\end{proposition}

\begin{proof}
We know that $\Exc(\sigma)=\B_+(-K_{\tilde{X}})$ by \cite[Proposition 2.3]{BBP}, 
where $\Exc(\sigma)$ is the exceptional locus of $\sigma$ and $\B_+(-K_{\tilde{X}})$ 
is the augmented base locus of $-K_{\tilde{X}}$ (see \cite[Section 2.4]{LM}). 
In particular, $\tilde{D}\not\subset\B_+(-K_{\tilde{X}})$. Thus, by the proof of 
\cite[Corollary 4.25]{LM}, the assertion follows. 
\end{proof}

\section{Divisorial stability}\label{div_section}

We define the notion of divisorial stability for $\Q$-Fano varieties. 
In this section, we fix a $\Q$-Fano variety $X$ of dimension $n$ and a nonzero 
effective Weil divisor $D$ on $X$. 

By Lemma \ref{MDS_lem}, the graded $\C$-algebra
\[
\bigoplus_{k, j\geq 0}H^0\left(X, \sO_X(-kK_X-jD)\right)
\]
is finitely generated. 
We remark that the above algebra is equal to 
\[
\bigoplus_{\substack{k\geq 0\\ 0\leq j\leq k\tau(D)}}H^0\left(X, \sO_X(-kK_X-jD)\right),
\]
since $H^0(X, \sO_X(-kK_X-jD))=0$ for $j>k\tau(D)$. 

\begin{definition}\label{gen_dfn}
We say that a positive integer $r\in\Z_{>0}$ \emph{satisfies the generating property 
with respects to} $(X, -K_X; -D)$ if $-rK_X$ is Cartier, $r\tau(D)\in\Z_{>0}$, and
the $\C$-algebra
\[
\bigoplus_{\substack{k\geq 0\\ 0\leq j\leq kr\tau(D)}}
H^0\left(X, \sO_X(-krK_X-jD)\right)
\]
is generated by 
\[
\bigoplus_{0\leq j\leq r\tau(D)}H^0\left(X, \sO_X(-rK_X-jD)\right)
\]
as a $\C$-algebra. 
\end{definition}

\begin{remark}\label{gen_rmk}
\begin{enumerate}
\renewcommand{\theenumi}{\arabic{enumi}}
\renewcommand{\labelenumi}{(\theenumi)}
\item\label{gen_rmk1}
If $r\in\Z_{>0}$ is sufficiently divisible, then $r$ satisfies the generating property 
with respects to $(X, -K_X; -D)$. 
\item\label{gen_rmk2}
We assume that a positive integer $r\in\Z_{>0}$ satisfies the generating property with respects to 
$(X, -K_X; -D)$. Then the $\C$-algebra
\[
\bigoplus_{k\geq 0}H^0\left(X, \sO_X(-krK_X)\right)
\]
is generated by $H^0(X, \sO_X(-rK_X))$. In particular, the divisor $-rK_X$ is very ample. 
\end{enumerate}
\end{remark}

Throughout the end of the section, we fix $r\in\Z_{>0}$ which satisfies the 
generating property with respects to $(X, -K_X; -D)$. 
From now on, we construct a semi test configuration of $(X, -rK_X)$. 

For any $j\geq 0$, we set the coherent ideal sheaf $I_j\subset\sO_X$ defined by 
the image of the composition of the homomorphisms 
\begin{eqnarray*}
H^0\left(X, \sO_X(-rK_X-jD)\right)\otimes_\C\sO_X(rK_X)
\hookrightarrow
H^0\left(X, \sO_X(-rK_X)\right)\otimes_\C\sO_X(rK_X)\to\sO_X.
\end{eqnarray*}
In other words, $I_j$ is the base ideal of the sub linear system of the complete 
linear system $|-rK_X|$
associates to the embedding 
\[
H^0\left(X, \sO_X(-rK_X-jD)\right)\subset H^0\left(X, \sO_X(-rK_X)\right).
\]
Obviously, we have 
\[
\sO_X=I_0\supset I_1\supset\cdots\supset I_{r\tau(D)}\supset I_{r\tau(D)+1}=0. 
\]

For any $k\in\Z_{>0}$ and $j\in\Z_{\geq 0}$, we define the coherent ideal sheaf 
$J_{(k,j)}\subset\sO_X$ such that 
\[
J_{(k, j)}:=\sum_{\substack{j_1+\cdots+j_k=j\\ j_1,\dots,j_k\geq 0}}I_{j_1}\cdots I_{j_k}.
\]

\begin{lemma}\label{gen_lem}
The above ideal sheaf $J_{(k, j)}\subset\sO_X$ is equal to the base ideal of the 
sub linear system of the complete linear system $|-krK_X|$ associates to 
the embedding
\[
H^0\left(X, \sO_X(-krK_X-jD)\right)\subset H^0\left(X, \sO_X(-krK_X)\right).
\]
In other words, $J_{(k, j)}$ is equal to the image of the composition of the 
homomorphisms
\begin{eqnarray*}
H^0\left(X, \sO_X(-krK_X-jD)\right)\otimes_\C\sO_X(krK_X)
\hookrightarrow
H^0\left(X, \sO_X(-krK_X)\right)\otimes_\C\sO_X(krK_X)\to\sO_X.
\end{eqnarray*}
In particular, we have 
\[
H^0\left(X, \sO_X(-krK_X-jD)\right)=H^0\left(X, \sO_X(-krK_X)\cdot J_{(k, j)}\right)
\]
as subspaces of $H^0(X, \sO_X(-krK_X))$.
\end{lemma}

\begin{proof}
We write 
\[
V_{(k, j)}:=H^0\left(X, \sO_X(-krK_X-jD)\right)
\]
for simplicity. 
By the definition of $r$, the homomorphism
\begin{eqnarray*}
\bigoplus_{\substack{j_1+\cdots+j_k=j\\ j_1,\dots,j_k\geq 0}}
V_{(1, j_1)}\otimes_\C\cdots\otimes_\C V_{(1, j_k)}
\to V_{(k, j)}
\end{eqnarray*}
is surjective. 
For any $j_i$, the image of the homomorphism 
\[
V_{(1, j_i)}\otimes_\C\sO_X(rK_X)\to\sO_X
\]
is equal to $I_{j_i}$. Thus the image of the homomorphism 
\begin{eqnarray*}
\bigoplus_{\substack{j_1+\cdots+j_k=j\\ j_1,\dots,j_k\geq 0}}
V_{(1, j_1)}\otimes_\C\cdots\otimes_\C V_{(1, j_k)}\otimes_\C\sO_X(krK_X)
\to\sO_X
\end{eqnarray*}
is equal to 
\[
\sum_{\substack{j_1+\cdots+j_k=j\\ j_1,\dots,j_k\geq 0}}I_{j_1}\cdots I_{j_k}.
\]
This is nothing but $J_{(k, j)}$. 
\end{proof}

We consider the following flag ideal 
\[
\sI:=I_{r\tau(D)}+I_{r\tau(D)-1}t^1+\cdots+ I_1t^{r\tau(D)-1}+(t^{r\tau(D)})
\subset\sO_{X\times\A^1_t}.
\]
By construction, for any $k\in\Z_{>0}$, we have 
\[
\sI^k=J_{(k, kr\tau(D))}+J_{(k, kr\tau(D)-1)}t^1+\cdots+J_{(k, 1)}t^{kr\tau(D)-1}
+(t^{kr\tau(D)}).
\]
Let $\Pi\colon\sB\to X\times\A^1$ be the blowup along $\sI$, 
let $E_\sB\subset\sB$ be the Cartier divisor on $\sB$ defined by 
$\sO_\sB(-E_\sB)=\sI\cdot\sO_\sB$. Moreover, we set 
$\sL:=\Pi^*p_1^*\sO_X(-rK_X)\otimes\sO_\sB(-E_\sB)$ and 
$\alpha:=p_2\circ\Pi\colon\sB\to\A^1$, where $p_1$ or $p_2$ is the first or the second 
projection morphism, respectively. 

\begin{lemma}\label{stc_lem}
$\alpha\colon(\sB, \sL)\to\A^1$ is a semi test configuration of $(X, -rK_X)$. 
\end{lemma}

\begin{proof}
It is enough to show that $\sL$ is semiample over $\A^1$. 
For any $k\in\Z_{>0}$ and $j\in\Z_{\geq 0}$, by Lemma \ref{gen_lem}, the homomorphism
\[
H^0\left(X, \sO_X(-krK_X)\cdot J_{(k, j)}\right)\otimes_\C\sO_X\to
\sO_X(-krK_X)\cdot J_{(k, j)}
\]
is surjective. Thus, for any $k\in\Z_{>0}$, the homomorphism 
\[
H^0\left(X\times\A^1, p_1^*\sO_X(-krK_X)\cdot\sI^k\right)\otimes_{\C[t]}
\sO_{X\times\A^1}\to p_1^*\sO_X(-krK_X)\cdot \sI^k
\]
is also surjective. 
Therefore, by \cite[Lemma 5.4.24]{L}, we have 
\begin{eqnarray*}
\alpha^*\alpha_*\sL^{\otimes k}&\simeq& 
\alpha^*(p_2)_*(p_1^*\sO_X(-krK_X)\cdot\sI^k)\\
&=&\Pi^*\left(H^0\left(X\times\A^1, p_1^*\sO_X(-krK_X)\cdot\sI^k\right)
\otimes_{\C[t]}\sO_{X\times\A^1}\right)\\
&\twoheadrightarrow&\Pi^*\left(p_1^*\sO_X(-krK_X)\cdot\sI^k\right)\\
&\twoheadrightarrow&
\Pi^*p_1^*\sO_X(-krK_X)\otimes\sO_\sB(-kE_\sB)=\sL^{\otimes k} 
\end{eqnarray*}
for $k\gg 0$. 
This means that $\sL$ is semiample over $\A^1$. 
\end{proof}

\begin{definition}\label{divst_dfn}
\begin{enumerate}
\renewcommand{\theenumi}{\arabic{enumi}}
\renewcommand{\labelenumi}{(\theenumi)}
\item\label{divst_dfn1}
The above flag ideal $\sI$ is called the \emph{basic flag ideal with respects to} 
$(X, -rK_X; -D)$, and the above semi test configuration 
$\alpha\colon(\sB, \sL)\to\A^1$ is called the \emph{basic semi test configuration 
of} $(X, -rK_X)$ \emph{via} $D$. 
\item\label{divst_dfn2}
The pair $(X, -K_X)$ is said to be \emph{divisorially stable} (resp.\ \emph{divisorially 
semistable}) \emph{along} $D$ if for any $r\in\Z_{>0}$ which satisfies the generating 
property with respects to $(X, -K_X; -D)$, the basic semi test configuration 
$(\sB, \sL)/\A^1$ of $(X, -rK_X)$ via $D$ satisfies that $\DF(\sB, \sL)>0$ 
(resp.\ $\geq 0$). (We will see in Theorem \ref{eta_thm} that the definition does not 
depend on the choice of $r$.)
\item\label{divst_dfn3}
The pair $(X, -K_X)$ is said to be \emph{divisorially stable} (resp.\ \emph{divisorially 
semistable}) if the pair is divisorially stable (resp.\ divisorially semistable) along any 
nonzero effective Weil divisor. 
\end{enumerate}
\end{definition}

Let $\alpha\colon(\sB, \sL)\to\A^1$ be the basic 
semi test configuration of $(X, -rK_X)$ 
via $D$ and let $w(k)$ be the total weight of the action of $\G_m$ on 
$(\alpha_*\sL^{\otimes k})|_{\{0\}}$. By \cite[Lemma 3.3]{odk}, $w(k)$ is equal to
\begin{eqnarray*}
&& -\dim\left(\frac{H^0\left(X\times\A^1, p_1^*\sO_X(-krK_X)\right)}
{H^0\left(X\times\A^1, p_1^*\sO_X(-krK_X)\cdot\sI^k\right)}\right)\\
& = & -\dim\bigoplus_{j=0}^{kr\tau(D)-1}t^j\cdot\left(
\frac{H^0\left(X, \sO_X(-krK_X)\right)}{H^0\left(X, \sO_X(-krK_X)
\cdot J_{(k, kr\tau(D)-j)}\right)}\right)\\
 & = & -kr\tau(D)h^0\left(X, \sO_X(-krK_X)\right)+\sum_{j=1}^{kr\tau(D)}
h^0\left(X, \sO_X(-krK_X)\cdot J_{(k, j)}\right)\\
 & = & -kr\tau(D)h^0\left(X, \sO_X(-krK_X)\right)+\sum_{j=1}^{kr\tau(D)}
h^0\left(X, \sO_X(-krK_X-jD)\right).
\end{eqnarray*}
Thus we have the following: 

\begin{proposition}\label{divst_prop}
Let $(\sB, \sL)/\A^1$ be the basic semi test configuration of $(X, -rK_X)$ via $D$. 
We set 
\begin{eqnarray*}
f(k):=\sum_{j=1}^\infty h^0\left(X, \sO_X(-krK_X-jD)\right)
=\sum_{j=1}^{kr\tau(D)}h^0\left(X, \sO_X(-krK_X-jD)\right).
\end{eqnarray*}
Then $f(k)$ is a polynomial function of degree at most $n+1$ for $k\gg 0$. 
Let $f_{n+1}$, $f_n$ be the $(n+1)$-th, $n$-th coefficient of $f(k)$, respectively. 
Then we have 
\[
\DF(\sB, \sL)=\frac{r^{2n}((-K_X)^{\cdot n})}{2\cdot(n!)^2}\eta(D), 
\]
where 
\[
\eta(D):=\frac{n!}{r^{n+1}}(nf_{n+1}-2rf_n).
\]
\end{proposition}

\begin{remark}\label{divst_rmk}
Obviously, if $(X, -K_X)$ is K-stable (resp.\ K-semistable), then $(X, -K_X)$ is 
divisorially stable (resp.\ divisorially semistable). 
In particular, if $X$ admits K\"ahler-Einstein metrics, then $(X, -K_X)$ 
is divisorially semistable by \cite{don, Berman}. 
\end{remark}

\begin{remark}\label{filt_rmk}
The relationship between test configurations and filtered graded linear series 
is discussed by many authors. See \cite{nystrom, sz} and references therein.
\end{remark}

\section{On the asymptotic Riemann-Roch theorem}\label{RR_section}

In this section, we prove the following proposition. 

\begin{proposition}\label{RR_prop}
Let $V$ be a normal projective variety of dimension $n$, let $H_V$, $D_V$ be 
$\Q$-Cartier Weil divisors on $V$ such that $D_V$ is effective, and let $a$, $b$ be 
rational numbers with $a<b$ such that $H_V-xD_V$ is ample for any $x\in(a, b)$. 
Then for any sufficiently divisible positive integer $k$, the function 
\[
v(k):=\sum_{j=ak+1}^{bk}h^0\left(V, \sO_V(kH_V-jD_V)\right)
\]
satisfies that 
\begin{eqnarray*}
v(k) & = & k^{n+1}\frac{1}{n!}\int_a^b((H_V-xD_V)^{\cdot n})dx\\
 & + & k^n\frac{1}{2\cdot (n-1)!}\int_a^b((H_V-xD_V)^{\cdot n-1}\cdot(-K_V-D_V))dx
+ O(k^{n-1}).
\end{eqnarray*}
\end{proposition}

\begin{proof}
We set
\[
v_0(k):=\sum_{j=ak+1}^{bk}\chi\left(V, \sO_V(kH_V-jD_V)\right).
\]
Pick any $c\in(a, b)\in\Q$ and pick $k_0\in\Z_{>0}$ such that both $k_0H_V$ and 
$ck_0D_V$ are Cartier. By Takao Fujita's vanishing theorem \cite[Theorem 1.4.35]{L}, 
there exists a positive integer $k_1$ which is divisible by $k_0$ such that 
\[
H^i\left(V, \sO_V\left((k_1+g_1k_0)H_V-(ck_1+g_2)D_V\right)\right)=0
\]
for any $i>0$, $g_1\in\Z_{>0}$ and $g_2\in\Z$ with $ag_1k_0\leq g_2\leq bg_1k_0$. 
Thus, for any sufficiently divisible positive integer $k$, we have 
$H^i(V, \sO_V(kH_V-jD_V))=0$ for any $i>0$ and any $j\in\Z$ with 
$ak+(c-a)k_1\leq j\leq bk+(c-b)k_1$. Hence $v(k)-v_0(k)$ is equal to
\begin{eqnarray*}
\sum_{\substack{j\in\{ak,\dots,ak+(c-a)k_1-1, \\
bk+(c-b)k_1+1,\dots,bk\}}}\left(h^0\left(V, \sO_V(kH_V-jD_V)\right)
-\chi\left(V, \sO_V(kH_V-jD_V)\right)\right).
\end{eqnarray*}
Thus $v(k)-v_0(k)=O(k^{n-1})$ (see \cite[Theorem VI.2.15]{kollar}). 
Hence it is enough to show the assertion for $v_0(k)$. 

We fix $h\in\Z_{>0}$ such that $hD_V$ is Cartier. 
We consider the $\Z/h\Z$-graded $\sO_V$-algebra 
\[
\sA:=\bigoplus_{i=0}^{h-1}\sO_V(-iD_V)
\]
defined by the effective Cartier divisor $hD_V$ (\cite[Definition 2.52]{KoMo}). 
More precisely, we use the multiplication 
\begin{eqnarray*}
\begin{cases}
\sO_V(-iD_V)\otimes\sO_V(-jD_V)\to\sO_V(-(i+j)D_V) & (i+j< h),\\
\sO_V(-iD_V)\otimes\sO_V(-jD_V)\to\sO_V(-(i+j)D_V)\xrightarrow{s}
\sO_V(-(i+j-h)D_V) & (i+j\geq h),
\end{cases}
\end{eqnarray*}
where $s\in H^0(V, \sO_V(hD_V))$ corresponds to the effective Cartier divisor $hD_V$. 
We consider the finite morphism 
\[
\theta\colon\tilde{V}:=\Spec_{\sO_V}\sA\to V. 
\]

\begin{claim}\label{RR_claim}
$\tilde{V}$ is a reduced algebraic scheme which is Gorenstein in codimension one 
and satisfies that Serre's condition $S_2$. Moreover, we have 
$K_{\tilde{V}}\sim\theta^*(K_V+(h-1)D_V)$, where $K_{\tilde{V}}$ is the canonical divisor 
of $\tilde{V}$ $($see \cite[Definition-Remark 2.7]{hartshorne}$)$.
\end{claim}

\begin{proof}[Proof of Claim \ref{RR_claim}]
It follows from the definition that $\tilde{V}$ satisfies Serre's condition $S_2$ and is 
reduced. Pick any irreducible component $D_0$ of $D_V$ and let $c_0$ be the 
coefficient of $D_V$ at $D_0$. Let $p\in D_0$ be a general point. Then $p\in V$ is 
smooth and we can take an analytic local coordinate $x_1,\dots,x_n$ around $p$ 
such that $D_V$ is defined by the equation $x_1^{c_0}=0$. Then $\tilde{V}$ 
around over $p$ is defined by the equation $(t^h-x_1^{c_0h}=0)\subset V\times\A^1_t$. 
Thus $\tilde{V}$ is Gorenstein in codimension one. Since the canonical sheaf 
$\omega_{V\times\A^1}$ of $V\times\A^1$ is generated by 
\[
\frac{d(t^h-x_1^{c_0h})}{t^{h-1}}\wedge dx_1\wedge\cdots\wedge dx_n,
\]
the canonical sheaf $\omega_{\tilde{V}}$ of $\tilde{V}$ is generated by 
\[
\frac{1}{t^{h-1}}dx_1\wedge\cdots\wedge dx_n
\]
around over $p$. Since $(t^{h-1})^h=(x_1^{c_0h})^{h-1}$, we get the assertion. 
\end{proof}

For any $k\in\Z_{>0}$ and $j\in\Z$ such that $kH_V$ is Cartier and $j$ is divisible 
by $h$, we have 
\begin{eqnarray*}
&&\chi\left(\tilde{V}, \theta^*\sO_V(kH_V-jD_V)\right)
=\chi\left(V, \sA\otimes\sO_V(kH_V-jD_V)\right)\\
&=&\sum_{i=0}^{h-1}\chi\left(V, \sO_V\left(kH_V-(i+j)D_V\right)\right).
\end{eqnarray*}
Therefore, by \cite[Lemma 3.5]{odk} and \cite[(4.16)]{RT}, 
for a sufficiently divisible positive integer $k$, $v_0(k)$ is equal to
\begin{eqnarray*}
&&\sum_{j'=ak/h+1}^{bk/h-1}\chi\left(\tilde{V}, \theta^*\sO_V(kH_V-j'hD_V)\right)
+\sum_{j\in\{ak+1,\dots,ak+h-1\}\cup\{bk\}}\chi\left(V, \sO_V(kH_V-jD_V)\right)\\
&=&\sum_{j'=ak/h+1}^{bk/h}\chi\left(\tilde{V}, \theta^*\sO_V(kH_V-j'hD_V)\right)\\
&+&k^n\frac{h-1}{n!}\left(((H_V-aD_V)^{\cdot n})-((H_V-bD_V)^{\cdot n})\right)+O(k^{n-1})\\
&=&\sum_{j'=ak/h+1}^{bk/h}\left\{\frac{k^n}{n!}
\left(\theta^*(H_V-
(j'/k)hD_V)^{\cdot n}\right)
-\frac{k^{n-1}}
{2\cdot(n-1)!}\left(\theta^*(H_V-(j'/k)hD_V)^{\cdot n-1}\cdot K_{\tilde{V}}\right)
\right\}\\
&+&k^n\frac{h-1}{n!}\left(((H_V-aD_V)^{\cdot n})-((H_V-bD_V)^{\cdot n})\right)+O(k^{n-1})\\
&=&\frac{k^{n+1}}{n!}\int_{a/h}^{b/h}h((H_V-xhD_V)^{\cdot n})dx
+k^n\biggl\{\frac{h-1}{n!}
\left(((H_V-aD_V)^{\cdot n})-((H_V-bD_V)^{\cdot n})\right)\\
&+&\frac{h}{2\cdot(n-1)!}\int_{a/h}^{b/h}\left((H_V-xhD_V)^{\cdot n-1}\cdot
(-K_V-(h-1)D_V)\right)dx\\
&+&\frac{h}{2\cdot(n-1)!}\int_{a/h}^{b/h}\left((H_V-xhD_V)^{\cdot n-1}\cdot(-hD_V)
\right)dx\biggr\}+O(k^{n-1}).
\end{eqnarray*}
Therefore we have proved Proposition \ref{RR_prop}. 
\end{proof}

\section{Interpretation of $\eta(D)$}\label{interpret_section}

In this section, we fix a $\Q$-Fano variety $X$ of dimension $n$ and a nonzero 
effective Weil divisor $D$ on $X$. In this section, we calculate the value $\eta(D)$ 
in Proposition \ref{divst_prop} via intersection numbers, via the volume functions, and 
via the restricted volume functions. 

\subsection{Via intersection numbers}\label{intersection_section}

Let $r\in\Z_{>0}$ which satisfies the generating property with respects to 
$(X, -K_X; -D)$, let $(\sB, \sL)/\A^1$ be the basic semi test configuration of 
$(X, -rK_X)$ via $D$, and let $\{(\tau_i, X_i)\}_{1\leq i\leq m}$ be the ample model 
sequence of $(X, -K_X; -D)$. 
We set $f(k):=\sum_{j=1}^{kr\tau(D)}h^0(X, \sO_X(-krK_X-jD))$ as in 
Proposition \ref{divst_prop}. By \cite[Remark 2.4 (i)]{KKL}, for any sufficiently 
divisible positive integer $k$, we have 
\[
f(k)=\sum_{i=1}^m\sum_{j=kr\tau_{i-1}+1}^{kr\tau_i}h^0\left(X_i, \sO_{X_i}(
-krK_{X_i}-jD_i)\right). 
\]
By Proposition \ref{RR_prop}, for any $1\leq i\leq m$, we have 
\begin{eqnarray*}
&&\sum_{j=kr\tau_{i-1}+1}^{kr\tau_i}h^0\left(X_i, \sO_{X_i}(-krK_{X_i}-jD_i)\right)\\
&=&\frac{k^{n+1}}{n!}\int_{r\tau_{i-1}}^{r\tau_i}((-rK_{X_i}-xD_i)^{\cdot n})dx\\
&+&\frac{k^n}{2\cdot(n-1)!}\int_{r\tau_{i-1}}^{r\tau_i}((-rK_{X_i}-xD_i)^{\cdot n-1}\cdot
(-K_{X_i}-D_i))dx+O(k^{n-1}). 
\end{eqnarray*}
Thus the $(n+1)$-th and $n$-th coefficients $f_{n+1}$ and $f_n$ of $f(k)$ are 
\begin{eqnarray*}
f_{n+1} &= & \sum_{i=1}^m\frac{r^{n+1}}{n!}\int_{\tau_{i-1}}^{\tau_i}
((-K_{X_i}-xD_i)^{\cdot n})dx,\\
f_n & = & \sum_{i=1}^m\frac{r^n}{2\cdot(n-1)!}\int_{\tau_{i-1}}^{\tau_i}
((-K_{X_i}-xD_i)^{\cdot n-1}\cdot(-K_{X_i}-D_i))dx.
\end{eqnarray*}
Therefore we have proved the following. 

\begin{thm}\label{eta_thm}
Let $X$ be a $\Q$-Fano variety of dimension $n$ and $D$ be a nonzero effective 
Weil divisor on $X$. Then 
\[
\eta(D)=\sum_{i=1}^mn\int_{\tau_{i-1}}^{\tau_i}(1-x)((-K_{X_i}-xD_i)^{\cdot n-1}
\cdot D_i)dx,
\]
where $\eta(D)$ is the value introduced in Proposition \ref{divst_prop}. 
In particular, divisorial stability and semistability of $(X, -K_X)$ along $D$ do not 
depend on the choice of the value $r$ in Definition \ref{divst_dfn}.
\end{thm}

\subsection{Via the volume functions}

\begin{thm}\label{volint_thm}
Let $X$, $D$ and $\eta(D)$ be as above. Then we have
\begin{eqnarray*}
\eta(D) &=& \vol_X(-K_X)-\int_0^{\tau(D)}\vol_X(-K_X-xD)dx\\
 &=& \vol_X(-K_X)-\int_0^\infty\vol_X(-K_X-xD)dx\\
 &=& \vol_X(-K_X)-n\int_0^{\tau(D)}x\vol_{X|D}(-K_X-xD)dx.
\end{eqnarray*}
\end{thm}

\begin{proof}
We note that $\eta(D)$ is equal to
\[
\sum_{i=1}^m\left\{\left[(x-1)((-K_{X_i}-xD_i)^{\cdot n})
\right]_{\tau_{i-1}}^{\tau_i}-\int_{\tau_{i-1}}^{\tau_i}((-K_{X_i}-xD_i)^{\cdot n})dx\right\}.
\]
By Lemma \ref{vol_lem}, this value is equal to
\[
\vol_X(-K_X)-\int_0^{\tau(D)}\vol_X(-K_X-xD)dx. 
\]
On the other hand, by Definition \ref{rvol_dfn} and Proposition \ref{rvol_prop}, this value 
is also equal to 
\begin{eqnarray*}
n\int_0^{\tau(D)}(1-x)\vol_{X|D}(-K_X-xD)dx
=\vol_X(-K_X)-n\int_0^{\tau(D)}x\vol_{X|D}(-K_X-xD)dx.
\end{eqnarray*}
Thus the assertion follows.
\end{proof}

\section{Toric case}\label{toric_section}

In this section, we see divisorial stability for toric $\Q$-Fano varieties. 
For the theory of toric varieties, we refer the readers to \cite{cox}. 
We fix a lattice $M:=\bigoplus_{i=1}^n\Z e_i$, a dual lattice 
$N:=\Hom_\Z(M, \Z)=\bigoplus_{i=1}^m\Z e_i^*$, and we set 
$M_\R:=M\otimes_\Z\R$ and $N_\R:=N\otimes_\Z\R$. 
We have a natural dual pairing $\langle, \rangle\colon M_\R\times N_\R\to\R$ with 
$\langle e_i, e^*_j\rangle=\delta_{ij}$. 
We fix a canonical Lebesgue measure $dx$ on $M_\R$, for which $M_\R/M$ is of 
measure $1$. 

Let $X$ be a toric $\Q$-Fano variety of dimension $n$ corresponds to a fan 
$\Sigma$ in $N_\R$. Let $\{v_\lambda\}_{\lambda\in\Lambda}$ be the set of 
the primitive generators of the one dimensional cones in $\Sigma$, let $D_\lambda$ be 
the torus invariant prime divisor on $X$ associated to the one dimensional cone 
$\R_{\geq 0}v_\lambda\in\Sigma$. We set the rational polytope $P\subset M_\R$ 
such that 
\[
P:=P_{(X, -K_X=\sum_{\lambda\in\Lambda}D_\lambda)}:=\{u\in M_\R\,\,|\,\,
\langle u, v_\lambda\rangle\geq -1\,\,(\forall\lambda\in\Lambda)\}
\]
as in \cite[(4.3.2)]{cox}. As is known in \cite[Proposition 3.2]{BB}, 
$P$ is a rational polytope which contains the origin in its interior. 
Let $b_P\in M_\R$ be the barycenter of $P$, that is, 
\[
b_P:=\frac{\int_Pxdx}{\int_Pdx}.
\]

\begin{thm}\label{etatoric_thm}
For any $\lambda\in\Lambda$, the signature of $\eta(D_\lambda)$ is equal to 
the signature of $-\langle b_P, v_\lambda\rangle$. 
\end{thm}

\begin{proof}
After a certain lattice transform, we can assume that $v_\lambda=e_1^*$. 

\begin{claim}\label{etatoric_claim}
For any $x\in[0, +\infty)$, we have
\[
\vol_X(-K_X-xD_\lambda)=n!\cdot\vol_{M_\R}(P|_{\nu_1\geq -1+x}), 
\]
where 
\begin{eqnarray*}
P|_{\nu_1\geq -1+x}&:=&\{u\in P\,\,|\,\,\langle u, e_1^*\rangle\geq -1+x\},\\
\vol_{M_\R}(P|_{\nu_1\geq -1+x})&:=&\int_{P|_{\nu_1\geq -1+x}}dx.
\end{eqnarray*}
\end{claim}

\begin{proof}[Proof of Claim \ref{etatoric_claim}]
It is enough to prove Claim \ref{etatoric_claim} for the case $x\in[0, +\infty)\cap\Q$ 
since both $\vol_X(-K_X-xD_\lambda)$ and $n!\cdot\vol_{M_\R}(P|_{\nu_1\geq -1+x})$
are continuous function over $x$. For a sufficiently divisible positive integer $k$, 
we have 
\[
H^0\left(X, \sO_X(-kK_X)\right)=\bigoplus_{u\in kP\cap M}\C\chi^u, 
\]
where $\chi^u$ is the character of the algebraic torus $(\G_m)^n$ defined  by $u\in M$ 
(see \cite[(1.1.1)]{cox}). Moreover, $H^0(X, \sO_X(-kK_X-kxD_\lambda))$ is equal to, 
as a subspace, 
\[
\bigoplus_{\substack{u\in kP\cap M\\ \langle u, e_1^*\rangle\geq -k+kx}}\C\chi^u
\]
(see \cite[Section 4.3]{cox}). Thus we have 
\[
h^0\left(X, \sO_X(-kK_X-kxD_\lambda)\right)=\#\{u\in k(P|_{\nu_1\geq -1+x})\cap M\}.
\]
Hence, by \cite[Proposition 2.1]{LM}, we have the assertion. 
\end{proof}

Let $Q(x)$ be the (restricted) volume of 
\[
P|_{\nu_1=-1+x}:=\{u\in P\,\,|\,\,\langle u, e_1^*\rangle=-1+x\}\subset\R^{n-1}.
\]
By Claim \ref{etatoric_claim}, we have 
\[
Q(x)=-\frac{1}{n!}\frac{d}{dx}\vol_X(-K_X-xD_\lambda)
=\frac{1}{(n-1)!}\vol_{X|D_\lambda}(-K_X-xD_\lambda)
\]
for any $x\in(0, \tau(D_\lambda))$. Thus we get the equation 
\[
\eta(D_\lambda)=n!\left(\vol_{M_\R}(P)-\int_0^{\tau(D_\lambda)}xQ(x)dx\right)
\]
from Theorem \ref{volint_thm}. On the other hand, we have 
\begin{eqnarray*}
\langle b_P, e_1^*\rangle=\frac{\int_P\langle x, e_1^*\rangle dx}
{\vol_{M_\R}(P)}=\frac{\int_0^{\tau(D_\lambda)}(-1+x)Q(x)dx}{\vol_{M_\R}(P)}
=\frac{\int_0^{\tau(D_\lambda)}xQ(x)dx}{\vol_{M_\R}(P)}-1.
\end{eqnarray*}
Thus we get the assertion. 
\end{proof}

\begin{corollary}\label{toric_cor}
Let $X$, $P$ and $b_P$ be as above. 
\begin{enumerate}
\renewcommand{\theenumi}{\arabic{enumi}}
\renewcommand{\labelenumi}{(\theenumi)}
\item\label{toric_cor1}
The pair $(X, -K_X)$ is not divisorially stable. 
\item\label{toric_cor2}
Assume that $b_P\neq 0$. Then there exists a torus invariant prime divisor $D$ on $X$ 
such that $(X, -K_X)$ is not divisorially semistable along $D$. 
\end{enumerate}
\end{corollary}

\begin{proof}
We consider the case $b_P\neq 0$. 
Let $c_P\in P$ be the intersection of the boundary $\partial P$ of $P$ 
and the half line 
\[
\{(1-t)b_P\,|\,t\in\R_{\geq 0}\}
\]
(cf.\ \cite{li}). 
Let $F_P$ be a facet (that is, $(n-1)$-dimensional face) of $P$ with $c_P\in F_P$, and 
let $\R_{\geq 0}v_\lambda$ be the one dimensional cone in $\Sigma$ associated to 
$F_P$. By construction, we have 
$\langle b_P, v_\lambda\rangle>0$. Thus $\eta(D_\lambda)<0$ 
by Theorem \ref{etatoric_thm}. 

If $b_P=0$, then $\langle b_P, v_\lambda\rangle=0$ for any $\lambda\in\Lambda$. 
Thus $\eta(D)=0$ for any torus invariant prime divisor $D$ on $X$. 
\end{proof}

\begin{remark}\label{toric_rmk}
The converse of Theorem \ref{introtoric_thm} follows from \cite[Theorem 1.2]{BB}. 
See also \cite{WZ, ZZ}. 
\end{remark}

As an immediate consequence of Corollary \ref{toric_cor}, for any 
non-K-semistable toric $\Q$-Fano variety, we can explicitly construct a flag ideal 
such that the Donaldson-Futaki invariant of the associated semi test configuration 
is strictly negative. In fact, the basic flag ideal of $(X, -rK_X; -D)$ for some $r\in\Z_{>0}$ 
and for some torus invariant $D$ is a desired flag ideal. 
We note that, for a toric $\Q$-Fano variety $X$ and a torus invariant prime divisor 
$D_\lambda$ on $X$, a positive integer $r\in\Z_{>0}$ satisfies the generating property 
with respects to $(X, -K_X; -D_\lambda)$ if and only if the $\C$-algebra
\[
\bigoplus_{k\geq 0}H^0\left(X, \sO_X(-krK_X)\right)
\]
is generated by $H^0(X, \sO_X(-rK_X))$. Indeed, as we have seen in the proof of 
Theorem \ref{etatoric_thm}, the space $H^0(X, \sO_X(-krK_X-jD_\lambda))$ is equal to 
\[
\bigoplus_{\substack{u\in krP\cap M\\ \langle u, v_\lambda\rangle\geq -kr+j}}\C\chi^u
\]
(see also \cite[Definition 2.2.9]{cox}). We see some examples. 

\begin{example}\label{toric1_ex}
Let $X$ be the blowup of $\pr^2$ along one point and let 
$E$ be the $(-1)$-curve on $X$. 
As we have seen in Corollary \ref{toric_cor}, $(X, -K_X)$ is not divisorially semistable 
along $E$. In fact, $\tau(E)=2$, $r=1$ satisfies the generating property 
with respects to $(X, -K_X; -E)$, and the basic flag ideal $\sI=I_2+I_1t+(t^2)$ 
with respects to $(X, -K_X; -E)$ satisfies that 
\begin{eqnarray*}
I_2&=&\sO_X(-2E),\\ 
I_1&=&\sO_X(-E).
\end{eqnarray*} 
In other words, $\sI$ is equal to the ideal sheaf $(\sO_X(-E)+(t))^2$. 
\end{example}

\begin{example}\label{toric2_ex}
Let $X$ be the blowup of $\pr^2$ along distinct two points, let $E_1$, $E_2$ be the 
distinct exceptional divisors of $X\to\pr^2$, and let $E_0$ be the strict transform 
of the line passing though the centers of the blowup. 
As we have seen in Corollary \ref{toric_cor}, $(X, -K_X)$ is not divisorially semistable 
along $E_0$. In fact, $\tau(E_0)=3$, $r=1$ satisfies the generating property 
with respects to $(X, -K_X; -E_0)$, and the basic flag ideal $\sI=I_3+I_2t+I_1t^2+(t^3)$ 
with respects to $(X, -K_X; -E_0)$ satisfies that 
\begin{eqnarray*}
I_3&=&\sO_X(-3E_0-2E_1-2E_2),\\
I_2&=&\sO_X(-2E_0-E_1-E_2),\\
I_1&=&\sO_X(-E_0).
\end{eqnarray*}
\end{example}

\section{On the barycenters of Okounkov bodies}\label{Okounkov_section}

We see the relation between the value $\eta(D)$ and the structure of Okounkov 
bodies of $-K_X$. For the theory of Okounkov bodies, we refer the readers to 
\cite{LM}. (See also \cite{okounkov, BC, KK, nystrom}.) 
In this section, we fix a $\Q$-Fano variety of dimension $n$, an admissible flag 
\[
Z_\bullet\colon X=Z_0\supset Z_1\supset\cdots\supset Z_n=\{\text{point}\}
\]
of $X$, that is, each $Z_i\subset X$ is a subvariety of dimension $n-i$ and 
each $Z_i$ is smooth around $Z_n$. We set $\Delta(-K_X):=\Delta_{Z_\bullet}(-K_X)
\subset\R^n$ the Okounkov body of $-K_X$ with respects to $Z_\bullet$ 
in the sense of \cite{LM}. 

\begin{thm}\label{Okounkov_thm}
Let $b_1$ be the first coordinate of the barycenter of the Okounkov body 
$\Delta(-K_X)$. The following are equivalent: 
\begin{enumerate}
\renewcommand{\theenumi}{\arabic{enumi}}
\renewcommand{\labelenumi}{(\theenumi)}
\item\label{Okounkov_thm1}
$(X, -K_X)$ is divisorially stable $($resp.\ divisorially semistable$)$ 
along $Z_1$. 
\item\label{Okounkov_thm2}
$b_1<1$ $($resp.\ $b_1\leq 1)$ holds. 
\end{enumerate}
\end{thm}

\begin{proof}
Let $\sigma\colon\tilde{X}\to X$ be a projective small $\Q$-factorial modification 
morphism. Since $Z_n$ is a smooth point of $X$, the morphism $\sigma$ is isomorphism 
around $Z_n$ (see \cite[Theorem VI.1.5]{kollar}). Hence we can consider the strict 
transform $\tilde{Z}_i$ of $Z_i$, and we get an admissible flag $\tilde{Z}_\bullet$ of 
$\tilde{X}$. By the construction of $\Delta(-K_X)$, the Okounkov body 
$\tilde{\Delta}:=\Delta_{\tilde{Z}_\bullet}(-K_{\tilde{X}})$ coincides with $\Delta(-K_X)$. 
By the proof of \cite[Corollary 4.25]{LM} 
(see also the proof of Proposition \ref{vr_prop}), 
\[
\vol_{\R^{n-1}}\left(\tilde{\Delta}|_{\nu_1=x}\right)=\frac{1}{(n-1)!}\vol_{X|Z_1}(-K_X-xZ_1)
\]
holds for any $x\in[0, \tau(Z_1))$. Hence we have 
\[
\eta(Z_1)=n!\left(\vol_{\R^n}\left(\tilde{\Delta}\right)-\int_0^{\tau(Z_1)}
x\vol_{\R^{n-1}}\left(\tilde{\Delta}|_{\nu_1=x}\right)dx\right).
\]
On the other hand, the value $b_1$ is equal to 
\[
\frac{\int_0^{\tau(Z_1)}x\vol_{\R^{n-1}}\left(\tilde{\Delta}|_{\nu_1=x}\right)dx}{\vol_{\R^n}
\left(\tilde{\Delta}\right)}.
\]
Thus the assertion follows. 
\end{proof}

\begin{remark}[{cf.\ \cite{BBvol}}]\label{Okounkov_rmk}
Let $(b_1,\dots,b_n)\in\R^n$ be the barycenter of $\Delta(-K_X)$. 
Assume that $(X, -K_X)$ is K-semistable. We may expect that the values 
$b_2,\dots,b_n$, 
in particular $\sum_{i=1}^nb_i$, are also small. However, it is not true in general. 
See the following example. 
\end{remark}

\begin{example}\label{Okounkov_ex}
Let $X:=\pr_{\pr^2}(T_{\pr^2})$. We know that $X$ is a Fano manifold of dimension three 
and is a rational homogeneous manifold. Thus $X$ admits K\"ahler-Einstein metrics. 
In particular, $(X, -K_X)$ is K-semistable. 
On the other hand, consider the admissible flag $Z_\bullet$ of $X$ such that 
\begin{itemize}
\item
$Z_1$ is the inverse image $\pi^{-1}(l)$ of a line $l\subset\pr^2$, where $\pi\colon X
\to \pr^2$ is the projection morphism (note that $Z_1\simeq\pr_{\pr^1}
(\sO\oplus\sO(1))$), 
\item
$Z_2$ is the $(-1)$-curve on $Z_1$, and
\item
$Z_3$ is a point on $Z_2$.
\end{itemize}
Then the Okounkov body $\Delta(-K_X)\subset\R^3$ 
of $-K_X$ with respects to $Z_\bullet$ is equal to 
\[
\{(\nu_1,\nu_2,\nu_3)\in\R^3\,\,|\,\,0\leq \nu_1,\nu_2\leq 2, \,0\leq 
\nu_3\leq 2-\nu_1+\nu_2\}.
\]
The barycenter of $\Delta(-K_X)$ is equal to $(5/6, 7/6, 7/6)$.
\end{example}

\section{MMP with scaling}\label{MMP_section}

In this section, we fix a $\Q$-Fano variety $X$ of dimension $n$, a nonzero effective 
Weil divisor $D$ on $X$, a projective small $\Q$-factorial modification 
$\sigma\colon\tilde{X}\to X$, and we set $\tilde{D}:=\sigma^{-1}_*D$. 
We show in this section that we can easily calculate the value $\eta(D)$ after we run 
a kind of MMP. 
More precisely, we run a $(-\tilde{D})$-MMP with scaling $-K_{\tilde{X}}$ 
(see \cite[Section 3.10]{BCHM}).
In other words, we consider the following program: 
\begin{itemize}
\item
Set $\mu_0:=+\infty$, $t_0:=\mu_0^{-1}=0$, $\tilde{X}_1:=\tilde{X}$ and 
$\tilde{D}_1:=\tilde{D}$. 
\item
Assume that we have constructed $\mu_{i-1}\in\R_{>0}\cup\{+\infty\}$, 
a $\Q$-factorial projective variety $\tilde{X}_i$ and a nonzero effective Weil divisor 
$\tilde{D}_i$ on $\tilde{X}_i$ such that $-\tilde{D}_i+\mu_{i-1}(-K_{\tilde{X}_i})$ is nef, 
that is, the value $t_{i-1}:=\mu_{i-1}^{-1}\in\R_{\geq 0}$ satisfies that 
$-K_{\tilde{X}_i}-t_{i-1}\tilde{D}_{i-1}$ is nef. 
Let 
\[
\mu_i:=\min\{\mu\in\R_{>0}\cup\{+\infty\}\,\,|\,\,-\tilde{D}_i+\mu(-K_{\tilde{X}_i})
\colon\text{ nef}\}.
\]
In other words, $t_i:=\mu_i^{-1}\in\R_{\geq 0}$ satisfies that 
\[
t_i=\max\{t\in\R_{\geq 0}\,\,|\,\,-K_{\tilde{X}_i}-t\tilde{D}_i\colon\text{ nef}\}.
\]
Since $\tilde{D}_i$ is nonzero effective, the values $\mu_i$ and $t_i$ can be defined 
and $t_{i-1}\leq t_i$. By \cite[Proposition 1.11 (1)]{HK} and the argument of 
\cite[Lemma 3.10.8]{BCHM}, the value $t_i$ is a rational number and there exists 
an extremal ray $R_i\subset\overline{\NE}(\tilde{X_i})$ such that 
$(-K_{\tilde{X}_i}-t_i\tilde{D}_i\cdot R_i)=0$ and $(\tilde{D}_i\cdot R_i)>0$. 
Moreover, we get the contraction morphism $\Phi_i\colon\tilde{X}_i\to\tilde{Y}_i$ 
of $R_i$. If $\Phi_i$ is of fiber type, then we set $m':=i$ and we stop the program. 
If $\Phi_i$ is divisorial, then we set $\tilde{X}_{i+1}:=\tilde{Y}_i$; if $\Phi_i$ is small, 
then let $\tilde{X}_{i+1}$ be the flip of $\Phi_i$. Let $\tilde{D}_{i+1}$ be the strict 
transform of $\tilde{D}_i$ and we continue the program. Since $(\tilde{D}_i\cdot R_i)>0$, 
the Weil divisor $\tilde{D}_{i+1}$ is nonzero effective. 
\item
By \cite[Proposition 1.11(1)]{HK}, after finitely many steps, we get a 
contraction morphism $\Phi_{m'}\colon\tilde{X}_{m'}\to\tilde{Y}_{m'}$ of fiber type. 
\end{itemize}

Thus we get the following datum: 
\begin{itemize}
\item
birational contraction maps 
\[
\tilde{X}=\tilde{X}_1\dashrightarrow\tilde{X}_2\dashrightarrow\cdots\dashrightarrow
\tilde{X}_{m'}
\]
among $\Q$-factorial projective varieties, 
\item
the strict transform $\tilde{D}_i$ of $\tilde{D}$ on $\tilde{X}_i$ such that all of them are 
nonzero effective, 
\item
a non-decreasing sequence
\[
0=t_0\leq t_1\leq\cdots\leq t_{m'}
\]
of rational numbers, and 
\item
a contraction morphism $\Phi_{m'}\colon\tilde{X}_{m'}\to\tilde{Y}_{m'}$ of fiber type 
of an extremal ray $R_{m'}\subset\overline{\NE}(\tilde{X}_{m'})$ such that 
$(-K_{\tilde{X}_{m'}}-t_{m'}\tilde{D}_{m'}\cdot R_{m'})=0$.
\end{itemize}

\begin{proposition}\label{mmp_prop}
\begin{enumerate}
\renewcommand{\theenumi}{\arabic{enumi}}
\renewcommand{\labelenumi}{(\theenumi)}
\item\label{mmp_prop1}
Let $\{(\tau_i, X_i)\}_{1\leq i\leq m}$ be the ample model sequence of $(X, -K_X; -D)$. 
Then we have 
\[
\{\tau_0,\dots,\tau_m\}=\{t_0,\dots,t_{m'}\},
\]
where $\tau_0:=0$. In particular, we have $t_{m'}=\tau(D)$. 
\item\label{mmp_prop2}
We have 
\[
\eta(D)=
\sum_{i=1}^{m'}n\int_{t_{i-1}}^{t_i}(1-x)\left(\left(-K_{\tilde{X}_i}-x\tilde{D}_i
\right)^{\cdot n-1}\cdot\tilde{D}_i\right)dx.
\]
\end{enumerate}
\end{proposition}

\begin{proof}
\eqref{mmp_prop1}
For any $1\leq i\leq m'$ with $t_{i-1}<t_i$, there exist a $1\leq j_i\leq m$ and a 
birational morphism $\Psi_i\colon\tilde{X}_i\to X_{j_i}$
such that $X_{j_i}$ is the ample model of $-K_{\tilde{X}_i}-t\tilde{D}_i$ 
(hence also the ample model of $-K_{\tilde{X}}-t\tilde{D}$)
for any $t\in(t_{i-1}, t_i)$. Moreover, the value $j_i$ and the morphism $\Psi_i$ are 
uniquely determined. Thus we have $(t_{i-1}, t_i)\subset(\tau_{j_i-1}, \tau_{j_i})$. 
Assume that there exists an $1\leq i'\leq m'$ such that $i'>i$ and $j_{i'}=j_i$. 
Then $\Psi_i^*(-K_{X_{j_i}}-tD_{j_i})=-K_{\tilde{X}_i}-t\tilde{D}_i$ is nef for any 
$t\in(t_{i'-1}, t_{i'})$. This contradicts to the construction of $t_i$. 
Thus $j_i\neq j_{i'}$ if $i\neq i'$. Since $-K_{\tilde{X}_{m'}}-t_{m'}\tilde{D}_{m'}$ is 
pseudo-effective but is not big, we have $t_{m'}=\tau(D)$. 
Thus we have proved \eqref{mmp_prop1}. 

\eqref{mmp_prop2}
The right-hand side is equal to 
\[
\sum_{\substack{1\leq i\leq m'\\ \text{with }t_{i-1}<t_i}}n\int_{t_{i-1}}^{t_i}(1-x)
\left(\left(-K_{\tilde{X}_i}-x\tilde{D}_i\right)^{\cdot n-1}\cdot \tilde{D}_i\right)dx.
\]
On the other hand, we have 
\begin{eqnarray*}
\left(\left(-K_{\tilde{X}_i}-x\tilde{D}_i\right)^{\cdot n-1}\cdot \tilde{D}_i\right)
=\left(\Psi_i^*\left(-K_{X_{j_i}}-xD_{j_i}\right)^{\cdot n-1}\cdot \tilde{D}_i\right)
=\left(\left(-K_{X_{j_i}}-xD_{j_i}\right)^{\cdot n-1}\cdot D_{j_i}\right).
\end{eqnarray*}
Thus we have proved \eqref{mmp_prop2}. 
\end{proof}

\begin{corollary}\label{slope_cor}
\begin{enumerate}
\renewcommand{\theenumi}{\arabic{enumi}}
\renewcommand{\labelenumi}{(\theenumi)}
\item\label{slope_cor1}
Assume that $X$ is $\Q$-factorial. Then the value $\tau_1$ is equal to the value 
$\varepsilon(D, (X, -K_X))$, where 
\[
\varepsilon(D, (X, -K_X)):=\max\{t\in\R_{>0}\,\,|\,\,-K_X-tD\,\text{\rm{ nef}}\}.
\]
\item\label{slope_cor2}
Assume that $X$ is smooth. 
\begin{enumerate}
\renewcommand{\theenumii}{\roman{enumii}}
\renewcommand{\labelenumii}{(\theenumii)}
\item\label{slope_cor21}
If $m=1$, then $\eta(D)=\xi(D)$, where 
\[
\xi(D):=n\int_0^{\varepsilon(D, (X, -K_X))}(1-x)\left((-K_X-xD)^{\cdot n-1}\cdot D
\right)dx
\]
$($see \cite[Proposition 3.2]{fjt}$)$. In particular, if $\Nef(X)=\overline{\Eff}(X)$, then 
$\eta(D)=\xi(D)$ for any nonzero effective divisor $D$ on $X$. 
\item\label{slope_cor22}
If $(X, -K_X)$ is divisorially stable $($resp.\ divisorially semistable$)$ along $D$, then 
$(X, -K_X)$ is slope stable $($resp.\ slope semistable$)$ along $D$. 
$($For the theory of slope stability, we refer the readers to \cite{RT, fjt}.$)$
\end{enumerate}
\end{enumerate}
\end{corollary}

\begin{proof}
\eqref{slope_cor1}
As we have already seen in the beginning of Section \ref{MMP_section}, 
$t_1=\varepsilon(D, (X, -K_X))$ holds. In particular, $t_1$ is a positive rational number. 
By Proposition \ref{mmp_prop} \eqref{mmp_prop1}, we have $t_1=\tau_1$. 

\eqref{slope_cor21}
If $m=1$, then $\tau(D)=\tau_1=\varepsilon(D, (X, -K_X))$. Thus $\eta(D)=\xi(D)$. 

\eqref{slope_cor22}
Assume that $\xi(D)\leq 0$. Then $t_1>1$ by \cite[Remark 3.4]{fjt}. Thus 
\[
\eta(D)=\xi(D)+\sum_{i=2}^{m'}n\int_{t_{i-1}}^{t_i}(1-x)
\left(\left(-K_{\tilde{X}_i}-x\tilde{D}_i\right)^{\cdot n-1}\cdot\tilde{D}_i\right)dx
\leq\xi(D)
\]
holds.
\end{proof}

\section{Basic properties}\label{basic_section}

In this section, we fix a $\Q$-Fano variety $X$ of dimension $n$ and 
a nonzero effective Weil divisor $D$ on $X$. 

\subsection{Proportional case}\label{proport_section}

\begin{lemma}\label{proport_lem}
Assume that a nonzero effective Weil divisor $D'$ on $X$ satisfies that 
$D'\sim_\Q cD$ for some $c\in\Q_{>0}$. 
\begin{enumerate}
\renewcommand{\theenumi}{\arabic{enumi}}
\renewcommand{\labelenumi}{(\theenumi)}
\item\label{proport_lem1}
Assume that $c=1$. Then $\eta(D)=\eta(D')$. 
\item\label{proport_lem2}
Assume that $c>1$. If $(X, -K_X)$ is divisorially semistable along $D$, then 
$(X, -K_X)$ is divisorially stable along $D'$. 
\end{enumerate}
\end{lemma}

\begin{proof}
This is obvious from the equations
\begin{eqnarray*}
\eta(D')&=&\vol_X(-K_X)-\int_0^{\infty}\vol_X(-K_X-cxD)dx\\
&=&\vol_X(-K_X)-\frac{1}{c}\int_0^{\infty}\vol_X(-K_X-xD)dx.
\end{eqnarray*}
\end{proof}

\begin{lemma}\label{rhoone_lem}
Assume that $-K_X\sim_\Q cD$ for some $c\in\Q_{>0}$. Then 
$\eta(D)>0$ $($resp.\ $\geq 0)$ holds if and only if $c<n+1$ $($resp.\ $\leq n+1)$
holds.
\end{lemma}

\begin{proof}
Let $\{(\tau_i, X_i)\}_{1\leq i\leq m}$ be the ample model sequence. 
Then $m=1$ and $\tau_1=c$. Hence we have 
$\eta(D)=((-K_X)^{\cdot n})\cdot(n+1-c)/(n+1)$.
\end{proof}

\begin{corollary}\label{rhoone_cor}
Assume that $X$ is smooth and $\rho(X)=1$. 
\begin{enumerate}
\renewcommand{\theenumi}{\arabic{enumi}}
\renewcommand{\labelenumi}{(\theenumi)}
\item\label{rhoone_cor1}
If $X\simeq\pr^n$, then $(X, -K_X)$ is divisorially semistable and is not 
divisorially stable along a hyperplane. 
\item\label{rhoone_cor2}
If $X\not\simeq\pr^n$, then $(X, -K_X)$ is divisorially stable. 
\end{enumerate}
\end{corollary}

\begin{proof}
Follows immediately from Lemma \ref{rhoone_lem} and \cite{KO}. 
See also \cite[Corollary 4.7]{fjt}. 
\end{proof}

\begin{remark}\label{divvsK_rmk}
There exists a Fano manifold $X$ of dimension three such that $\rho(X)=1$ but 
$X$ does not admit K\"ahler-Einstein metrics \cite{tian1}. 
Thus, from Corollary \ref{rhoone_cor}, divisorial stability is strictly weaker than 
K-stability. Recently, the author found Fano manifolds $X$ with $\rho(X)=1$ such that 
the pairs $(X, -K_X)$ are not K-semistable (the examples will appear elsewhere). 
Thus divisorial semistablility is also strictly weaker than K-semistability. 
\end{remark}

\subsection{Convexity of the volume functions}\label{convvol_section}

\begin{lemma}\label{big_lem}
If $\eta(D)\leq 0$, then $\tau(D)>1$. 
In particular, if $D$ is $\Q$-Cartier and $\eta(D)\leq 0$, then 
$-K_X-D$ is big.
\end{lemma}

\begin{proof}
Follows immediately since 
$((-K_{X_i}-xD_i)^{\cdot n-1}\cdot D_i)>0$ holds for any $x\in(\tau_{i-1}, \tau_i)$. 
\end{proof}

\begin{remark}\label{big_rmk}
Assume that $X$ is smooth. 
It is enough 
to check the signature of $\eta(D)$ for only a \emph{finite} number of divisors $D$ 
on $X$ for testing divisorial (semi)stability. 
This follows from Lemmas \ref{big_lem} and \ref{proport_lem} \eqref{proport_lem1} 
and the fact 
\[
\#\overline{\Eff}(X)\cap([-K_X]-\BIG(X))\cap\Pic(X)<+\infty, 
\]
where $[-K_X]$ is the class of $-K_X$ in $\ND(X)$. 
\end{remark}

\begin{lemma}[{cf.\ \cite[Proposition 4.1]{fjt}}]\label{conv_lem}
Let $\{(\tau_i, X_i)\}_{1\leq i\leq m}$ be the ample model sequence of $(X, -K_X; -D)$ 
and $D_i\subset X_i$ be the strict transform of $D$. 
Assume that the following:
\begin{itemize}
\item
for any $1\leq i\leq m$, there exists a positive integer $h_i\in\Z_{>0}$ such that 
$h_iD_i$ is Cartier and $H^0(D_i, \sO_{X_i}(h_iD_i)|_{D_i})\neq 0$
$($resp.\ $H^0(D_i, \sO_{X_i}(-h_iD_i)|_{D_i})\neq 0)$, 
and
\item
$\tau(D)\leq 2$ $($resp.\ $\geq 2)$. 
\end{itemize}
Then $\eta(D)\geq 0$ $($resp.\ $\leq 0)$. 
If we further assume one of 
\begin{itemize}
\item
there exists an $1\leq i\leq m$ such that $\sO_{X_i}(h_iD_i)|_{D_i}$ is not numerically 
trivial on $D_i$, or
\item
$\tau(D)<2$ $($resp.\ $>2)$, 
\end{itemize}
then $\eta(D)>0$ $($resp.\ $<0)$. 
\end{lemma}

\begin{proof}
The function $\vol_X(-K_X-xD)$ over $x\in[0, \tau(D)]$ is $\sC^1$ and monotone 
decreasing function by Definition \ref{rvol_dfn} and Proposition \ref{rvol_prop}. 
Moreover, by Lemma \ref{vol_lem}, 
\[
\frac{d^2}{dx^2}\vol_X(-K_X-xD)=(n-1)\left(\left(-K_{X_i}-xD_i\right)^{\cdot n-2}
\cdot D_i^{\cdot 2}\right)
\]
for any $x\in(\tau_{i-1}, \tau_i)$. Thus the assertion follows from the convexity 
property of the 
function $\vol_X(-K_X-xD)$ (see also \cite[Proof of Proposition 4.1]{fjt}). 
\end{proof}

\subsection{Basic examples}\label{basicex_section}

\begin{proposition}\label{surf_prop}
Assume that $X$ is smooth and $n\leq 2$. Then 
$(X, -K_X)$ is K-stable $($resp.\ K-semistable$)$ if and only if 
$(X, -K_X)$ is divisorially stable $($resp.\ divisorially semistable$)$. 
\end{proposition}

\begin{proof}
The case $n=1$ is trivial. We consider the case $n=2$. If $X$ is toric, then the assertion 
follows from Corollary \ref{toric_cor} and \cite[Theorem 1.2]{BB}. 
Assume that $X$ is not toric. Then $X$ is obtained by the blowup of $\pr^2$ along 
distinct general points $p_1,\dots,p_k$ with $4\leq k\leq 8$. In this case, 
$\Aut(X)$ is finite and $X$ admits K\"ahler-Einstein metrics by \cite{tian}. 
Hence $(X, -K_X)$ is K-stable by \cite[Theorem 1.2]{stoppa}. 
By Remark \ref{divst_rmk}, $(X, -K_X)$ is divisorially stable. 
\end{proof}

\begin{lemma}\label{rhotwo_lem}
Let $Y$ be a Fano manifold of dimension $n\geq 3$ such that $\Pic(Y)=\Z[\sO_Y(1)]$. 
Assume that $r\in\Z_{>0}$ satisfies that $\sO_Y(-K_Y)\simeq\sO_Y(r)$. 
Let $1\leq d_1<d_2\leq r-1$ and let $D_1$, $D_2\subset Y$ be divisors such that 
$\sO_Y(D_i)\simeq\sO_Y(d_i)$ $(i=1, 2)$ and assume that 
the scheme theoretic intersection 
$C:=D_1\cap D_2\subset Y$ is a smooth codimension two subvariety. 
Let $X$ be the blowup along $C$. Then $X$ is a Fano manifold. 
If $d_2\geq 2d_1$, then $(X, -K_X)$ is not divisorially semistable along the strict 
transform of $D_1$. 
\end{lemma}

\begin{proof}
Since $\overline{\NE}(X)$ is spanned by the class of a curve contracted by the 
morphism $X\to Y$ and a curve in the strict transform $\hat{D}_1$ of $D_1$, 
$X$ is a Fano manifold. We consider divisorial stability along $\hat{D}_1$. 
In this case, we have $\tau_1=1$, $X_2=Y$, $\tau_2=r/d_2$ and $m=2$. Hence we get 
\begin{eqnarray*}
\eta(\hat{D}_1)=\frac{(\sO_Y(1)^{\cdot n})}{(n+1)d_1(d_2-d_1)^2}
\Bigl\{-\left((d_2-d_1)^2-d_1^2\right)(r-d_1)^{n+1}\\
-\left((n+1)(d_2-d_1)+r-d_2\right)(r-d_2)^n\Bigr\}.
\end{eqnarray*}
Therefore, if $d_2-d_1\geq d_1$, then $\eta(\hat{D}_1)<0$. 
\end{proof}

\begin{lemma}\label{rhothree_lem}
Let $Z$ be a Fano manifold of dimension $n-1$ with $n\geq 3$ and 
$\Pic(Z)=\Z[\sO_Z(1)]$. Assume that $r\in\Z_{>0}$ satisfies that $\sO_Z(-K_Z)\simeq
\sO_Z(r)$. We set $Y:=\pr_Z(\sO_Z\oplus\sO_Z(s))$, where $r>s>0$. 
Let $E\subset Y$ be a section of the $\pr^1$-bundle with $\sN_{E/Y}\simeq\sO_Z(s)$ 
and let $W\subset E$ be a smooth divisor on $E$ with $\sO_E(W)\simeq\sO_Z(d)$ 
such that $r>d-s$. 
Let $X$ be the blowup of $Y$ along $W$ and let $E'\subset Y$ be the strict transform 
of the negative section of $Y\to Z$. Then $X$ is a Fano manifold. Moreover, if 
$d=2s$, then $\eta(E')=0$; if $d<2s$, then $\eta(E')<0$. 
\end{lemma}

\begin{proof}
By \cite[Lemma 2.5]{fuj} or \cite[Remark 3.4.5]{casdru}, $X$ is a Fano manifold. 
We consider divisorial stability along $E'$. In this case, we have $\tau_1=1$, 
$X_2=\pr_Z(\sO_Z\oplus\sO_Z(d-s))$, $\sN_{E'_2/X_2}\simeq\sO_Z(d-s)$, where $E'_2$ 
is the image of $E'$ on $X_2$, and $\tau_2=2$, $m=2$. 
If $d\leq s$, then $\eta(E')<0$ by Lemma \ref{conv_lem}. From now on, we assume that 
$s<d\leq 2s$. Set $s':=d-s$. Then $0<s'\leq s<r$ holds. Moreover, we have 
\[
\eta(E')=\frac{(\sO_Z(1)^{\cdot n-1})}{n+1}\left(g_1(s)-g_1(s')\right), 
\]
where 
\[
g_1(x):=\frac{1}{x^2}\left(r^{n+1}-(nx+r)(r-x)^n\right).
\]
Thus it is enough to show that the function $g_1(x)$ is strictly monotone decreasing 
function over $x\in(0, r)$. Since
$g'_1(x)=g_2(x)/x^3$, where 
\[
g_2(x):=(r-x)^{n-1}\left(n(n-1)x^2+2(n-1)rx+2r^2\right)-2r^{n+1}, 
\]
and $g_2(0)=0$, it is enough to show that the function $g_2(x)$ is 
strictly monotone decreasing function over $x\in(0, r)$. Since 
\[
g'_2(x)=-n(n-1)(n+1)(r-x)^{n-2}x^2<0, 
\]
the assertion follows. 
\end{proof}

\section{Three-dimensional case}\label{three_section}

In this section, we consider divisorial stability for Fano manifolds of dimension three. 
The goal of this section is to prove the following theorem: 

\begin{thm}\label{three_thm}
Let $X$ be a Fano manifold of dimension three. 
\begin{enumerate}
\renewcommand{\theenumi}{\arabic{enumi}}
\renewcommand{\labelenumi}{(\theenumi)}
\item\label{three_thm1}
The pair $(X, -K_X)$ is divisorially semistable but not divisorially stable if and only if 
one of the following satisfied: 
\begin{itemize}
\item
$X\simeq\pr^3$, 
\item
$\rho(X)=2$ and $X$ belongs to either No.\ 26 or No.\ 34 in \cite[Table 2]{MoMu}, 
\item
$\rho(X)=3$ and $X$ belongs to one of No.\ 9, No.\ 19 or No.\ 25 
in \cite[Table 3]{MoMu}, 
\item
$\rho(X)=4$ and $X$ belongs to one of No.\ 2, No.\ 4 or No.\ 7 
in \cite[Table 4]{MoMu}, 
\item
$\rho(X)=5$ and $X$ belongs to No.\ 3 in \cite[Table 5]{MoMu}, or
\item
$\rho(X)\geq 6$.
\end{itemize}
\item\label{three_thm2}
The pair $(X, -K_X)$ is not divisorially semistable if and only if 
one of the following satisfied: 
\begin{itemize}
\item
$\rho(X)=2$ and $X$ belongs to one of No.\ 23, No.\ 28, No.\ 30, No.\ 31, No.\ 33, 
No.\ 35 or No.\ 36 in \cite[Table 2]{MoMu}, 
\item
$\rho(X)=3$ and $X$ belongs to one of No.\ 14, No.\ 16, No.\ 18, No.\ 21, No.\ 22, 
No.\ 23, No.\ 24, No.\ 26, No.\ 28, No.\ 29, No.\ 30 or No.\ 31 in \cite[Table 3]{MoMu}, 
\item
$\rho(X)=4$ and $X$ belongs to one of No.\ 5, No.\ 8, No.\ 9, No.\ 10, No.\ 11 
or No.\ 12 in \cite[Table 4]{MoMu}, or 
\item
$\rho(X)=5$ and $X$ belongs to No.\ 2 in \cite[Table 5]{MoMu}. 
\end{itemize}
\end{enumerate}
\end{thm}

We prove Theorem \ref{three_thm}. 
We fix the notation. Let $X$ be a Fano manifold of dimension three, let $D$ be a 
nonzero effective divisor on $X$ and let $\{(\tau_i, X_i)\}_{1\leq i\leq m}$ be the 
ample model sequence of $(X, -K_X; -D)$. 
We set 
\[
\eta_i:=\int_{\tau_{i-1}}^{\tau_i}(1-x)
\left(\left(-K_{X_i}-xD_i\right)^{\cdot 2}\cdot D_i\right)dx
\]
for simplicity. By definition, $\eta(D)/3=\eta_1+\cdots+\eta_m$. 
Let $V_d$ be a del Pezzo manifold of degree $d$ $(1\leq d\leq 5)$ and 
of dimension three (see \cite[I, Section 8]{fujita}). Let 
$V_7:=\pr_{\pr^2}(\sO\oplus\sO(1))$, $W_6:=\pr_{\pr^2}(T_{\pr^2})$ and $Q$ be 
a smooth hyperquadric in $\pr^4$. For $d\in\Z_{>0}$, we set 
$\F_d:=\pr_{\pr^1}(\sO\oplus\sO(d))$, let $\sigma_d\subset\F_d$ be the $(-d)$-curve 
and let $f_d\subset\F_d$ be a fiber of the morphism $\F_d\to\pr^1$. 
Moreover, on $\pr_{\pr^s}(\sO\oplus\sO(a_1)\oplus\cdots\oplus\sO(a_k))$, 
let $\xi_\pr$ be a tautological line bundle and $H_{\pr^s}$ be a pullback 
of $\sO_{\pr^s}(1)$, if there is no confusion. 

We can assume that $\rho(X)\geq 2$ by Corollary \ref{rhoone_cor}. 
Let $\{l_1,\dots,l_k\}$ be the set of minimal extremal rational curves on $X$ as in 
\cite[Section III-3]{mat}. We note that the nef cone $\Nef(X)$ of $X$ is the dual cone 
of $\overline{\NE}(X)$ and the pseudo-effective cone $\overline{\Eff}(X)$ of $X$ 
is equal to 
\[
\Nef(X)+\sum_{E\in\sE}\R_{\geq 0}[E], 
\]
where $\sE$ is the set of the exceptional divisors of all elementary divisorial contraction 
of $X$ (see \cite[Proposition 1.2]{bar}). Thus we can easily calculate those cones 
from \cite[Section III-3]{mat}. 

\begin{definition}\label{suspicious_dfn}
The divisor $D$ is said to be a \emph{suspicious divisor} if 
the following conditions are satisfied: 
\begin{itemize}
\item
$-K_X-D$ is big, 
\item
the class $[D]\in\Pic(X)$ of $D$ is primitive in $\Pic(X)$ and is not proportional to 
$-K_X$, 
\item
$m\geq 2$, and
\item
if $\rho(X)=2$ and $D$ is nef, then $\tau(D)>2$. 
\end{itemize}
\end{definition}

Assume that $X$ is neither a toric, the product of $\pr^1$ and a del Pezzo surface, nor the blowup of $Q$ along a line. 
Furthermore, if 
$\eta(D)\leq 0$, then $D$ must be a suspicious divisor by Corollary \ref{slope_cor} 
\eqref{slope_cor22}, \cite[Theorem 1.5]{fjt}, 
Lemmas \ref{proport_lem}, 
\ref{rhoone_lem}, \ref{big_lem} and \ref{conv_lem}. 

The following lemma is obvious. 

\begin{lemma}\label{three_lem}
Assume that $m=2$, $\tau_1=1$, $X_2$ is smooth and $D_2\simeq D$. 
\begin{enumerate}
\renewcommand{\theenumi}{\arabic{enumi}}
\renewcommand{\labelenumi}{(\theenumi)}
\item\label{three_lem1}
We consider the case $D\simeq\pr^1\times\pr^1$. 
\begin{enumerate}
\renewcommand{\theenumii}{\roman{enumii}}
\renewcommand{\labelenumii}{(\theenumii)}
\item\label{three_lem11}
If $\tau_2=3/2$, $\sN_{D/X}\simeq\sO_{\pr^1\times\pr^1}(-1, a)$ and 
$\sN_{D_2/X_2}\simeq\sO_{\pr^1\times\pr^1}(0, 4)$ with $a\geq -1$, then $\eta(D)>0$. 
\item\label{three_lem12}
If $\tau_2=2$, $\sN_{D/X}\simeq\sO_{\pr^1\times\pr^1}(-1, a)$ and 
$\sN_{D_2/X_2}\simeq\sO_{\pr^1\times\pr^1}(2, b)$ with $a$, $b\geq -1$, 
then $\eta(D)>0$. 
\item\label{three_lem13}
If $\tau_2\geq 2$, $\sN_{D/X}\simeq\sO_{\pr^1\times\pr^1}(-1, a)$ and 
$\sN_{D_2/X_2}\simeq\sO_{\pr^1\times\pr^1}(0, b)$ with $a\leq 0$ and $b\leq 1$, 
then $\eta(D)\leq0$. Moreover, the equality holds if and only if $(\tau_2, a, b)=(2, 0, 1)$. 
\item\label{three_lem14}
If $\tau_2=3$, $\sN_{D/X}\simeq\sO_{\pr^1\times\pr^1}(-1, 1)$ and 
$\sN_{D_2/X_2}\simeq\sO_{\pr^1\times\pr^1}(1, 1)$, then $\eta(D)>0$. 
\end{enumerate}
\item\label{three_lem2}
We consider the case $D\simeq\F_1$. 
\begin{enumerate}
\renewcommand{\theenumii}{\roman{enumii}}
\renewcommand{\labelenumii}{(\theenumii)}
\item\label{three_lem21}
If $\tau_2=2$, $\sN_{D/X}\simeq\sO_{\F_1}(-\sigma_1+af_1)$ and 
$\sN_{D_2/X_2}\simeq\sO_{\F_1}(\sigma_1+2f_1)$ with $a\geq -1$, then $\eta(D)>0$. 
\item\label{three_lem22}
If $\tau_2\geq 2$, $\sN_{D/X}\simeq\sO_{\F_1}(-\sigma_1-f_1)$ and 
$\sN_{D_2/X_2}\simeq\sO_{\F_1}(a\sigma_1+f_1)$ with $-1\leq a\leq 1$, 
then $\eta(D)\leq 0$. Moreover, the equality holds if and only if $(\tau_2, a)=(2, 1)$. 
\item\label{three_lem23}
If $\tau_2=2$, $\sN_{D/X}\simeq\sO_{\F_1}(-\sigma_1)$ and 
$\sN_{D_2/X_2}\simeq\sO_{\F_1}(f_1)$, then $\eta(D)>0$. 
\end{enumerate}
\end{enumerate}
\end{lemma}

\subsection{The case $\rho(X)=2$}\label{rho2_section}

We consider the case $\rho(X)=2$. We prepare the following lemma. The proof 
is straightforward. 

\begin{lemma}\label{three-two_lem}
Assume that $\rho(X)=2$, $X_2$ is smooth with $\Pic(X_2)=\Z[\sO_{X_2}(H)]$, 
$-K_{X_2}\sim rH$ for some $r\in\Z_{>0}$, and there exists a morphism 
$\psi\colon X\to X_2$ which is obtained by the blowup of a smooth curve 
$C\subset X_2$ of degree $d$ and genus $g$. Let $F$ be the exceptional divisor of 
$\psi$ and let $e$, $h\in\Z_{>0}$ with $D+hF=\psi^*D_2$ and $D_2\sim eH$. 
Then we have the equality 
\begin{eqnarray*}
\frac{1}{3}\eta(D)&=&(H^{\cdot 3})\int_0^{\tau_2}e(1-x)(r-ex)^2dx\\
&+&\int_0^{\tau_1}(1-x)\left\{-(1-hx)(hr+e+h(hr-3e)x)d+(2g-2)h(1-hx)^2\right\}dx. 
\end{eqnarray*}
\end{lemma}

Assume that $X$ belongs to No.\ 33--36 in \cite[Table 2]{MoMu}. Then $X$ is toric. 
Then $(X, -K_X)$ is not divisorially stable by Corollary \ref{toric_cor}. 
Moreover, $(X, -K_X)$ is divisorially semistable if and only if $X$ belongs to 
No.\ 34 in \cite[Table 2]{MoMu} by Theorem \ref{introtoric_thm} and \cite{mab}. 
Assume that $X$ belongs to No.\ 31 in \cite[Table 2]{MoMu}. Then $(X, -K_X)$ is not 
divisorially semistable by Corollary \ref{slope_cor} \eqref{slope_cor22} and 
\cite[Theorem 1.5]{fjt}. 
Assume that $X$ belongs to one of No.\ 23, No.\ 28 or No.\ 30 in \cite[Table 2]{MoMu}. 
Then $(X, -K_X)$ is not divisorially semistable by Lemma \ref{rhotwo_lem}. 

We assume that $X$ belongs to neither No.\ 23, No.\ 28, No.\ 30, No.\ 31, 
nor No.\ 33--36. We also assume the existence of a suspicious divisor $D$. 
By \cite[Theorem 5.1]{MoMu83}, the possibility of $X$ and $D$ is one of the following: 
\begin{enumerate}
\renewcommand{\theenumii}{\roman{enumii}}
\renewcommand{\labelenumii}{(\theenumii)}
\item\label{rho2-1}
$X$ belongs to No.\ 15 in \cite[Table 2]{MoMu} and $D$ is the strict transform 
of $A$.
\item\label{rho2-2}
$X$ belongs to No.\ 19 in \cite[Table 2]{MoMu} and $D$ is the exceptional divisor of 
the morphism $X\to V_4$. 
\item\label{rho2-3}
$X$ belongs to No.\ 22 in \cite[Table 2]{MoMu} and $D$ is the exceptional divisor of 
the morphism $X\to V_5$. 
\item\label{rho2-4}
$X$ belongs to No.\ 26 in \cite[Table 2]{MoMu} and $D$ is the exceptional divisor of 
the morphism $X\to V_5$. 
\item\label{rho2-5}
$X$ belongs to No.\ 26 in \cite[Table 2]{MoMu} and $D$ is the exceptional divisor of 
the morphism $X\to Q$. 
\item\label{rho2-6}
$X$ belongs to No.\ 29 in \cite[Table 2]{MoMu} and 
$\sO_X(D)\simeq\cont_{l_2}^*\sO_{\pr^1}(1)$, 
where $\cont_{l_2}$ is the contraction morphism associated to the extremal ray 
$\R_{\geq 0}[l_2]$ (\cite{mat}).
\end{enumerate}

\eqref{rho2-1}
In this case, under the notation in Lemma \ref{three-two_lem}, 
$\tau_1=1$, $\tau_2=2$, $X_2=\pr^3$, $r=4$, $h=1$, $e=2$, $d=6$ and $g=4$. 
Hence $\eta(D)/3=7/6>0$. 

\eqref{rho2-2}
In this case, under the notation in Lemma \ref{three-two_lem}, 
$\tau_1=1$, $\tau_2=2$, $X_2=\pr^3$, $r=4$, $h=1$, $e=2$, $d=5$ and $g=2$ 
(see also \cite[p.\ 117]{MoMu83}). 
Hence $\eta(D)/3=2>0$. 

\eqref{rho2-3}
In this case, under the notation in Lemma \ref{three-two_lem}, 
$\tau_1=1$, $\tau_2=2$, $X_2=\pr^3$, $r=4$, $h=1$, $e=2$, $d=4$ and $g=0$ 
(see also \cite[p.\ 117]{MoMu83}). 
Hence $\eta(D)/3=17/6>0$. 

\eqref{rho2-4}
In this case, under the notation in Lemma \ref{three-two_lem}, 
$\tau_1=1$, $\tau_2=3$, $X_2=Q$, $r=3$, $h=1$, $e=1$, $d=3$ and $g=0$ 
(see also \cite[p.\ 117]{MoMu83}). Hence $\eta(D)/3=0$. 

\eqref{rho2-5}
In this case, under the notation in Lemma \ref{three-two_lem}, 
$\tau_1=1/2$, $\tau_2=2$, $X_2=V_5$, $r=2$, $h=2$, $e=1$, $d=1$ and $g=0$. 
Hence $\eta(D)/3=239/48>0$. 

\eqref{rho2-6}
In this case, under the notation in Lemma \ref{three-two_lem}, 
$\tau_1=1$, $\tau_2=3$, $X_2=Q$, $r=3$, $h=1$, $e=1$, $d=2$ and $g=0$. 
Hence $\eta(D)/3=4/3>0$. 

Therefore we have proved Theorem \ref{three_thm} for the case $\rho(X)=2$.

\subsection{The case $\rho(X)=3$}\label{rho3_section}

We consider the case $\rho(X)=3$. We assume that $D$ is a suspicious divisor. 

\noindent\textbf{The case No.\ 1.}\,\,
Assume that $X$ belongs to No.\ 1 in \cite[Table 3]{MoMu}. 
Then $\Nef(X)=\overline{\Eff}(X)$ by \cite{mat}. Thus there is no 
suspicious divisor. Hence $(X, -K_X)$ is divisorially stable. 

\noindent\textbf{The case No.\ 2.}\,\,
Assume that $X$ belongs to No.\ 2 in \cite[Table 3]{MoMu}. 
Let $H_1$, $H_2$ be a divisor on $X$ corresponds to the pullback of 
$\sO_{\pr^1\times\pr^1}(1, 0)$, $\sO_{\pr^1\times\pr^1}(0, 1)$, respectively. 
Let $F$ be the exceptional divisor of the morphism $\cont_{l_2}$. 
Then we have 
\begin{eqnarray*}
\Nef(X)&=&\R_{\geq 0}[H_1]+\R_{\geq 0}[H_2]+\R_{\geq 0}[2H_1+H_2+2F], \\
\overline{\Eff}(X)&=&\R_{\geq 0}[H_1]+\R_{\geq 0}[H_2]+\R_{\geq 0}[F],\\
-K_X&\sim& 2H_1+H_2+F,\\
\Pic(X)&=&\Z[H_1]\oplus\Z[H_2]\oplus\Z[F].
\end{eqnarray*}
Hence $D\sim H_1$. 
In this case, $\tau_1=1$, $X_2$ is the image of the morphism $\cont_{l_3}$, 
$\sN_{D_2/X_2}$ is nonzero effective, $\tau_2=2$ and $m=2$. Thus $\eta(D)>0$ 
by Lemma \ref{conv_lem}. Hence $(X, -K_X)$ is divisorially stable. 

\noindent\textbf{The case No.\ 3.}\,\,
Assume that $X$ belongs to No.\ 3 in \cite[Table 3]{MoMu}. 
Let $H_1$, $H_2$, $H_3$ be a divisor corresponds to the restriction of 
$\sO(1, 0, 0)$, $\sO(0, 1, 0)$, $\sO(0, 0, 1)$ 
on $X$, respectively. Let $E_2$, $E_3$ be the exceptional divisor of $\cont_{l_2}$, 
$\cont_{l_3}$, respectively. Then $E_2\sim H_1-H_2+2H_3$, $E_3\sim -H_1+H_2+2H_3$ 
and 
\begin{eqnarray*}
\Nef(X)&=&\R_{\geq 0}[H_1]+\R_{\geq 0}[H_2]+\R_{\geq 0}[H_3], \\
\overline{\Eff}(X)&=&\R_{\geq 0}[H_1]+\R_{\geq 0}[H_2]+\R_{\geq 0}[H_1-H_2+2H_3]
+\R_{\geq 0}[-H_1+H_2+2H_3],\\
-K_X&\sim& H_1+H_2+H_3,\\
\Pic(X)&=&\Z[H_1]\oplus\Z[H_2]\oplus\Z[H_3].
\end{eqnarray*}
Hence $D\sim H_1$ or $H_2$. If $D\sim H_1$, then $\tau_1=1$, 
$X_2=\pr^1\times\pr^2$, $\tau_2=3/2$ and $m=2$. Thus $\eta(D)>0$ by 
Lemma \ref{conv_lem}. If $D\sim H_2$, then we have $\eta(D)>0$ in the same way. 
Hence $(X, -K_X)$ is divisorially stable. 

\noindent\textbf{The case No.\ 4.}\,\,
Assume that $X$ belongs to No.\ 4 in \cite[Table 3]{MoMu}. 
Let $H_1$, $H_2$ be a divisor corresponds to the pullback of 
$\sO_{\pr^1\times\pr^2}(1, 0)$, $\sO_{\pr^1\times\pr^2}(0, 1)$ 
on $X$, respectively. Let $E$ be the exceptional divisor of $\cont_{l_1}$. 
Then we have 
\begin{eqnarray*}
\Nef(X)&=&\R_{\geq 0}[H_1]+\R_{\geq 0}[H_2]+\R_{\geq 0}[H_2-E], \\
\overline{\Eff}(X)&=&\R_{\geq 0}[H_1]+\R_{\geq 0}[H_2-E]+\R_{\geq 0}[E],\\
-K_X&\sim& H_1+2H_2-E,\\
\Pic(X)&=&\Z[H_1]\oplus\Z[H_2]\oplus\Z[E].
\end{eqnarray*}
Hence $D\sim H_2-E$. In this case, $D$ is nef, $\tau_1=1$, $X_2$ is the image 
of the morphism $\cont_{l_1}$, $\tau_2=2$ and $m=2$. Thus $\eta(D)>0$ 
by Lemma \ref{conv_lem}. Hence $(X, -K_X)$ is divisorially stable. 

\noindent\textbf{The case No.\ 5.}\,\,
Assume that $X$ belongs to No.\ 5 in \cite[Table 3]{MoMu}. 
Let $H_1$, $H_2$ be a divisor corresponds to the pullback of 
$\sO_{\pr^1\times\pr^2}(1, 0)$, $\sO_{\pr^1\times\pr^2}(0, 1)$ 
on $X$, respectively. Let $E_1$, $E_2$ be the exceptional divisor of $\cont_{l_1}$, 
$\cont_{l_2}$, respectively. Then $E_2\sim 2H_2-E_1$ and 
\begin{eqnarray*}
\Nef(X)&=&\R_{\geq 0}[H_1]+\R_{\geq 0}[H_2]+\R_{\geq 0}[2H_1+5H_2-2E_1], \\
\overline{\Eff}(X)&=&\R_{\geq 0}[H_1]+\R_{\geq 0}[2H_2-E_1]+\R_{\geq 0}[E_1],\\
-K_X&\sim& 2H_1+3H_2-E_1,\\
\Pic(X)&=&\Z[H_1]\oplus\Z[H_2]\oplus\Z[E_1].
\end{eqnarray*}
Hence $D\sim 2H_2-E_1$, $H_1$ or $H_1+2H_2-E_1$. 
Assume that $D\sim 2H_2-E_1$, that is, $D=E_2$. 
Then $\sN_{D/X}\simeq\sO_{\pr^1\times\pr^1}(-1, -1)$, $\tau_1=1$, 
$X_2=\pr^1\times\pr^2$, $\sN_{D_2/X_2}\simeq\sO_{\pr^1\times\pr^1}(0, 4)$, 
$\tau_2=3/2$ and $m=2$. Thus $\eta(D)>0$ by Lemma \ref{three_lem}. 
Assume that $D\sim H_1$. In this case, $\tau_1=1$, $X_2$ is the image of the morphism 
$\cont_{l_2}$, $\sN_{D_2/X_2}$ is nonzero effective, $\tau_2=2$ and $m=2$. 
Thus $\eta(D)>0$ by Lemma \ref{conv_lem}. 
Assume that $D\sim H_1+2H_2-E_1$. In this case, $\tau_1=1$, $X_2=\pr^1\times\pr^2$, 
$D_2\in|\sO(1, 2)|$, $\tau_2=3/2$ and $m=2$. Since $\eta_1=35/12$ and  
$\eta_2=-7/48$, we have $\eta(D)/3>0$. 
Hence $(X, -K_X)$ is divisorially stable. 

\noindent\textbf{The case No.\ 6.}\,\,
Assume that $X$ belongs to No.\ 6 in \cite[Table 3]{MoMu}. 
Let $H_3$ be a divisor corresponds to the pullback of 
$\sO_{\pr^3}(1)$, let $E_1$, $E_2$ be the exceptional divisor of $\cont_{l_1}$, 
$\cont_{l_2}$, respectively. Let $H_1:=-E_1+H_3$ and $H_2:=-E_2+2H_3$. Then 
\begin{eqnarray*}
\Nef(X)&=&\R_{\geq 0}[H_1]+\R_{\geq 0}[H_2]+\R_{\geq 0}[H_3], \\
\overline{\Eff}(X)&=&\R_{\geq 0}[H_1]+\R_{\geq 0}[H_2]+\R_{\geq 0}[-H_1+H_3]+
\R_{\geq 0}[-H_2+2H_3],\\
-K_X&\sim& H_1+H_2+H_3,\\
\Pic(X)&=&\Z[H_1]\oplus\Z[H_2]\oplus\Z[H_3].
\end{eqnarray*}
Hence $D\sim H_1$, $H_2$ or $H_1+H_2$. 
Assume that $D\sim H_1$. In this case, $\tau_1=1$, $X_2$ is the image of the morphism 
$\cont_{l_1}$, $D_2$ corresponds to the pullback of $\sO_{\pr^3}(1)$, 
$\tau_2=2$ and $m=2$. Thus $\eta(D)>0$ by Lemma \ref{conv_lem}. 
Assume that $D\sim H_2$. In this case, $\tau_1=1$, $X_2$ is the image of the morphism 
$\cont_{l_2}$, $D_2$ corresponds to the pullback of $\sO_{\pr^3}(2)$, 
$\tau_2=3/2$ and $m=2$. Thus $\eta(D)>0$ by Lemma \ref{conv_lem}. 
Assume that $D\sim H_1+H_2$. In this case, $\tau_1=1$, $X_2=\pr^3$, 
$D_2\in|\sO_{\pr^3}(3)|$, $\tau_2=4/3$ and $m=2$. 
Thus $\eta(D)>0$ by Lemma \ref{conv_lem}. 
Hence $(X, -K_X)$ is divisorially stable. 

\noindent\textbf{The case No.\ 7.}\,\,
Assume that $X$ belongs to No.\ 7 in \cite[Table 3]{MoMu}. 
Let $H_2$, $H_3$ be a divisor corresponds to the pullback of 
$\sO_{W_6}(1, 0)$, $\sO_{W_6}(0, 1)$, respectively. 
let $E_1$, $E_2$, $E_3$ be the exceptional divisor of $\cont_{l_1}$, 
$\cont_{l_2}$, $\cont_{l_3}$, respectively. 
Let $H_1:=-E_1+H_2+H_3$. Then 
$E_2\sim H_1+2H_2-H_3$, $E_3\sim H_1-H_2+2H_3$ and 
\begin{eqnarray*}
\Nef(X)&=&\R_{\geq 0}[H_1]+\R_{\geq 0}[H_2]+\R_{\geq 0}[H_3], \\
\overline{\Eff}(X)&=&\R_{\geq 0}[H_1]+\R_{\geq 0}[H_1+2H_2-H_3]
+\R_{\geq 0}[-H_1+H_2+H_3]+\R_{\geq 0}[H_1-H_2+2H_3],\\
-K_X&\sim& H_1+H_2+H_3,\\
\Pic(X)&=&\Z[H_1]\oplus\Z[H_2]\oplus\Z[H_3].
\end{eqnarray*}
Hence $D\sim H_1$, $H_2$ or $H_3$. 
Assume that $D\sim H_1$. In this case, $\tau_1=1$, $X_2=W_6$, 
$D_2\sim(-1/2)K_{W_6}$, 
$\tau_2=2$ and $m=2$. Thus $\eta(D)>0$ by Lemma \ref{conv_lem}. 
Assume that $D\sim H_2$. In this case, $\tau_1=1$, $X_2=\pr^1\times\pr^2$, 
$\tau_2=3/2$ and $m=2$. Thus $\eta(D)>0$ by Lemma \ref{conv_lem}. 
If $D\sim H_3$, then we have $\eta(D)>0$ in the same way. 
Hence $(X, -K_X)$ is divisorially stable. 

\noindent\textbf{The case No.\ 8.}\,\,
Assume that $X$ belongs to No.\ 8 in \cite[Table 3]{MoMu}. 
Let $H_1$, $H_2$, $H_3$ be a divisor corresponds to the restriction of 
$\sO_{\F_1\times\pr^2}(\sigma_1+f_1, 0)$, $\sO_{\F_1\times\pr^2}(0, 1)$, 
$\sO_{\F_1\times\pr^2}(f_1, 0)$, respectively. 
Let $E_1$, $E_2$ be the exceptional divisor of $\cont_{l_1}$, 
$\cont_{l_2}$, respectively. Then 
$E_1\sim H_1-H_3$, $E_2\sim -H_1+2H_2+H_3$ and 
\begin{eqnarray*}
\Nef(X)&=&\R_{\geq 0}[H_1]+\R_{\geq 0}[H_2]+\R_{\geq 0}[H_3], \\
\overline{\Eff}(X)&=&\R_{\geq 0}[H_1-H_3]+\R_{\geq 0}[H_3]
+\R_{\geq 0}[-H_1+2H_2+H_3],\\
-K_X&\sim& H_1+H_2+H_3,\\
\Pic(X)&=&\Z[H_1]\oplus\Z[H_2]\oplus\Z[H_3].
\end{eqnarray*}
Hence $D\sim H_1$, $H_3$ or $H_1-H_3$. 
Assume that $D\sim H_1$. In this case, $\tau_1=1$, $X_2=\pr^1\times\pr^2$, 
$\tau_2=3/2$ and $m=2$. Thus $\eta(D)>0$ by Lemma \ref{conv_lem}. 
Assume that $D\sim H_3$. In this case, $\tau_1=1$, $X_2$ is the image of the morphism 
$\cont_{l_1}$, $\sN_{D_2/X_2}$ is nonzero effective, 
$\tau_2=2$ and $m=2$. Thus $\eta(D)>0$ by Lemma \ref{conv_lem}. 
Assume that $D\sim H_1-H_3$, that is, $D=E_1$. 
In this case, $\sN_{D/X}\simeq\sO_{\pr^1\times\pr^1}(-1, 0)$, $\tau_1=1$, $X_2=\pr^1\times\pr^2$, $\sN_{D_2/X_2}\simeq\sO_{\pr^1\times\pr^1}(0, 4)$, 
$\tau_2=2$ and $m=2$. Thus $\eta(D)>0$ by Lemma \ref{conv_lem}. 
Hence $(X, -K_X)$ is divisorially stable. 

\noindent\textbf{The case No.\ 9.}\,\,
Assume that $X$ belongs to No.\ 9 in \cite[Table 3]{MoMu}. 
Let $E_1,\dots,E_4$ be the exceptional divisor of $\cont_{l_1},\dots,\cont_{l_4}$, 
respectively. Let $H:=(E_1+E_3)/4$. Then $E_4\sim -2H+E_1+E_2$ and 
\begin{eqnarray*}
\Nef(X)&=&\R_{\geq 0}[E_1+E_2]+\R_{\geq 0}[E_1+2E_2]+
\R_{\geq 0}[H]+\R_{\geq 0}[E_2+2H], \\
\overline{\Eff}(X)&=&\R_{\geq 0}[4H-E_1]+\R_{\geq 0}[-2H+E_1+E_2]+\R_{\geq 0}[E_1]
+\R_{\geq 0}[E_2],\\
-K_X&\sim& H+E_1+2E_2,\\
\Pic(X)&=&\Z[H]\oplus\Z[E_1]\oplus\Z[E_2].
\end{eqnarray*}
Hence $D\sim 4H-E_1$, $2H+E_2$, $H$, $H+E_2$, $E_1$, $E_2$, $E_1+E_2$, 
$-H+E_1+E_2$ or $-2H+E_1+E_2$. If $D\sim E_2$ or $-2H+E_1+E_2$, then $\eta(D)=0$ 
by Lemma \ref{rhothree_lem}. Assume that $D\sim E_1$. In this case, $\tau_1=1/4$, 
$X_2$ is the image of the morphism $\cont_{l_2}$, $\tau_2=1$, $X_3$ is the projective 
cone of a Veronese surface, $\tau_3=5/4$ and $m=3$. 
We have $\eta_1=1225/384$ and $\eta_3=-1/96$
since $((-K_{X_3})^{\cdot 3})=125/2$. Thus $\eta(D)/3>1225/384-1/96>0$. 
If $D\sim 4H-E_1$, then $\eta(D)>0$ in the same way. 
Assume that $D\sim E_1+E_2$. In this case, $D$ is nef, $\tau_1=1/2$, 
$X_2$ is the image of the morphism $\cont_{l_2}$, $\tau_2=1$, $X_3$ is the projective 
cone of a Veronese surface, $\tau_3=5/4$ and $m=3$. 
Thus $\eta(D)>0$ by Lemma \ref{conv_lem}. If $D\sim 2H+E_2$, then $\eta(D)>0$ 
in the same way. Assume that $D\sim H_1$. In this case, $\tau_1=1$, 
$X_2$ is the image of the contraction morphism associated to the extremal face 
spanned by $\R_{\geq 0}[l_2]$ and $\R_{\geq 0}[l_4]$, 
$D_2\sim_\Q(-1/3)K_{X_2}$, $\tau_2=3$ and $m=2$. 
We note that $((-K_{X_2})^{\cdot 3})=27$. Since $\eta_1=10/3$ and $\eta_2=4/3$, 
we have $\eta(D)/3>0$. 
Assume that $D\sim H+E_2$. In this case, $\tau_1=1$, $X_2$ is the projective cone 
of a Veronese surface, $D_2\sim_\Q(-3/5)K_{X_2}$, $\tau_2=5/3$ and $m=2$. 
Since $\eta_1=19/6$ and $\eta_2=-2/9$, we have $\eta(D)/3>0$. 
If $D\sim -H+E_1+E_2$, then $\eta(D)>0$ 
in the same way. Hence $(X, -K_X)$ is divisorially semistable but not divisorially stable. 

\noindent\textbf{The case No.\ 10.}\,\,
Assume that $X$ belongs to No.\ 10 in \cite[Table 3]{MoMu}. 
Let $H_3$ be a divisor corresponds to the pullback of 
$\sO_Q(1)$, let $E_1$, $E_2$ be the exceptional divisor of $\cont_{l_1}$, 
$\cont_{l_2}$, respectively. Let $H_1:=H_3-E_2$ and $H_2:=H_3-E_1$. Then we have 
\begin{eqnarray*}
\Nef(X)&=&\R_{\geq 0}[H_1]+\R_{\geq 0}[H_2]+\R_{\geq 0}[H_3], \\
\overline{\Eff}(X)&=&\R_{\geq 0}[H_1]+\R_{\geq 0}[H_2]+\R_{\geq 0}[-H_1+H_3]
+\R_{\geq 0}[-H_2+H_3],\\
-K_X&\sim& H_1+H_2+H_3,\\
\Pic(X)&=&\Z[H_1]\oplus\Z[H_2]\oplus\Z[H_3].
\end{eqnarray*}
Hence $D\sim H_1$, $H_2$ or $H_1+H_2$. 
Assume that $D\sim H_1$. In this case, $\tau_1=1$, $X_2$ is the blowup of $Q$ along 
a conic, $D_2$ corresponds to a pullback of $\sO_Q(1)$, 
$\tau_2=2$ and $m=2$. Thus $\eta(D)>0$ by Lemma \ref{conv_lem}. 
If $D\sim H_2$, then $\eta(D)>0$ in the same way. 
Assume that $D\sim H_1+H_2$. In this case, $\tau_1=1$, $X_2=Q$, 
$D_2\in|\sO_Q(2)|$, $\tau_2=3/2$ and $m=2$. 
Thus $\eta(D)>0$ by Lemma \ref{conv_lem}. Hence $(X, -K_X)$ is divisorially stable. 

\noindent\textbf{The case No.\ 11.}\,\,
Assume that $X$ belongs to No.\ 11 in \cite[Table 3]{MoMu}. 
Let $H_2$, $H_3$ be a divisor corresponds to the pullback of 
$H_{\pr^2}$, $\xi_\pr$ on $V_7$, respectively. 
Let $E_1,\dots,E_3$ be the exceptional divisor of $\cont_{l_1},\dots,\cont_{l_3}$, 
respectively. Let $H_1:=-E_1+H_2+H_3$. Then $E_2\sim -H_2+H_3$, 
$E_3\sim H_1+2H_2-H_3$ and 
\begin{eqnarray*}
\Nef(X)&=&\R_{\geq 0}[H_1]+\R_{\geq 0}[H_2]+\R_{\geq 0}[H_3], \\
\overline{\Eff}(X)&=&\R_{\geq 0}[H_1]+\R_{\geq 0}[H_1+2H_2-H_3]
+\R_{\geq 0}[-H_2+H_3]+\R_{\geq 0}[-H_1+H_2+H_3],\\
-K_X&\sim& H_1+H_2+H_3,\\
\Pic(X)&=&\Z[H_1]\oplus\Z[H_2]\oplus\Z[H_3].
\end{eqnarray*}
Hence $D\sim H_3$, $-H_2+H_3$, $H_1$, $H_2$, $H_1+H_2$ or $H_1+2H_2-H_3$. 
Assume that $D\sim H_3$. In this case, $\tau_1=1$, $X_2=\pr^1\times\pr^2$, 
$\tau_2=3/2$ and $m=2$. Thus $\eta(D)>0$ by Lemma \ref{conv_lem}. 
Assume that $D\sim -H_2+H_3$, that is, $D=E_2$. In this case, 
$\sN_{D/X}\simeq\sO_{\F_1}(-\sigma_1-f_1)$, $\tau_1=1$, $X_2=\pr^1\times\pr^2$, 
$D_2\in|\sO(1, 1)|$, $\tau_2=2$ and $m=2$. Thus $\eta(D)>0$ by Lemma \ref{three_lem}. 
Assume that $D\sim H_1$. In this case, $\tau_1=1$, $X_2=V_7$, $D_2\sim(-1/2)K_{V_7}$, 
$\tau_2=2$ and $m=2$. Thus $\eta(D)>0$ by Lemma \ref{conv_lem}. 
Assume that $D\sim H_2$. In this case, $\tau_1=1$, $X_2$ is the blowup of $\pr^3$ 
along a quartic which is an intersection of two quadrics, $D_2$ corresponds to 
the pullback of $\sO_{\pr^3}(1)$, $\tau_2=2$ and $m=2$. 
Thus $\eta(D)>0$ by Lemma \ref{conv_lem}. 
Assume that $D\sim H_1+H_2$. In this case, $\tau_1=1$, $X_2=\pr^3$, 
$\tau_2=4/3$ and $m=2$. Thus $\eta(D)>0$ by Lemma \ref{conv_lem}. 
Assume that $D\sim H_1+2H_2-H_3$, that is, $D=E_3$. In this case, $\tau_1=1/2$, 
$X_2$ is the blowup of $\pr^3$ 
along a quartic which is an intersection of two quadrics, $D_2$ corresponds to 
the sum of the pullback of $\sO_{\pr^3}(1)$ and the pull back of $\sO_{\pr^1}(1)$, 
$\tau_2=1$, $X_3=\pr^3$, $D_3\in|\sO_{\pr^3}(3)|$, $\tau_3=4/3$ and $m=3$. 
Since $\tau_2=173/192$ and $\tau_3=-1/36$, 
we have $\eta(D)/3>173/192-1/36>0$. Hence $(X, -K_X)$ is divisorially stable. 

\noindent\textbf{The case No.\ 12.}\,\,
Assume that $X$ belongs to No.\ 12 in \cite[Table 3]{MoMu}. 
Let $H_3$ be a divisor corresponds to the pullback of $\sO_{\pr^3}(1)$. 
Let $E_1,\dots,E_3$ be the exceptional divisor of $\cont_{l_1},\dots,\cont_{l_3}$, 
respectively. Let $H_1:=-E_1+H_3$ and $H_2:=-E_2+2H_3$. Then 
$E_3\sim H_1+2H_2-H_3$ and 
\begin{eqnarray*}
\Nef(X)&=&\R_{\geq 0}[H_1]+\R_{\geq 0}[H_2]+\R_{\geq 0}[H_3], \\
\overline{\Eff}(X)&=&\R_{\geq 0}[H_1]+\R_{\geq 0}[H_1+2H_2-H_3]
+\R_{\geq 0}[-H_1+H_3]+\R_{\geq 0}[-H_2+2H_3],\\
-K_X&\sim& H_1+H_2+H_3,\\
\Pic(X)&=&\Z[H_1]\oplus\Z[H_2]\oplus\Z[H_3].
\end{eqnarray*}
Hence $D\sim -H_1+H_3$, $H_3$, $H_1$, $H_2$ or $H_1+H_2$. 
Assume that $D\sim -H_1+H_3$, that is, $D=E_1$. In this case, 
$\sN_{D/X}\simeq\sO_{\pr^1\times\pr^1}(-1, 1)$, $\tau_1=1$, $X_2=\pr^1\times\pr^2$, 
$\sN_{D_2/X_2}\simeq\sO_{\pr^1\times\pr^1}(0, 4)$, $\tau_2=3/2$ and $m=2$. 
Thus $\eta(D)>0$ by Lemma \ref{three_lem}. 
Assume that $D\sim H_3$. In this case, $\tau_1=1$, $X_2=\pr^1\times\pr^2$, 
$\tau_2=3/2$ and $m=2$. Thus $\eta(D)>0$ by Lemma \ref{conv_lem}. 
Assume that $D\sim H_1$. In this case, $\tau_1=1$, $X_2$ is the blowup of $\pr^3$ 
along a twisted cubic, $D_2$ corresponds to the pullback of $\sO_{\pr^3}(1)$, 
$\tau_2=2$ and $m=2$. Thus $\eta(D)>0$ by Lemma \ref{conv_lem}. 
Assume that $D\sim H_2$. In this case, $\tau_1=1$, 
$X_2=\pr_{\pr^1}(\sO^{\oplus 2}\oplus\sO(1))$, $D_2\in|\xi_\pr^{\otimes 2}|$, 
$\tau_2=3/2$ and $m=2$. Thus $\eta(D)>0$ by Lemma \ref{conv_lem}. 
Assume that $D\sim H_1+H_2$. In this case, $\tau_1=1$, 
$X_2=\pr^3$, $\tau_2=4/3$ and $m=2$. Thus $\eta(D)>0$ by Lemma \ref{conv_lem}. 
Hence $(X, -K_X)$ is divisorially stable. 

\noindent\textbf{The case No.\ 13.}\,\,
Assume that $X$ belongs to No.\ 13 in \cite[Table 3]{MoMu}. 
Let $H_2$, $H_3$ be a divisor corresponds to the pullback of $\sO_{W_6}(1, 0)$, 
$\sO_{W_6}(0, 1)$, respectively. 
Let $E_1,\dots,E_3$ be the exceptional divisor of $\cont_{l_1},\dots,\cont_{l_3}$, 
respectively. Let $H_1:=-E_1+H_2+H_3$. Then 
$E_2\sim H_1-H_2+H_3$, $E_3\sim H_1+H_2-H_3$ and 
\begin{eqnarray*}
\Nef(X)&=&\R_{\geq 0}[H_1]+\R_{\geq 0}[H_2]+\R_{\geq 0}[H_3], \\
\overline{\Eff}(X)&=&\R_{\geq 0}[H_1+H_2-H_3]+\R_{\geq 0}[H_1-H_2+H_3]
+\R_{\geq 0}[-H_1+H_2+H_3],\\
-K_X&\sim& H_1+H_2+H_3,\\
\Pic(X)&=&\Z[H_1]\oplus\Z[H_2]\oplus\Z[H_3].
\end{eqnarray*}
Hence $D\sim H_1$, $H_2$ or $H_3$. 
Assume that $D\sim H_1$. In this case, $\tau_1=1$, $X_2=W_6$, 
$D_2\sim(-1/2)K_{W_6}$, $\tau_2=2$ and $m=2$. 
Thus $\eta(D)>0$ by Lemma \ref{conv_lem}. If $D\sim H_2$ or $H_3$, then 
$\eta(D)>0$ 
in the same way. Hence $(X, -K_X)$ is divisorially stable. 

\noindent\textbf{The case No.\ 14.}\,\,
Assume that $X$ belongs to No.\ 14 in \cite[Table 3]{MoMu}. 
By Lemma \ref{rhothree_lem}, $(X, -K_X)$ is not divisorially semistable. 

\noindent\textbf{The case No.\ 15.}\,\,
Assume that $X$ belongs to No.\ 15 in \cite[Table 3]{MoMu}. 
Let $H_3$ be a divisor corresponds to the pullback of $\sO_Q(1)$. 
Let $E_1,\dots,E_3$ be the exceptional divisor of $\cont_{l_1},\dots,\cont_{l_3}$, 
respectively. Let $H_1:=-E_2+H_3$ and $H_2:=-E_1+H_3$. Then 
$E_3\sim H_1+2H_2-H_3$ and 
\begin{eqnarray*}
\Nef(X)&=&\R_{\geq 0}[H_1]+\R_{\geq 0}[H_2]+\R_{\geq 0}[H_3], \\
\overline{\Eff}(X)&=&\R_{\geq 0}[H_1]+\R_{\geq 0}[H_1+2H_2-H_3]
+\R_{\geq 0}[-H_1+H_3]+\R_{\geq 0}[-H_2+H_3],\\
-K_X&\sim& H_1+H_2+H_3,\\
\Pic(X)&=&\Z[H_1]\oplus\Z[H_2]\oplus\Z[H_3].
\end{eqnarray*}
Hence $D\sim H_3$, $-H_1+H_3$, $-H_2+H_3$, $H_1$, $H_2$, $H_1+H_2$ or 
$H_1+2H_2-H_3$. 
Assume that $D\sim H_3$. In this case, $\tau_1=1$, $X_2=\pr^1\times\pr^2$, 
$D_2\in|\sO(1, 2)|$, $\tau_2=3/2$ and $m=2$. 
Thus $\eta(D)>0$ by Lemma \ref{conv_lem}. 
Assume that $D\sim -H_1+H_3$, that is, $D=E_2$. 
In this case, $\sN_{D/X}\simeq\sO_{\pr^1\times\pr^1}(-1, 2)$, 
$\tau_1=1$, $X_2=\pr^1\times\pr^2$, 
$\sN_{D_2/X_2}\simeq\sO_{\pr^1\times\pr^1}(0, 2)$, $\tau_2=3/2$ and $m=2$. 
Thus $\eta(D)>0$ by Lemma \ref{three_lem}. 
Assume that $D\sim -H_2+H_3$, that is, $D=E_1$. 
In this case, $\sN_{D/X}\simeq\sO_{\F_1}(-\sigma_1)$, 
$\tau_1=1$, $X_2=\pr^1\times\pr^2$, 
$\sN_{D_2/X_2}\simeq\sO_{\F_1}(\sigma_1+2f_1)$, $\tau_2=2$ and $m=2$. 
Thus $\eta(D)>0$ by Lemma \ref{three_lem}. 
Assume that $D\sim H_1$. In this case, $\tau_1=1$, $X_2$ is the blowup of $Q$ along 
a line, $D_2$ corresponds to the pullback of $\sO_Q(1)$, $\tau_2=2$ and $m=2$. 
Thus $\eta(D)>0$ by Lemma \ref{conv_lem}. 
Assume that $D\sim H_2$. In this case, $\tau_1=1$, $X_2$ is the blowup of $Q$ along 
a conic, $D_2$ corresponds to the pullback of $\sO_Q(1)$, $\tau_2=2$ and $m=2$. 
Thus $\eta(D)>0$ by Lemma \ref{conv_lem}. 
Assume that $D\sim H_1+H_2$. In this case, $\tau_1=1$, $X_2=Q$, 
$\tau_2=3/2$ and $m=2$. Thus $\eta(D)>0$ by Lemma \ref{conv_lem}. 
Assume that $D\sim H_1+2H_2-H_3$, that is, $D=E_3$. 
In this case, $\tau_1=1/2$, $X_2$ is the blowup of $Q$ along a conic, 
$D_2$ corresponds to the sum of the pullback of $\sO_Q(1)$ and the pullback 
of $\sO_{\pr^1}(1)$, $\tau_2=1$, $X_3=Q$, $D_2\in|\sO_Q(2)|$, $\tau_3=3/2$ and $m=3$. 
Since $\eta_2=9/8$ and $\eta_3=-1/12$, 
we have $\eta(D)/3>9/8-1/12>0$. Hence $(X, -K_X)$ is divisorially stable. 

\noindent\textbf{The case No.\ 16.}\,\,
Assume that $X$ belongs to No.\ 16 in \cite[Table 3]{MoMu}. 
We consider the case that $D$ is the strict transform of the negative section of 
the morphism $V_7\to\pr^2$. Then $\sN_{D/X}\simeq\sO_{\F_1}(-\sigma_1-f_1)$, 
$\tau_1=1$, $X_2=W_6$, $D_2\in|\sO_{W_6}(1, 0)|$, $\tau_2=2$ and $m=2$. 
Thus $\eta(D)<0$ by Lemma \ref{three_lem}. 
Hence $(X, -K_X)$ is not divisorially semistable. 

\noindent\textbf{The case No.\ 17.}\,\,
Assume that $X$ belongs to No.\ 17 in \cite[Table 3]{MoMu}. 
Let $H_1$, $H_2$, $H_3$ be a divisor corresponds to the restriction of 
$\sO_{\pr^1\times\pr^1\times\pr^2}(1, 0, 0)$, 
$\sO_{\pr^1\times\pr^1\times\pr^2}(0, 1, 0)$, 
$\sO_{\pr^1\times\pr^1\times\pr^2}(0, 0, 1)$, respectively. 
Let $E_1$, $E_2$ be the exceptional divisor of $\cont_{l_1}$, $\cont_{l_2}$, respectively. 
Then $E_1\sim H_1-H_2+H_3$, $E_2\sim -H_1+H_2+H_3$ and 
\begin{eqnarray*}
\Nef(X)&=&\R_{\geq 0}[H_1]+\R_{\geq 0}[H_2]+\R_{\geq 0}[H_3], \\
\overline{\Eff}(X)&=&\R_{\geq 0}[H_1]+\R_{\geq 0}[H_2]
+\R_{\geq 0}[H_1-H_2+H_3]+\R_{\geq 0}[-H_1+H_2+H_3],\\
-K_X&\sim& H_1+H_2+2H_3,\\
\Pic(X)&=&\Z[H_1]\oplus\Z[H_2]\oplus\Z[H_3].
\end{eqnarray*}
Hence $D\sim H_1+H_3$, $H_2+H_3$, $H_1-H_2+H_3$, $-H_1+H_2+H_3$, 
$H_1$ or $H_2$. 
Assume that $D\sim H_1+H_3$. In this case, $\tau_1=1$, $X_2=\pr^1\times\pr^2$, 
$\tau_2=3/2$ and $m=2$. Thus $\eta(D)>0$ by Lemma \ref{conv_lem}. 
If $D\sim H_2+H_3$, then $\eta(D)>0$ in the same way. 
Assume that $D\sim H_1-H_2+H_3$. In this case, $\tau_1=1$, $X_2=\pr^1\times\pr^2$, 
$D_2\in|\sO(0, 2)|$, $\tau_2=3/2$ and $m=2$. Since $\eta_1=31/6$ and 
$\eta_2=-1/3$, 
we have $\eta(D)/3>0$. If $D\sim -H_1+H_2+H_3$, then $\eta(D)>0$ in the same way. 
Assume that $D\sim H_1$. In this case, $\tau_1=1$, $X_2=\pr^1\times\pr^2$, 
$D_2\in|\sO(1, 1)|$, $\tau_2=2$ and $m=2$. Thus $\eta(D)>0$ by Lemma \ref{conv_lem}. 
If $D\sim H_2$, then $\eta(D)>0$ in the same way. 
Hence $(X, -K_X)$ is divisorially stable.

\noindent\textbf{The case No.\ 18.}\,\,
Assume that $X$ belongs to No.\ 18 in \cite[Table 3]{MoMu}. 
We consider the case that $D$ is the strict transform of the plane in $\pr^3$ passing 
through the conic which is the center of the blowup. 
Then $\sN_{D/X}\simeq\sO_{\F_1}(-\sigma_1-f_1)$, 
$\tau_1=1$, $X_2=\pr_{\pr^1}(\sO^{\oplus 2}\oplus\sO(1))$, 
$D_2\in|\xi_\pr|$, $\tau_2=3$ and $m=2$. 
Thus $\eta(D)<0$ by Lemma \ref{three_lem}. 
Hence $(X, -K_X)$ is not divisorially semistable. 

\noindent\textbf{The case No.\ 19.}\,\,
Assume that $X$ belongs to No.\ 19 in \cite[Table 3]{MoMu}. 
Let $H$ be a divisor corresponds to the pullback of $\sO_Q(1)$. 
Let $E_1,\dots,E_4$ be the exceptional divisor of $\cont_{l_1},\dots,\cont_{l_4}$, 
respectively. 
Then $E_3\sim -2E_2+H$, $E_4\sim -2E_1+H$ and 
\begin{eqnarray*}
\Nef(X)&=&\R_{\geq 0}[H]+\R_{\geq 0}[-E_1+H]+\R_{\geq 0}[-E_2+H]
+\R_{\geq 0}[-E_1-E_2+H], \\
\overline{\Eff}(X)&=&\R_{\geq 0}[E_1]+\R_{\geq 0}[E_2]
+\R_{\geq 0}[-2E_1+H]+\R_{\geq 0}[-2E_2+H],\\
-K_X&\sim& -2E_1-2E_2+3H,\\
\Pic(X)&=&\Z[E_1]\oplus\Z[E_2]\oplus\Z[H].
\end{eqnarray*}
Hence $D\sim E_1$, $E_2$, $-E_1-E_2+H$, $-E_1+H$, $-E_2+H$, $-2E_1+H$, $-2E_2+H$, 
$-2E_1+E_2+H$, $E_1-2E_2+H$, $-2E_1-E_2+2H$, $-E_1-2E_2+2H$, 
$-3E_1-E_2+2H$ or $-E_1-3E_2+2H$. 
Assume that $D\sim E_1$ or $E_2$. In this case, $\eta(D)=0$ by Lemma \ref{three_lem}. 
Assume that $D\sim -E_1-E_2+H$. In this case, $\tau_1=1$, $X_2=Q$, 
$D_2\in|\sO_Q(1)|$, $\tau_2=3$ and $m=2$. Since $\eta_1=8/3$ and $\eta_2=-5/6$, 
we have $\eta(D)/3>0$. 
Assume that $D\sim -E_1+H$. In this case, $D$ is nef, $\tau_1=1$, $X_2=V_7$, 
$D_2\sim(-1/2)K_{V_7}$, $\tau_2=2$ and $m=2$. 
Thus $\eta(D)>0$ by Lemma \ref{conv_lem}. If $D\sim -E_2+H$, then $\eta(D)>0$ 
in the same way. 
Assume that $D\sim -2E_1+H$, that is, $D=E_4$. In this case, $\tau_1=1$, $X_2=\pr^3$, 
$D_2\in|\sO(2)|$, $\tau_2=2$ and $m=2$. Since $\eta_1=31/6$ and $\eta_2=-2/3$, 
we have $\eta(D)/3>0$. If $D\sim -2E_2+H$, then $\eta(D)>0$ in the same way. 
Assume that $D\sim -2E_1+E_2+H$. In this case, $\tau_1=1/2$, $X_2=V_7$, 
$D_2\in|\xi_\pr\otimes H_{\pr^2}^{\otimes 2}|$, $\tau_2=1$, $X_3=\pr^3$, 
$D_3\in|\sO(3)|$, $\tau_3=4/3$ and $m=3$. 
Since $\eta_2=91/64$ and $\eta_3=-1/36$, 
we have $\eta(D)/3>91/64-1/36>0$. If $D\sim E_1-2E_2+H$, then $\eta(D)>0$ 
in the same way. 
Assume that $D\sim -2E_1-E_2+2H$. In this case, $D$ is nef, $\tau_1=1$, $X_2=\pr^3$, 
$D_2\in|\sO(3)|$, $\tau_2=4/3$ and $m=2$. 
Thus $\eta(D)>0$ by Lemma \ref{conv_lem}. If $D\sim -E_1-2E_2+2H$, 
then $\eta(D)>0$ in the same way. 
Assume that $D\sim -3E_1-E_2+2H$. In this case, $\tau_1=2/3$, $X_2$ is the blowup 
of $Q$ along a point, 
$D_2\sim -E+2H$, where $E$ is the exceptional divisor of the morphism $X_2\to Q$ 
and $H$ corresponds to the pullback of $\sO_Q(1)$, $\tau_2=1$, $X_3=\pr^3$, 
$D_3\in|\sO(3)|$, $\tau_3=4/3$ and $m=3$. 
Since $\eta_2=125/324$ and $\eta_3=-1/36$, 
we have $\eta(D)/3>125/324-1/36>0$. If $D\sim -E_1-3E_2+2H$, then $\eta(D)>0$ 
in the same way. 
Hence $(X, -K_X)$ is divisorially semistable but not divisorially stable. 

\noindent\textbf{The case No.\ 20.}\,\,
Assume that $X$ belongs to No.\ 20 in \cite[Table 3]{MoMu}. 
Let $H$ be a divisor corresponds to the pullback of $\sO_Q(1)$. 
Let $E_1,\dots,E_3$ be the exceptional divisor of $\cont_{l_1},\dots,\cont_{l_3}$, 
respectively. 
Then $E_3\sim -E_1-E_2+H$ and 
\begin{eqnarray*}
\Nef(X)&=&\R_{\geq 0}[H]+\R_{\geq 0}[-E_1+H]+\R_{\geq 0}[-E_2+H], \\
\overline{\Eff}(X)&=&\R_{\geq 0}[E_1]+\R_{\geq 0}[E_2]
+\R_{\geq 0}[-E_1-E_2+H],\\
-K_X&\sim& -E_1-E_2+3H,\\
\Pic(X)&=&\Z[E_1]\oplus\Z[E_2]\oplus\Z[H].
\end{eqnarray*}
Hence $D\sim E_1$, $E_2$, $E_1+E_2$, $H$, $-E_1+H$, $-E_2+H$, $-E_1-E_2+H$, 
$-E_1-E_2+2H$, $-2E_1-E_2+2H$ or $-E_1-2E_2+2H$. 
Assume that $D\sim E_1$. In this case, $\sN_{D/X}\simeq\sO_{\F_1}(-\sigma_1)$, 
$\tau_1=1$, $X_2=W_6$, $D_2\in|\sO_{W_6}(1, 0)|$, 
$\tau_2=2$ and $m=2$. Thus $\eta(D)>0$ by Lemma \ref{three_lem}. 
If $D\sim E_2$, then $\eta(D)>0$ in the same way. 
Assume that $D\sim E_1+E_2$. In this case, $\tau_1=1/2$, $X_2=W_6$, 
$D_2\sim(-1/2)K_{W_6}$, $\tau_2=2$ and $m=2$. 
Since $\eta_2=27/32$, we have $\eta(D)>27/32>0$. 
Assume that $D\sim -E_1+H$. In this case, $D$ is nef, 
$\tau_1=1$, $X_2$ is the blowup of $Q$ along a line, $D_2$ corresponds to the 
pullback of $\sO_Q(1)$, $\tau_2=2$ and $m=2$. 
Thus $\eta(D)>0$ by Lemma \ref{conv_lem}. If $D\sim -E_2+H$, then $\eta(D)>0$ 
in the same way. 
Assume that $D\sim -E_1-E_2+H$, that is, $D=E_3$. In this case, 
$\sN_{D/X}\simeq\sO_{\pr^1\times\pr^1}(-1, 1)$, 
$\tau_1=1$, $X_2=Q$, $D_2\in|\sO_Q(1)|$, $\tau_2=3$ and $m=2$. 
Thus $\eta(D)>0$ by Lemma \ref{three_lem}. 
Assume that $D\sim -E_1-E_2+2H$. In this case, $D$ is nef, 
$\tau_1=1$, $X_2=Q$, $\tau_2=3/2$ and $m=2$. 
Thus $\eta(D)>0$ by Lemma \ref{conv_lem}. 
Assume that $D\sim -2E_1-E_2+2H$. In this case, 
$\tau_1=1/2$, $X_2$ is the blowup of $Q$ along a line, $D_2\sim -E+2H$, 
where $H$ corresponds to the pullback of $\sO_Q(1)$ and $E$ is the exceptional 
divisor of the morphism $X_2\to Q$, $\tau_2=1$, $X_3=Q$, $D_3\in|\sO_Q(2)|$, 
$\tau_3=3/2$ and $m=2$. Since $\eta_2=241/192$ and $\eta_3=-1/12$, we have 
$\eta(D)/3>241/192-1/12>0$. 
If $D\sim -E_1-2E_2+2H$, then $\eta(D)>0$ in the same way. 
Hence $(X, -K_X)$ is divisorially stable. 

\noindent\textbf{The case No.\ 21.}\,\,
Assume that $X$ belongs to No.\ 21 in \cite[Table 3]{MoMu}. 
We consider the case that $D$ is the strict transform of the divisor 
in $|\sO_{\pr^1\times\pr^2}(0, 1)|$ passing 
through the conic which is the center of the blowup. 
Then $\sN_{D/X}\simeq\sO_{\pr^1\times\pr^1}(-1, -1)$, 
$\tau_1=1$, $X_2=\pr^1\times\pr^2$, 
$D_2\in|\sO(0, 1)|$, $\tau_2=3$ and $m=2$. 
Thus $\eta(D)<0$ by Lemma \ref{three_lem}. 
Hence $(X, -K_X)$ is not divisorially semistable. 

\noindent\textbf{The case No.\ 22.}\,\,
Assume that $X$ belongs to No.\ 22 in \cite[Table 3]{MoMu}. 
By Lemma \ref{rhothree_lem}, $(X, -K_X)$ is not divisorially semistable. 

\noindent\textbf{The case No.\ 23.}\,\,
Assume that $X$ belongs to No.\ 23 in \cite[Table 3]{MoMu}. 
We consider the case that $D$ is the strict transform of the hyperplane 
in $\pr^3$ passing 
through the conic which is the center of the blowup. 
Then $\sN_{D/X}\simeq\sO_{\F_1}(-\sigma_1-f_1)$, 
$\tau_1=1$, $X_2=V_7$, 
$\sN_{D_2/X_2}\simeq\sO_{\F_1}(f_1)$, $\tau_2=2$, $X_3=\pr^3$, $D_3\in|\sO(1)|$, 
$\tau_3=4$ and $m=3$. 
Thus $\eta(D)/3<\eta_1+\eta_2<0$ by Lemma \ref{three_lem}. 
Hence $(X, -K_X)$ is not divisorially semistable. 

\noindent\textbf{The case No.\ 24.}\,\,
Assume that $X$ belongs to No.\ 24 in \cite[Table 3]{MoMu}. 
We consider the case that $D$ is the exceptional divisor of the morphism $X\to W_6$. 
Then $\sN_{D/X}\simeq\sO_{\pr^1\times\pr^1}(-1, 0)$, 
$\tau_1=1$, $X_2=\pr^1\times\pr^2$, $D_2\in|\sO(0, 1)|$, $\tau_2=3$ and $m=2$. 
Thus $\eta(D)<0$ by Lemma \ref{three_lem}. 
Hence $(X, -K_X)$ is not divisorially semistable. 

\noindent\textbf{The case No.\ 25--31.}\,\,
Assume that $X$ belongs to one of No.\ 25--31 in \cite[Table 3]{MoMu}. 
Then $X$ is toric. Then $(X, -K_X)$ is not divisorially stable by Corollary \ref{toric_cor}. 
Moreover, $(X, -K_X)$ is divisorially semistable if and only if $X$ belongs to either 
No.\ 25 or No.\ 27 in \cite[Table 3]{MoMu} by 
Theorem \ref{introtoric_thm} and \cite{mab}. 

Therefore we have proved Theorem \ref{three_thm} for the case $\rho(X)=3$.

\subsection{The case $\rho(X)=4$}\label{rho4_section}

We consider the case $\rho(X)=4$. We assume that $D$ is a suspicious divisor. 

\noindent\textbf{The case No.\ 1.}\,\,
Assume that $X$ belongs to No.\ 1 in \cite[Table 4]{MoMu}. 
Let $H_1,\dots,H_4$ be a divisor corresponds to the restriction of 
$\sO_{\pr^1\times\pr^1\times\pr^1\times\pr^1}(1, 0, 0, 0),\dots,
\sO_{\pr^1\times\pr^1\times\pr^1\times\pr^1}(0, 0, 0, 1)$, respectively. 
Let $E_1,\dots,E_4$ be the exceptional divisor of 
$\cont_{l_1},\dots,\cont_{l_4}$, respectively. 
Then $E_1\sim -H_1+H_2+H_3+H_4$, $E_2\sim H_1-H_2+H_3+H_4$, 
$E_3\sim H_1+H_2-H_3+H_4$, $E_4\sim H_1+H_2+H_3-H_4$ and 
\begin{eqnarray*}
\Nef(X)&=&\R_{\geq 0}[H_1]+\R_{\geq 0}[H_2]+\R_{\geq 0}[H_3]+\R_{\geq 0}[H_4], \\
\overline{\Eff}(X)&=&\R_{\geq 0}[H_1]+\R_{\geq 0}[H_2]
+\R_{\geq 0}[H_3]+\R_{\geq 0}[H_4]\\
&+&\R_{\geq 0}[-H_1+H_2+H_3+H_4]+\R_{\geq 0}[H_1-H_2+H_3+H_4]\\
&+&\R_{\geq 0}[H_1+H_2-H_3+H_4]+\R_{\geq 0}[H_1+H_2+H_3-H_4],\\
-K_X&\sim& H_1+H_2+H_3+H_4,\\
\Pic(X)&=&\Z[H_1]\oplus\Z[H_2]\oplus\Z[H_3]\oplus\Z[H_4].
\end{eqnarray*}
Hence $D\sim H_1$, $H_2$, $H_3$ or $H_4$. 
Assume that $D\sim H_1$. In this case, $\tau_1=1$, 
$X_2=\pr^1\times\pr^1\times\pr^1$, $D_2\in|\sO(1, 1, 1)|$, 
$\tau_2=2$ and $m=2$. Thus $\eta(D)>0$ by Lemma \ref{conv_lem}. 
If $D\sim H_2$, $H_3$ or $H_4$, then $\eta(D)>0$ in the same way. 
Hence $(X, -K_X)$ is divisorially stable. 

\noindent\textbf{The case No.\ 2.}\,\,
Assume that $X$ belongs to No.\ 2 in \cite[Table 4]{MoMu}. 
Let $H_1$, $H_2$ be a divisor corresponds to the pullback of 
$\sO_{\pr^1\times\pr^1}(1, 0)$, $\sO_{\pr^1\times\pr^1}(0, 1)$, respectively. 
Let $E_1$, $E_2$, $E_3$, $E_5$ be the exceptional divisor of 
$\cont_{l_1}$, $\cont_{l_2}$, $\cont_{l_3}$, $\cont_{l_5}$, respectively. 
Then $E_1\sim H_1+H_2-E_3+E_5$, $E_2\sim H_1+H_2+E_3-E_5$ and 
\begin{eqnarray*}
\Nef(X)&=&\R_{\geq 0}[H_1]+\R_{\geq 0}[H_2]+\R_{\geq 0}[H_1+H_2+E_3]
+\R_{\geq 0}[H_1+H_2+E_5]\\
&+&\R_{\geq 0}[H_1+H_2+E_3+E_5], \\
\overline{\Eff}(X)&=&\R_{\geq 0}[H_1]+\R_{\geq 0}[H_2]
+\R_{\geq 0}[E_3]+\R_{\geq 0}[E_5]\\
&+&\R_{\geq 0}[H_1+H_2+E_3-E_5]+\R_{\geq 0}[H_1+H_2-E_3+E_5],\\
-K_X&\sim& 2H_1+2H_2+E_3+E_5,\\
\Pic(X)&=&\Z[H_1]\oplus\Z[H_2]\oplus\Z[E_3]\oplus\Z[E_5].
\end{eqnarray*}
Hence $D\sim H_1$, $H_2$, $H_1+H_2$, $E_3$, $E_5$, $H_1+E_3$, $H_2+E_3$, 
$H_1+E_5$, $H_2+E_5$, $H_1+H_2+E_3$, $H_1+H_2+E_5$, 
$H_1+H_2+E_3-E_5$ or $H_1+H_2-E_3+E_5$. 
Assume that $D\sim H_1$. In this case, $\tau_1=1$, 
$X_2$ is the image of the morphism associated to the extremal face spanned by 
$\R_{\geq 0}[l_3]$ and $\R_{\geq 0}[l_5]$, $\sN_{D_2/X_2}$ is nonzero effective, 
$\tau_2=2$ and $m=2$. Thus $\eta(D)>0$ by Lemma \ref{conv_lem}. 
If $D\sim H_2$, then $\eta(D)>0$ in the same way. 
Assume that $D\sim H_1+H_2$. In this case, $\tau_1=1$, 
$X_2$ is the image of the morphism associated to the extremal face spanned by 
$\R_{\geq 0}[l_3],\dots,\R_{\geq 0}[l_6]$, $\tau_2=2$ and $m=2$. 
Thus $\eta(D)>0$ by Lemma \ref{conv_lem}. 
Assume that $D\sim E_3$. In this case, 
$\sN_{D/X}\simeq\sO_{\pr^1\times\pr^1}(-1, -1)$, 
$\tau_1=1$, $\sN_{D_2/X_2}\simeq\sO_{\pr^1\times\pr^1}(1, 1)$, 
$\tau_2=2$ and $m=2$. Thus $\eta(D)=0$ by the same calculation in 
Lemma \ref{rhothree_lem}. If $D\sim E_5$, then $\eta(D)=0$ in the same way. 
Assume that $D\sim H_1+E_3$. In this case, $\tau_1=1$, 
$X_2=\pr_{\pr^1}(\sO\oplus\sO(1)^{\oplus 2})$, 
$D_2\in|\xi_\pr^{\otimes 2}\otimes H_{\pr^1}^{\otimes (-1)}|$, 
$\tau_2=3/2$ and $m=2$. Since $\eta_1=47/12$ and $\eta_2=-13/48$, we have 
$\eta(D)/3>0$. If $D\sim H_1+E_5$, $H_2+E_3$ or $H_2+E_5$, then $\eta(D)>0$ 
in the same way. 
Assume that $D\sim H_1+H_2+E_3$. In this case, $D$ is nef, $\tau_1=1$, 
$X_2$ is the projective cone of a quadric hypersurface in $\pr^3$, 
$\tau_2=3/2$ and $m=2$. Thus $\eta(D)>0$ by Lemma \ref{conv_lem}. 
If $D\sim H_1+H_2+E_5$, then $\eta(D)>0$ in the same way. 
Assume that $D\sim H_1+H_2+E_3-E_5$, that is, $D=E_2$. 
In this case, $\tau_1=1/2$, $\tau_2=1$, 
$X_3$ is the projective cone of a quadric hypersurface in $\pr^3$, 
$D_3\sim(-2/3)K_{X_3}$, $\tau_3=3/2$ and $m=3$. 
Since $\eta_1=67/16$ and $\eta_3=-1/12$, we have 
$\eta(D)/3>67/16-1/12>0$. 
If $D\sim H_1+H_2-E_3+E_5$, then $\eta(D)>0$ in the same way. 
Hence $(X, -K_X)$ is divisorially semistable but not divisorially stable. 

\noindent\textbf{The case No.\ 3.}\,\,
Assume that $X$ belongs to No.\ 3 in \cite[Table 4]{MoMu}. 
Let $H_1,\dots,H_3$ be a divisor corresponds to the pullback of 
$\sO_{\pr^1\times\pr^1\times\pr^1}(1, 0, 0),\dots,
\sO_{\pr^1\times\pr^1\times\pr^1}(0, 0, 1)$, respectively. 
Let $E_1,\dots,E_4$ be the exceptional divisor of 
$\cont_{l_1},\dots,\cont_{l_4}$, respectively. 
Then $E_1\sim 2H_2+H_3-E_4$, $E_2\sim 2H_1+H_3-E_4$, 
$E_3\sim H_1+H_2-E_4$ and 
\begin{eqnarray*}
\Nef(X)&=&\R_{\geq 0}[H_1]+\R_{\geq 0}[H_2]+\R_{\geq 0}[H_3]
+\R_{\geq 0}[H_1+H_2+H_3-E_4], \\
\overline{\Eff}(X)&=&\R_{\geq 0}[H_1]+\R_{\geq 0}[H_2]
+\R_{\geq 0}[H_3]+\R_{\geq 0}[H_1+H_2-E_4]\\
&+&\R_{\geq 0}[2H_1+H_3-E_4]+\R_{\geq 0}[2H_2+H_3-E_4]+\R_{\geq 0}[E_4],\\
-K_X&\sim& 2H_1+2H_2+2H_3-E_4,\\
\Pic(X)&=&\Z[H_1]\oplus\Z[H_2]\oplus\Z[H_3]\oplus\Z[E_4].
\end{eqnarray*}
Hence $D\sim H_1$, $H_2$, $H_3$, $H_1+H_3$, $H_2+H_3$, $H_1+H_2-E_4$ 
or $H_1+H_2+H_3-E_4$. 
Assume that $D\sim H_1$. In this case, $\tau_1=1$, 
$X_2=\pr^1\times\F_1$, $D_2\in|\sO_{\pr^1\times\F_1}(1, \sigma_1+f_1)|$, 
$\tau_2=2$ and $m=2$. Thus $\eta(D)>0$ by Lemma \ref{conv_lem}. 
If $D\sim H_2$, then $\eta(D)>0$ in the same way. 
Assume that $D\sim H_3$. In this case, $\tau_1=1$, 
$X_2\in|\sO_{\pr^1\times\pr^1\times\pr^2}(1, 1, 2)|$, 
$D_2\in|\sO_{X_2}(0, 0, 1)|$, 
$\tau_2=2$ and $m=2$. Thus $\eta(D)>0$ by Lemma \ref{conv_lem}. 
Assume that $D\sim H_1+H_3$. In this case, $\tau_1=1$, 
$X_2=\pr^1\times\pr^2$, $D_2\in|\sO(1, 2)|$, 
$\tau_2=3/2$ and $m=2$. Thus $\eta(D)>0$ by Lemma \ref{conv_lem}. 
If $D\sim H_2+H_3$, then $\eta(D)>0$ in the same way. 
Assume that $D\sim H_1+H_2-E_4$, that is, $D=E_3$. In this case, 
$\sN_{D/X}\simeq\sO_{\pr^1\times\pr^1}(-1, 0)$, $\tau_1=1$, 
$X_2=\pr^1\times\pr^1\times\pr^1$, $D_2\in|\sO(1, 1, 0)|$, 
$\tau_2=2$ and $m=2$. Thus $\eta(D)>0$ by Lemma \ref{three_lem}. 
Assume that $D\sim H_1+H_2+H_3-E_4$. In this case, $D$ is nef, $\tau_1=1$, 
$X_2=\pr^1\times\pr^1\times\pr^1$, $D_2\in|\sO(1, 1, 1)|$, 
$\tau_2=2$ and $m=2$. Thus $\eta(D)>0$ by Lemma \ref{conv_lem}. 
Hence $(X, -K_X)$ is divisorially stable. 

\noindent\textbf{The case No.\ 4.}\,\,
Assume that $X$ belongs to No.\ 4 in \cite[Table 4]{MoMu}. 
Let $H$ be a divisor corresponds to the pullback of 
$\sO_Q(1)$. Let $E_1,\dots,E_5$ be the exceptional divisor of 
$\cont_{l_1},\dots,\cont_{l_5}$, respectively. 
Then $E_3\sim -2E_1+H$, $E_4\sim -2E_2+H$ and 
\begin{eqnarray*}
\Nef(X)&=&\R_{\geq 0}[H]+\R_{\geq 0}[-E_1+H]+\R_{\geq 0}[-E_2+H]\\
&+&\R_{\geq 0}[-E_1-E_2+H]+\R_{\geq 0}[-E_1-E_2+H-E_5], \\
\overline{\Eff}(X)&=&\R_{\geq 0}[E_1]+\R_{\geq 0}[E_2]
+\R_{\geq 0}[-2E_1+H]+\R_{\geq 0}[-2E_2+H]\\
&+&\R_{\geq 0}[-E_1-E_2+H-E_5]+\R_{\geq 0}[E_5],\\
-K_X&\sim& -2E_1-2E_2+3H-E_5,\\
\Pic(X)&=&\Z[E_1]\oplus\Z[E_2]\oplus\Z[H]\oplus\Z[E_5].
\end{eqnarray*}
Hence $D\sim -2E_1-E_2+2H-2E_5$, $-E_1-2E_2+2H-2E_5$, 
$-E_1+H-E_5$, $-E_2+H-E_5$, $-E_1-E_2+H-E_5$, $-2E_1-E_2+2H-E_5$, 
$-E_1-2E_2+2H-E_5$, $-3E_1-E_2+2H-E_5$, $-E_1-3E_2+2H-E_5$, $-2E_1-2E_2+2H-E_5$, 
$E_1$, $E_2$, $-E_1+H$, $-E_2+H$, $-2E_1+H$, $-2E_2+H$, $-E_1-E_2+H$, 
$E_5$, $E_1+E_5$ or $E_2+E_5$. 
Assume that $D\sim -2E_1-E_2+2H-2E_5$. 
In this case, $\tau_1=1/2$, 
$X_2$ is the blowup of $Q$ along general two points, 
$D_2\sim -2E_1-E_2+2H$, where $E_1$, $E_2$ are the exceptional divisors of the 
morphism $X_2\to Q$, $H$ corresponds to the pullback of $\sO_Q(1)$, 
$\tau_2=1$, $X_3=\pr^3$, $D_3\in|\sO(3)|$, $\tau_3=4/3$ and $m=3$. 
Since $\eta_1=991/192$ and $\eta_3=-1/36$, we have 
$\eta(D)/3>991/192-1/36>0$. 
If $D\sim -E_1-2E_2+2H-2E_5$, then $\eta(D)>0$ in the same way. 
Assume that $D\sim -E_1+H-E_5$. In this case, $\tau_1=1$, 
$X_2=V_7$, $D_2\sim(-1/2)K_{V_7}$, $\tau_2=2$ and $m=2$. 
Since $\eta_1=13/3$ and $\eta_2=-7/12$, we have $\eta(D)/3>0$. 
If $D\sim -E_2+H-E_5$, then $\eta(D)>0$ in the same way. 
Assume that $D\sim -E_1-E_2+H-E_5$. In this case, $\tau_1=1$, 
$X_2$ is the blowup of $Q$ along general two points, 
$D_2\sim -E_1-E_2+H$, where $E_1$, $E_2$ are the exceptional divisors of the 
morphism $X_2\to Q$, $H$ corresponds to the pullback of $\sO_Q(1)$, 
$\tau_2=2$, $X_3=Q$, $D_3\in|\sO(1)|$, $\tau_3=3$ and $m=3$. 
Since $\eta_1=3$, $\eta_2=-5/3$ and $\eta_3=-5/6$, we have $\eta(D)/3>0$. 
Assume that $D\sim -2E_1-E_2+2H-E_5$. In this case, $D$ is nef, $\tau_1=1$, 
$X_2=\pr^3$, $D_2\in|\sO(3)|$, $\tau_2=4/3$ and $m=2$. 
Thus $\eta(D)>0$ by Lemma \ref{conv_lem}. 
If $D\sim -E_1-2E_2+2H-E_5$, then $\eta(D)>0$ in the same way. 
Assume that $D\sim -3E_1-E_2+2H-E_5$. In this case, $\tau_1=1/2$, 
$X_2$ is the blowup of $\pr^3$ along a line and a conic, 
$\tau_2=1$, $X_3=\pr^3$, $D_3\in|\sO(3)|$, $\tau_3=4/3$ and $m=3$. 
Since $\eta_1=257/48$ and $\eta_3=-1/36$, we have $\eta(D)/3>257/48-1/36>0$. 
If $D\sim -E_1-3E_2+2H-E_5$, then $\eta(D)>0$ in the same way. 
Assume that $D\sim -2E_1-2E_2+2H-E_5$. In this case, $D$ is nef, $\tau_1=1$, 
$X_2=Q$, $D_2\in|\sO(2)|$, $\tau_2=3/2$ and $m=2$. 
Thus $\eta(D)>0$ by Lemma \ref{conv_lem}. 
Assume that $D\sim E_1$. In this case, $\sN_{D/X}\simeq\sO_{\F_1}(-\sigma_1-f_1)$, 
$\tau_1=1$, $X_2$ belongs to No.\ 30 in \cite[Table 3]{MoMu}, 
$\sN_{D_2/X_2}\simeq\sO_{\F_1}(\sigma_1+f_1)$, $\tau_2=2$ and $m=2$. 
Thus $\eta(D)=0$ by Lemma \ref{three_lem}. 
If $D\sim E_2$, then $\eta(D)=0$ in the same way. 
Assume that $D\sim -E_1+H$. In this case, $D$ is nef, $\tau_1=1$, 
$X_2=\pr_{\pr^1}(\sO^{\oplus 2}\oplus\sO(1))$, $D_2\in|\xi_\pr^{\otimes 2}|$, 
$\tau_2=3/2$ and $m=2$. Thus $\eta(D)>0$ by Lemma \ref{conv_lem}. 
If $D\sim -E_2+H$, then $\eta(D)>0$ in the same way. 
Assume that $D\sim -2E_1+H$. In this case, $\tau_1=1/2$, 
$X_2$ is the blowup of $\pr^3$ along a line and a conic, 
$\tau_2=1$, $X_3=\pr_{\pr^1}(\sO^{\oplus 2}\oplus\sO(1))$, 
$D_3\in|\xi_\pr^{\otimes 2}|$, $\tau_3=3/2$ and $m=3$. 
Since $\eta_1=377/96$ and $\eta_3=-5/24$, we have $\eta(D)/3>377/96-5/24>0$. 
If $D\sim -2E_2+H$, then $\eta(D)>0$ in the same way. 
Assume that $D\sim -E_1-E_2+H$. In this case, $D$ is nef, $\tau_1=1$, 
$X_2$ is the blowup of $Q$ along a conic, $D_2$ corresponds to the pullback of 
$\sO_Q(1)$, $\tau_2=2$ and $m=2$. Thus $\eta(D)>0$ by Lemma \ref{conv_lem}. 
Assume that $D\sim E_1+E_5$. In this case, $\tau_1=1$, 
$X_2=\pr_{\pr^1}(\sO^{\oplus 2}\oplus\sO(1))$, 
$D_2\in|\xi_\pr^{\otimes 2}\otimes H_{\pr^1}^{\otimes (-1)}|$, $\tau_2=3/2$ and $m=2$. 
Since $\eta_1=4$ and $\eta_2=-19/48$, we have $\eta(D)/3>0$. 
If $D\sim E_2+E_5$, then $\eta(D)>0$ in the same way. 
Assume that $D\sim E_5$. In this case, 
$\sN_{D/X}\simeq\sO_{\pr^1\times\pr^1}(-1, 0)$, $\tau_1=1$, 
$X_2$ is the blowup of $Q$ along a conic, $D_2$ is the exceptional divisor 
of the morphism $X_2\to Q$, $\sN_{D_2/X_2}\simeq\sO_{\pr^1\times\pr^1}(-1, 2)$, 
$\tau_2=2$ and $m=2$. Thus $\eta(D)>0$ by Lemma \ref{three_lem}. 
Hence $(X, -K_X)$ is divisorially semistable but not divisorially stable. 

\noindent\textbf{The case No.\ 5.}\,\,
Assume that $X$ belongs to No.\ 5 in \cite[Table 4]{MoMu}. 
We consider the case that $D$ is the strict transform of the divisor 
in $|\sO_{\pr^1\times\pr^2}(0, 1)|$ passing through the conic which is one of 
the center of the blowup $X\to\pr^1\times\pr^2$. 
Then $\sN_{D/X}\simeq\sO_{\pr^1\times\pr^1}(-1, -1)$, 
$\tau_1=1$, $X_2=\pr^1\times\pr^2$, $D_2\in|\sO(0, 1)|$, $\tau_2=3$ and $m=2$. 
Thus $\eta(D)<0$ by Lemma \ref{three_lem}. 
Hence $(X, -K_X)$ is not divisorially semistable. 

\noindent\textbf{The case No.\ 6.}\,\,
Assume that $X$ belongs to No.\ 6 in \cite[Table 4]{MoMu}. 
Let $H_1,\dots,H_3$ be a divisor corresponds to the pullback of 
$\sO_{\pr^1\times\pr^1\times\pr^1}(1, 0, 0),\dots,
\sO_{\pr^1\times\pr^1\times\pr^1}(0, 0, 1)$, respectively. 
Let $E_1,\dots,E_4$ be the exceptional divisor of 
$\cont_{l_1},\dots,\cont_{l_4}$, respectively. 
Then $E_1\sim H_2+H_3-E_4$, $E_2\sim H_1+H_3-E_4$, 
$E_3\sim H_1+H_2-E_4$ and 
\begin{eqnarray*}
\Nef(X)&=&\R_{\geq 0}[H_1]+\R_{\geq 0}[H_2]+\R_{\geq 0}[H_3]
+\R_{\geq 0}[H_1+H_2+H_3-E_4], \\
\overline{\Eff}(X)&=&\R_{\geq 0}[H_1]+\R_{\geq 0}[H_2]
+\R_{\geq 0}[H_3]+\R_{\geq 0}[H_1+H_2-E_4]+\R_{\geq 0}[H_1+H_3-E_4]\\
&+&\R_{\geq 0}[H_2+H_3-E_4]+\R_{\geq 0}[E_4],\\
-K_X&\sim& 2H_1+2H_2+2H_3-E_4,\\
\Pic(X)&=&\Z[H_1]\oplus\Z[H_2]\oplus\Z[H_3]\oplus\Z[E_4].
\end{eqnarray*}
Hence $D\sim H_1+H_2-E_4$, $H_1+H_3-E_4$, $H_2+H_3-E_4$, $H_1+H_2+H_3-E_4$, 
$H_1$, $H_2$, $H_3$, $H_1+H_2$, $H_1+H_3$, $H_2+H_3$, $H_1+H_2+H_3$, 
$E_4$, $H_1+E_4$, $H_2+E_4$ or $H_3+E_4$. 
Assume that $D\sim H_1+H_2-E_4$, that is, $D=E_3$. In this case, 
$\sN_{D/X}\simeq\sO_{\pr^1\times\pr^1}(-1, 1)$, $\tau_1=1$, 
$X_2=\pr^1\times\pr^1\times\pr^1$, $D_2\in|\sO(1, 1, 0)|$, 
$\tau_2=2$ and $m=2$. Thus $\eta(D)>0$ by Lemma \ref{three_lem}. 
If $D\sim H_1+H_3-E_4$ or $H_2+H_3-E_4$, then $\eta(D)>0$ in the same way. 
Assume that $D\sim H_1+H_2+H_3-E_4$. In this case, $D$ is nef, 
$\tau_1=1$, $X_2=\pr^1\times\pr^1\times\pr^1$, $D_2\in|\sO(1, 1, 1)|$, 
$\tau_2=2$ and $m=2$. Thus $\eta(D)>0$ by Lemma \ref{conv_lem}. 
Assume that $D\sim H_1$. In this case, $\tau_1=1$, 
$X_2=\pr_{\pr^1\times\pr^1}(\sO(1, 0)\oplus\sO(0, 1))$, $D_2$ corresponds to 
the pullback of $\sO_{\pr^3}(1)$, $\tau_2=2$ and $m=2$. 
Thus $\eta(D)>0$ by Lemma \ref{conv_lem}. 
If $D\sim H_2$ or $H_3$, then $\eta(D)>0$ in the same way. 
Assume that $D\sim H_1+H_2$. In this case, $\tau_1=1$, 
$X_2=\pr_{\pr^1}(\sO^{\oplus 2}\oplus\sO(1))$, $D_2\in|\xi_\pr^{\otimes 2}|$, 
$\tau_2=3/2$ and $m=2$. 
Thus $\eta(D)>0$ by Lemma \ref{conv_lem}. 
If $D\sim H_1+H_3$ or $H_2+H_3$, then $\eta(D)>0$ in the same way. 
Assume that $D\sim H_1+H_2+H_3$. In this case, $\tau_1=1$, 
$X_2=\pr^3$, $D_2\in|\sO(3)|$, $\tau_2=4/3$ and $m=2$. 
Thus $\eta(D)>0$ by Lemma \ref{conv_lem}. 
Assume that $D\sim E_4$. In this case, 
$\sN_{D/X}\simeq\sO_{\pr^1\times\pr^1}(-1, 2)$, $\tau_1=1$, 
$X_2=\pr^3$, $D_2\in|\sO(2)|$, 
$\tau_2=2$ and $m=2$. Thus $\eta(D)>0$ by Lemma \ref{three_lem}. 
Assume that $D\sim H_1+E_4$. In this case, $\tau_1=1/2$, 
$X_2=\pr_{\pr^1\times\pr^1}(\sO(1, 0)\oplus\sO(0, 1))$, 
$\tau_2=1$, $X_3=\pr^3$, $D_3\in|\sO(3)|$, $\tau_3=4/3$ and $m=3$. 
Since $\eta_1=1045/192$ and $\eta_3=-1/36$, we have $\eta(D)/3>1045/192-1/36>0$. 
If $D\sim H_2+E_4$ or $H_3+E_4$, then $\eta(D)>0$ in the same way. 
Hence $(X, -K_X)$ is divisorially stable. 

\noindent\textbf{The case No.\ 7.}\,\,
Assume that $X$ belongs to No.\ 7 in \cite[Table 4]{MoMu}. 
Let $H_1$, $H_2$ be a divisor corresponds to the pullback of 
$\sO_{W_6}(1, 0)$, $\sO_{W_6}(0, 1)$, respectively. 
Let $E_1,\dots,E_4$ be the exceptional divisor of 
$\cont_{l_1},\dots,\cont_{l_4}$, respectively. 
Then $E_3\sim -E_1+H_2$, $E_4\sim -E_2+H_1$ and 
\begin{eqnarray*}
\Nef(X)&=&\R_{\geq 0}[H_1]+\R_{\geq 0}[H_2]+\R_{\geq 0}[-E_1+H_1]
+\R_{\geq 0}[-E_2+H_2], \\
\overline{\Eff}(X)&=&\R_{\geq 0}[E_1]+\R_{\geq 0}[E_2]
+\R_{\geq 0}[-E_1+H_1]+\R_{\geq 0}[-E_2+H_1]\\
&+&\R_{\geq 0}[-E_1+H_2]+\R_{\geq 0}[-E_2+H_2],\\
-K_X&\sim& -E_1-E_2+2H_1+2H_2,\\
\Pic(X)&=&\Z[E_1]\oplus\Z[E_2]\oplus\Z[H_1]\oplus\Z[H_2].
\end{eqnarray*}
Hence $D\sim E_1$, $E_2$, $H_1$, $H_2$, $-E_1+H_1$, $-E_2+H_2$, 
$-E_1+H_2$, $-E_2+H_1$, $E_1-E_2+H_1$, $-E_1+E_2+H_2$, $-E_1+E_2+H_1$, 
$E_1-E_2+H_2$, $-E_1+H_1+H_2$, $-E_2+H_1+H_2$, $-2E_1+H_1+H_2$, 
$-2E_2+H_1+H_2$ or $-E_1-E_2+H_1+H_2$. 
Assume that $D\sim E_1$. In this case, 
$\sN_{D/X}\simeq\sO_{\pr^1\times\pr^1}(-1, 0)$, $\tau_1=1$, 
$X_2=\pr^1\times\F_1$, $D_2\in|\sO_{\pr^1\times\F_1}(0, \sigma_1+f_1)|$, 
$\tau_2=2$ and $m=2$. Thus $\eta(D)=0$ by Lemma \ref{three_lem}. 
If $D\sim E_2$, then $\eta(D)=0$ in the same way. 
Assume that $D\sim H_1$. In this case, $\tau_1=1$, 
$X_2=\pr^1\times\F_1$, $D_2\in|\sO_{\pr^1\times\F_1}(1, \sigma_1+f_1)|$, 
$\tau_2=2$ and $m=2$. Thus $\eta(D)>0$ by Lemma \ref{conv_lem}. 
If $D\sim H_2$, then $\eta(D)>0$ in the same way. 
Assume that $D\sim -E_1+H_1$. In this case, $D$ is nef, $\tau_1=1$, 
$X_2=W_6\times_{\pr^2}\F_1$, $D_2$ corresponds to 
the pullback of $\sO_{W_6}(1, 0)$, $\tau_2=2$ and $m=2$. 
Thus $\eta(D)>0$ by Lemma \ref{conv_lem}. 
If $D\sim -E_2+H_2$, then $\eta(D)>0$ in the same way. 
Assume that $D\sim -E_1+H_2$, that is, $D=E_3$. In this case, 
$\sN_{D/X}\simeq\sO_{\F_1}(-\sigma_1+f_1)$, $\tau_1=1$, 
$X_2=\pr^1\times\pr^2$, $D_2\in|\sO(1, 1)|$, 
$\tau_2=2$ and $m=2$. Thus $\eta(D)>0$ by Lemma \ref{three_lem}. 
If $D\sim -E_2+H_1$, then $\eta(D)>0$ in the same way. 
Assume that $D\sim E_1-E_2+H_1$. In this case, $\tau_1=1/2$, 
$X_2=\pr^1\times\F_1$, $D_2\in|\sO_{\pr^1\times\F_1}(1, \sigma_1+2f_1)|$, 
$\tau_2=1$, $X_3=\pr^1\times\pr^2$, $D_3\in|\sO(1, 2)|$, $\tau_3=3/2$ and $m=3$. 
Since $\eta_2=259/192$ and $\eta_3=-7/48$, we have 
$\eta(D)/3>259/192-7/48>0$. 
If $D\sim -E_1+E_2+H_2$, then $\eta(D)>0$ in the same way. 
Assume that $D\sim -E_1+E_2+H_1$. In this case, $\tau_1=1$, 
$X_2=\pr^1\times\pr^2$, $D_2\in|\sO(0, 2)|$, 
$\tau_2=3/2$ and $m=2$. 
Since $\eta_1=21/4$ and $\eta_2=-1/3$, we have $\eta(D)/3>0$. 
If $D\sim E_1-E_2+H_2$, then $\eta(D)>0$ in the same way. 
Assume that $D\sim -E_1+H_1+H_2$. In this case, $D$ is nef, $\tau_1=1$, 
$X_2=\pr^1\times\pr^2$, $D_2\in|\sO(1, 2)|$, 
$\tau_2=3/2$ and $m=2$. 
Thus $\eta(D)>0$ by Lemma \ref{conv_lem}. 
If $D\sim -E_2+H_1+H_2$, then $\eta(D)>0$ in the same way. 
Assume that $D\sim -2E_1+H_1+H_2$. In this case, $\tau_1=1$, 
$X_2=W_6\times_{\pr^2}\F_1$, 
$\tau_2=1$, $X_3=\pr^1\times\pr^2$, $D_3\in|\sO(1, 2)|$, $\tau_3=3/2$ and $m=3$. 
Since $\eta_1=163/32$ and $\eta_3=-7/48$, we have $\eta(D)/3>163/32-7/48>0$. 
If $D\sim -2E_2+H_1+H_2$, then $\eta(D)>0$ in the same way. 
Assume that $D\sim -E_1-E_2+H_1+H_2$. In this case, $D$ is nef, $\tau_1=1$, 
$X_2=W_6$, $D_2\sim(-1/2)K_{W_6}$, $\tau_2=2$ and $m=2$. 
Thus $\eta(D)>0$ by Lemma \ref{conv_lem}. 
Hence $(X, -K_X)$ is divisorially semistable but not divisorially stable. 

\noindent\textbf{The case No.\ 8.}\,\,
Assume that $X$ belongs to No.\ 8 in \cite[Table 4]{MoMu}. 
We consider the case that $D$ is the strict transform of the divisor 
in $|\sO_{\pr^1\times\pr^1\times\pr^1}(1, 0, 0)|$ 
passing through the center of the blowup $X\to\pr^1\times\pr^1\times\pr^1$. 
Then $\sN_{D/X}\simeq\sO_{\pr^1\times\pr^1}(-1, -1)$, 
$\tau_1=1$, $X_2=\pr^1\times\pr^1\times\pr^1$, $D_2\in|\sO(1, 0, 0)|$, 
$\tau_2=2$ and $m=2$. Thus $\eta(D)<0$ by Lemma \ref{conv_lem}. 
Hence $(X, -K_X)$ is not divisorially semistable. 

\noindent\textbf{The case No.\ 9--12.}\,\,
Assume that $X$ belongs to one of No.\ 9--12 in \cite[Table 4]{MoMu}. 
Then $X$ is toric. Then $(X, -K_X)$ is not divisorially semistable by 
Theorem \ref{introtoric_thm} and \cite{mab}.

\noindent\textbf{The case No.\ 13.}\,\,
Assume that $X$ belongs to No.\ 13 in \cite[Table 4]{MoMu}. 
Let $l_1$, $l_2$, $l_3\subset X$ be the strict transform of a curve on 
$\pr^1\times\pr^1\times\pr^1$ of tridegree $(1, 0, 0)$, $(0, 1, 0)$, $(0, 0, 1)$, 
passing through the center of the blowup 
$X\to\pr^1\times\pr^1\times\pr^1$, respectively. 
Let $l_4\subset X$ be the strict transform of a curve on 
$\pr^1\times\pr^1\times\pr^1$ of tridegree $(1, 1, 0)$ 
which is contained in the divisor in $|\sO_{\pr^1\times\pr^1\times\pr^1}(1, 1, 0)|$
which contains the center of the blowup $X\to\pr^1\times\pr^1\times\pr^1$. 
Let $l_5\subset X$ be an exceptional curve of the blowup 
$X\to\pr^1\times\pr^1\times\pr^1$. 
Then $\overline{\NE}(X)$ is spanned by the classes of $l_1,\dots,l_5$. 
Let $H_1,\dots,H_3$ be a divisor corresponds to the pullback of 
$\sO_{\pr^1\times\pr^1\times\pr^1}(1, 0, 0),\dots,\sO_{\pr^1\times\pr^1\times\pr^1}
(0, 0, 1)$, respectively. 
Let $E_1$, $E_2$, $E_3$, $E_5$ be the exceptional divisor of 
$\cont_{l_1}$, $\cont_{l_2}$, $\cont_{l_3}$, $\cont_{l_5}$, respectively. 
Then $E_1\sim 3H_2+H_3-E_5$, $E_2\sim 3H_1+H_3-E_5$, $E_3\sim H_1+H_2-E_5$ and 
\begin{eqnarray*}
\Nef(X)&=&\R_{\geq 0}[H_1]+\R_{\geq 0}[H_2]+\R_{\geq 0}[H_3]\\
&+&\R_{\geq 0}[2H_1+H_2+H_3-E_5]+\R_{\geq 0}[H_1+2H_2+H_3-E_5], \\
\overline{\Eff}(X)&=&\R_{\geq 0}[H_1]+\R_{\geq 0}[H_2]+\R_{\geq 0}[H_3]
+\R_{\geq 0}[H_1+H_2-E_5]\\
&+&\R_{\geq 0}[3H_1+H_3-E_5]+\R_{\geq 0}[3H_2+H_3-E_5]+\R_{\geq 0}[E_5],\\
-K_X&\sim& 2H_1+2H_2+2H_3-E_5,\\
\Pic(X)&=&\Z[H_1]\oplus\Z[H_2]\oplus\Z[H_3]\oplus\Z[E_5].
\end{eqnarray*}
Hence $D\sim H_1$, $H_2$, $H_3$, 
$H_1+H_3$, $H_2+H_3$, $H_1+H_2-E_5$ or $H_1+H_2+H_3-E_5$. 
Assume that $D\sim H_1$. In this case, $\tau_1=1$, 
$X_2=\pr_{\pr^1}(\sO\oplus\sO(1)^{\oplus 2})$, $D_2\in|\xi_\pr^{\otimes 2}|$, 
$\tau_2=3/2$ and $m=2$. Thus $\eta(D)>0$ by Lemma \ref{conv_lem}. 
If $D\sim H_2$, then $\eta(D)>0$ in the same way. 
Assume that $D\sim H_3$. In this case, $\tau_1=1$, 
$X_2$ is the image of the morphism $\cont_{l_3}$. 
If we see $\ND(X_2)$ as a subspace of 
$\ND(X)$, then $-K_{X_2}\sim 3H_1+3H_2+2H_3-2E_5$ and $D_2\sim H_1+H_2+H_3-E_5$. 
Moreover, $\Nef(X_2)$ is spanned by the classes of 
$H_1$, $H_2$, $2H_1+H_2+H_3-E_5$ and $H_1+2H_2+H_3-E_5$. 
Since $\tau(D)=2$
 and $-K_{X_2}-2D_2\in\Nef(X_2)$, we have $\tau_2=2$ and $m=2$. 
Since $\sN_{D_2/X_2}$ is nonzero effective, we have 
$\eta(D)>0$ by Lemma \ref{conv_lem}. 
Assume that $D\sim H_1+H_3$. In this case, $\tau_1=1$, 
$X_2$ is the image of the contraction of the negative section of 
$\pr_{\pr^1}(\sO\oplus\sO(1)^{\oplus 2})$, 
$D_2\sim(-2/3)K_{X_2}$, $\tau_2=3/2$ and $m=2$. 
Thus $\eta(D)>0$ by Lemma \ref{conv_lem}. 
If $D\sim H_2+H_3$, then $\eta(D)>0$ in the same way. 
Assume that $D\sim H_1+H_2-E_5$, that is, $D=E_3$. In this case, 
$\sN_{D/X}\simeq\sO_{\pr^1\times\pr^1}(-1, -1)$, $\tau_1=1$, 
$X_2=\pr^1\times\pr^1\times\pr^1$, $D_2\in|\sO(1, 1, 0)|$, 
$\tau_2=2$ and $m=2$. Thus $\eta(D)>0$ by Lemma \ref{three_lem}. 
Assume that $D\sim H_1+H_2+H_3-E_5$. In this case, $\tau_1=1$, 
$X_2=\pr^1\times\pr^1\times\pr^1$, $D_2\in|\sO(1, 1, 1)|$, 
$\tau_2=2$ and $m=2$. 
Since $\eta_1=41/12$ and $\eta_2=-1/2$, we have $\eta(D)/3>0$. 
Hence $(X, -K_X)$ is divisorially stable. 

Therefore we have proved Theorem \ref{three_thm} for the case $\rho(X)=4$.

\subsection{The case $\rho(X)=5$}\label{rho5_section}

We consider the case $\rho(X)=5$. We assume that $D$ is a suspicious divisor. 

\noindent\textbf{The case No.\ 1.}\,\,
Assume that $X$ belongs to No.\ 1 in \cite[Table 5]{MoMu}. 
Let $E_7\subset X$ be the prime divisor such that the center on $Q$ is a conic. 
In \cite{mat}, $l_7$ is a fiber of the ruling $E_7\simeq\pr^1\times\pr^1\to\pr^1$. 
Let $l_8\subset X$ be a fiber of the other ruling $E_7\simeq\pr^1\times\pr^1\to\pr^1$. 
Then $\overline{\NE}(X)$ is spanned by the classes of $l_1\dots, l_8$ 
(in \cite{mat}, the ray $\R_{\geq 0}[l_8]$ is forgotten). 
Let $H$ be a divisor corresponds to the pullback of 
$\sO_Q(1)$. 
Let $E_1,\dots,E_6$ be the exceptional divisor of 
$\cont_{l_1},\dots,\cont_{l_6}$, respectively. 
Then $E_4\sim -2E_1+H$, $E_5\sim -2E_2+H$, 
$E_6\sim -2E_3+H$ and 
\begin{eqnarray*}
\Nef(X)&=&\R_{\geq 0}[H]+\R_{\geq 0}[-E_1+H]+\R_{\geq 0}[-E_2+H]
+\R_{\geq 0}[-E_1-E_2+H]\\
&+&\R_{\geq 0}[-E_3+H]+\R_{\geq 0}[-E_1-E_3+H]+\R_{\geq 0}[-E_2-E_3+H]\\
&+&\R_{\geq 0}[-E_1-E_2-E_3-E_7+H], \\
\overline{\Eff}(X)&=&\R_{\geq 0}[E_1]+\R_{\geq 0}[E_2]
+\R_{\geq 0}[E_3]+\R_{\geq 0}[E_7]+\R_{\geq 0}[-2E_1+H]\\
&+&\R_{\geq 0}[-2E_2+H]+\R_{\geq 0}[-2E_3+H]+\R_{\geq 0}[-E_1-E_2-E_3-E_7+H],\\
-K_X&\sim& -2E_1-2E_2-2E_3-E_7+3H,\\
\Pic(X)&=&\Z[E_1]\oplus\Z[E_2]\oplus\Z[E_3]\oplus\Z[E_7]\oplus\Z[H].
\end{eqnarray*}
Hence $D\sim E_1$, $E_2$, $E_3$, $E_7$, $E_1+E_7$, $E_2+E_7$, $E_3+E_7$, 
$-E_2-E_3-E_7+H$, $-E_1-E_3-E_7+H$, $-E_1-E_2-E_7+H$, 
$-E_1-E_2-E_3-E_7+H$, $-E_2-E_3+H$, $-E_1-E_3+H$, $-E_1-E_2+H$, 
$-E_1-E_2-E_3+H$, $-E_1-2E_2-2E_3-2E_7+2H$, $-2E_1-E_2-2E_3-2E_7+2H$, 
$-2E_1-2E_2-E_3-2E_7+2H$, $-E_1-2E_2-2E_3-E_7+2H$, 
$-2E_1-E_2-2E_3-E_7+2H$, $-2E_1-2E_2-E_3-E_7+2H$ or 
$-2E_1-2E_2-2E_3-E_7+2H$. 

Assume that $D\sim E_1$. In this case, $\sN_{D/X}\simeq\sO_{\F_1}(-\sigma_1-f_1)$, 
$\tau_1=1$, $\tau(D)=2$, 
$X_2$ is the image of the morphism associated to the extremal face spanned by 
$\R_{\geq 0}[l_4]$ and $\R_{\geq 0}[l_8]$. If we see $\ND(X_2)$ as a subspace of 
$\ND(X)$, then $-K_{X_2}\sim -4E_1-2E_2-2E_3+4H$ and $D_2\sim -E_1+E_7+H$. 
Moreover, $\Nef(X_2)$ is spanned by the classes of $-E_1-E_2+H$, 
$-E_1-E_3+H$ and $-E_1-E_2-E_3-E_7+H$. Since $-K_{X_2}-2D_2$ is nef, we have 
$\tau_2=2$ and $m=2$. 
Since $\eta_1=9/4$ and $\eta_2=-4/3$, we have $\eta(D)/3>0$. 
If $D\sim E_2$ or $E_3$, then $\eta(D)>0$ in the same way. 
Assume that $D\sim E_7$. In this case, 
$\sN_{D/X}\simeq\sO_{\pr^1\times\pr^1}(-1, -1)$, $\tau_1=1$, 
$X_2$ is the blowup of $Q$ along a conic, $D_2$ is the exceptional divisor of the 
morphism $X_2\to Q$, $\sN_{D_2/X_2}\simeq\sO_{\pr^1\times\pr^1}(-1, 2)$, 
$\tau_2=2$ and $m=2$. Thus $\eta(D)>0$ by Lemma \ref{three_lem}. 
Assume that $D\sim E_1+E_7$. In this case, $\tau_1=1$, 
$X_2=\pr_{\pr^1}(\sO^{\oplus 2}\oplus\sO(1))$, 
$D_2\in|\xi_\pr^{\otimes 2}\otimes H_{\pr^1}^{\otimes (-1)}|$, 
$\tau_2=3/2$ and $m=2$. 
Since $\eta_1=19/6$ and $\eta_2=-5/24$, we have $\eta(D)/3>0$. 
If $D\sim E_2+E_7$ or $E_3+E_7$, then $\eta(D)>0$ in the same way. 
Assume that $D\sim -E_2-E_3-E_7+H$. In this case, $\tau_1=1$, 
$X_2$ is the image of the morphism associated to the extremal face spanned by 
$\R_{\geq 0}[l_4]$, $\R_{\geq 0}[l_7]$ and $\R_{\geq 0}[l_8]$. 
If we see $\ND(X_2)$ as a subspace of 
$\ND(X)$, then $-K_{X_2}\sim -4E_1-2E_2-2E_3+4H$ and $D_2\sim -2E_1-E_2-E_3+2H$. 
Thus $D_2\sim(-1/2)K_{X_2}$. Hence $\tau_2=2$ and $m=2$. 
Since $\eta_1=23/6$ and $\eta_2=-1/2$, we have $\eta(D)/3>0$. 
If $D\sim -E_1-E_3-E_7+H$ or $-E_1-E_2-E_7+H$, then $\eta(D)>0$ in the same way. 
Assume that $D\sim -E_1-E_2-E_3-E_7+H$. In this case, $\tau_1=1$, 
$X_2$ is the blowup of $Q$ along general three points, 
$D_2\sim -E_1-E_2-E_3+H$, where $E_1,\dots,E_3$ are the exceptional divisors of the 
morphism $X_2\to Q$ and $H$ corresponds to the pullback of $\sO_Q(1)$, 
$\tau_2=2$, $X_3=Q$, $D_3\in|\sO_Q(1)|$, $\tau_3=3$ and $m=3$. 
Since $\eta_1=5/2$, $\eta_2=-19/12$ and $\eta_3=-5/6$, we have 
$\eta(D)/3=1/12>0$. 
Assume that $D\sim -E_2-E_3+H$. In this case, $D$ is nef, $\tau_1=1$, 
$X_2=\pr_{\pr^1}(\sO^{\oplus 2}\oplus\sO(1))$, 
$D_2\in|\xi_\pr^{\otimes 2}|$, $\tau_2=3/2$ and $m=2$. 
Thus $\eta(D)>0$ by Lemma \ref{conv_lem}. 
If $D\sim -E_1-E_3+H$ or $-E_1-E_2+H$, then $\eta(D)>0$ in the same way. 
Assume that $D\sim -E_1-E_2-E_3+H$. In this case, $\tau_1=1$, 
$X_2$ is the blowup of $Q$ along a conic, 
$D_2$ corresponds to the pullback of $\sO_Q(1)$, 
$\tau_2=2$ and $m=2$. 
Since $\eta_1=41/12$ and $\eta_2=-5/6$, we have $\eta(D)/3>0$. 
Assume that $D\sim -E_1-2E_2-2E_3-2E_7+2H$. In this case, $\tau_1=1/2$, 
$X_2$ is the blowup of $Q$ along general three points, 
$D_2\sim -E_1-2E_2-2E_3+2H$, where $E_1,\dots,E_3$ are the exceptional divisors of 
the morphism $X_2\to Q$ and $H$ corresponds to the pullback of $\sO_Q(1)$, 
$\tau_2=1$, $X_3=\pr^3$, $D_3\in|\sO(3)|$, $\tau_3=4/3$ and $m=3$. 
Since $\eta_1=23/4$ and $\eta_3=-1/36$, we have $\eta(D)/3>23/4-1/36>0$. 
If $D\sim -2E_1-E_2-2E_3-2E_7+2H$ or $-2E_1-2E_2-E_3-2E_7+2H$, 
then $\eta(D)>0$ in the same way. 
Assume that $D\sim -E_1-2E_2-2E_3-E_7+2H$. In this case, $D$ is nef, $\tau_1=1$, 
$X_2=\pr^3$, $D_2\in|\sO(3)|$, $\tau_2=4/3$ and $m=2$. 
Thus $\eta(D)>0$ by Lemma \ref{conv_lem}. 
If $D\sim -2E_1-E_2-2E_3-E_7+2H$ or $-2E_1-2E_2-E_3-E_7+2H$, 
then $\eta(D)>0$ in the same way. 
Assume that $D\sim -2E_1-2E_2-2E_3-E_7+2H$. In this case, $\tau_1=1$, 
$X_2=Q$, 
$D_2\in|\sO_Q(2)|$, $\tau_2=3/2$ and $m=2$. 
Since $\eta_1=5$ and $\eta_2=-1/12$, we have $\eta(D)/3>0$. 
Hence $(X, -K_X)$ is divisorially stable. 

\noindent\textbf{The case No.\ 2--3.}\,\,
Assume that $X$ belongs to either No.\ 2 or No.\ 3 in \cite[Table 5]{MoMu}. 
Then $X$ is toric. Then $(X, -K_X)$ is not divisorially stable by Corollary \ref{toric_cor}. 
Moreover, $(X, -K_X)$ is divisorially semistable if and only if $X$ belongs to  
No.\ 3 in \cite[Table 5]{MoMu} by Theorem \ref{introtoric_thm} and \cite{mab}. 

Therefore we have proved Theorem \ref{three_thm} for the case $\rho(X)=5$.

\subsection{The case $\rho(X)\geq 6$}\label{rhobig_section}

We consider the case $\rho(X)\geq 6$. In this case, $X$ is isomorphic to the product 
of $\pr^1$ and a del Pezzo surface. By \cite{tian}, 
$X$ admits K\"ahler-Einstein metrics. 
Thus $(X, -K_X)$ is divisorially semistable by Remark \ref{divst_rmk}. 
Moreover, by \cite[Theorem 1.5]{fjt} and Corollary \ref{slope_cor} \eqref{slope_cor22}, 
$(X, -K_X)$ is not divisorially stable. 

As a consequence, we have completed the proof of Theorem \ref{three_thm}.

\end{document}